\documentclass[preprint,3p]{elsarticle}

\usepackage{amssymb}

\usepackage[english]{babel}
\usepackage[utf8]{inputenc}
\usepackage[T1]{fontenc}
\usepackage{lmodern}
\usepackage{latexsym}
\usepackage{array}
\usepackage{xcolor}
\usepackage{pdfcolmk} %behebt Probleme mit Farben bei pdflatex
\usepackage{amsmath}
\usepackage{amssymb}
\usepackage{mathtools}
\usepackage{dsfont}
\usepackage{upgreek}
\usepackage{amsthm}    % fuer Saetze, Definitionen, Beweise, etc.
\usepackage{amsfonts}  % spezielle AMS-Mathematik-Fonts
\usepackage{relsize}   % fuer \smaller 
\usepackage{graphicx}
\usepackage{enumitem}
\usepackage{esvect}
\usepackage{url}
\usepackage{float}                        
\usepackage{siunitx} 
\usepackage{colortbl} 
\usepackage{color,soul}   
\usepackage{todonotes}
\usepackage{subcaption}
\DeclareUnicodeCharacter{0306}{\v}

\usepackage{tikz, pgfplots}
\tikzstyle{bag} = [align=center]
\usetikzlibrary{shapes.misc}
\tikzset{cross/.style={cross out, draw=black, minimum size=2*(#1-\pgflinewidth), inner sep=0pt, outer sep=0pt},
	cross/.default={1pt}}
\usetikzlibrary{arrows}
\pgfplotsset{compat = newest}

\newcommand{\logLogSlopeTriangleUpsideDown}[6]
{
	\pgfplotsextra
	{
		\pgfkeysgetvalue{/pgfplots/xmin}{\xmin}
		\pgfkeysgetvalue{/pgfplots/xmax}{\xmax}
		\pgfkeysgetvalue{/pgfplots/ymin}{\ymin}
		\pgfkeysgetvalue{/pgfplots/ymax}{\ymax}
		
		\pgfmathsetmacro{\xArel}{#1}
		\pgfmathsetmacro{\yArel}{#3}
		\pgfmathsetmacro{\xBrel}{#1-#2}
		\pgfmathsetmacro{\yBrel}{\yArel}
		\pgfmathsetmacro{\xCrel}{\xArel}
		\pgfmathsetmacro{\lnxB}{\xmin*(1-(#1-#2))+\xmax*(#1-#2)} % in [xmin,xmax].
		\pgfmathsetmacro{\lnxA}{\xmin*(1-#1)+\xmax*#1} % in [xmin,xmax].
		\pgfmathsetmacro{\lnyA}{\ymin*(1-#3)+\ymax*#3} % in [ymin,ymax].
		\pgfmathsetmacro{\lnyC}{\lnyA-#4*(\lnxA-\lnxB)}
		\pgfmathsetmacro{\yCrel}{\lnyC-\ymin)/(\ymax-\ymin)} % THE IMPROVED EXPRESSION WITHOUT 'DIMENSION TOO LARGE' ERROR.
		
		\coordinate (A) at (rel axis cs:\xArel,\yArel);
		\coordinate (B) at (rel axis cs:\xBrel,\yBrel);
		\coordinate (C) at (rel axis cs:\xCrel,\yCrel);
		
		\draw[#5]   (A)-- node[pos=0.5,anchor=south] {1}
		(B)-- 
		(C)-- node[pos=0.5,anchor=west] {#6}
		cycle;
	}
}

\newcommand{\logLogSlopeTriangle}[6]
{
	\pgfplotsextra
	{
		\pgfkeysgetvalue{/pgfplots/xmin}{\xmin}
		\pgfkeysgetvalue{/pgfplots/xmax}{\xmax}
		\pgfkeysgetvalue{/pgfplots/ymin}{\ymin}
		\pgfkeysgetvalue{/pgfplots/ymax}{\ymax}
		
		\pgfmathsetmacro{\xArel}{#1}
		\pgfmathsetmacro{\yArel}{#3}
		\pgfmathsetmacro{\xBrel}{#1+#2}
		\pgfmathsetmacro{\yBrel}{\yArel}
		\pgfmathsetmacro{\xCrel}{\xArel}
		\pgfmathsetmacro{\lnxB}{\xmin*(1-(#1-#2))+\xmax*(#1-#2)} % in [xmin,xmax].
		\pgfmathsetmacro{\lnxA}{\xmin*(1-#1)+\xmax*#1} % in [xmin,xmax].
		\pgfmathsetmacro{\lnyA}{\ymin*(1-#3)+\ymax*#3} % in [ymin,ymax].
		\pgfmathsetmacro{\lnyC}{\lnyA+#4*(\lnxA-\lnxB)}
		\pgfmathsetmacro{\yCrel}{\lnyC-\ymin)/(\ymax-\ymin)} % THE IMPROVED EXPRESSION WITHOUT 'DIMENSION TOO LARGE' ERROR.
		
		\coordinate (A) at (rel axis cs:\xArel,\yArel);
		\coordinate (B) at (rel axis cs:\xBrel,\yBrel);
		\coordinate (C) at (rel axis cs:\xCrel,\yCrel);
		
		\draw[#5]   (A)-- node[pos=0.5,anchor=north] {1}
		(B)-- 
		(C)-- node[pos=0.5,anchor=east] {#6}
		cycle;
	}
}

\usepackage[
colorlinks=true,
urlcolor=blue,
linkcolor=red,
citecolor=teal
]{hyperref}

\newcommand{\N}{{\mathbb N}}
\newcommand{\Q}{{\mathbb Q}}
\newcommand{\R}{{\mathbb R}}

\newcommand{\uproman}[1]{\uppercase\expandafter{\romannumeral#1}}

\newcommand{\ba}[1]{\begin{array}{#1}}     %Array
\newcommand{\ea}[0]{\end{array}}     
\newcommand{\bma}[1]{\left[\begin{array}{#1}}     %Matrix
\newcommand{\ema}[0]{\end{array}\right]}     
\newcommand{\beq}[0]{\begin{equation}}     %eine Zeile nummeriert
\newcommand{\eeq}[0]{\end{equation}}
\newcommand{\beqn}[0]{\begin{eqnarray}}    %nummeriertes   Gl.syst.
\newcommand{\eeqn}[0]{\end{eqnarray}}
\newcommand{\beqk}[0]{\begin{displaymath}} %eine Zeile unnummeriert
\newcommand{\eeqk}[0]{\end{displaymath}}
\newcommand{\beqnk}[0]{\begin{eqnarray*}}  %unnummeriertes Gl.syst.
\newcommand{\eeqnk}[0]{\end{eqnarray*}}
\newcommand{\bseqn}[0]{\begin{subequations}}   % amsmath - subequations
\newcommand{\eseqn}[0]{\end{subequations}}     
\newcommand{\bsplit}[0]{\begin{split}}   % amsmath - split; endsplit geht nicht
\def\.{\cdot}

\newcommand{\dd}[2]{\frac{\partial #1}{\partial #2}}
\newcommand{\ddd}[2]{\frac{\partial^2 #1}{\partial #2 ^2}}

\newcommand{\DD}[2]{\frac{d #1}{d #2}}
\newcommand{\DDD}[2]{\frac{d^2 #1}{d #2 ^2}}

\newcommand{\bs}{\boldsymbol}

\newcommand{\abs}[1]{\left\lvert#1\right\rvert}
\newcommand{\norm}[1]{\left\lVert#1\right\rVert}

\newcommand{\vphi}{\varphi}

\newtheorem{lemma}{Lemma}
\theoremstyle{remark}
\newtheorem{remark}{Remark}
\theoremstyle{definition}

\theoremstyle{definition}
\newtheorem{example}{Example}[section]

\graphicspath{{Bilder/}}

\numberwithin{equation}{section}

\journal{Elsevier}

\begin{document}

\begin{frontmatter}

%% Title, authors and addresses

%% use the tnoteref command within \title for footnotes;
%% use the tnotetext command for theassociated footnote;
%% use the fnref command within \author or \address for footnotes;
%% use the fntext command for theassociated footnote;
%% use the corref command within \author for corresponding author footnotes;
%% use the cortext command for theassociated footnote;
%% use the ead command for the email address,
%% and the form \ead[url] for the home page:
%% \title{Title\tnoteref{label1}}
%% \tnotetext[label1]{}
%% \author{Name\corref{cor1}\fnref{label2}}
%% \ead{email address}
%% \ead[url]{home page}
%% \fntext[label2]{}
%% \cortext[cor1]{}
%% \affiliation{organization={},
%%             addressline={},
%%             city={},
%%             postcode={},
%%             state={},
%%             country={}}
%% \fntext[label3]{}

\title{Isogeometric analysis of the Laplace eigenvalue problem on circular sectors: Regularity properties, graded meshes \& variational crimes}

% use optional labels to link authors explicitly to addresses:
\author[unibw]{Thomas Apel}
\ead{thomas.apel@unibw.de}
 
\author[unibw]{Philipp Zilk\corref{cor1}}
\ead{philipp.zilk@unibw.de}
 
\affiliation[unibw]{organization={Institute for Mathematics and Computer-Based Simulation, Universität der Bundeswehr München},
             addressline={\\Werner-Heisenberg-Weg 39},
             city={Neubiberg},
             postcode={85577},
%             state={},
             country={Germany}}

\cortext[cor1]{Corresponding author}

% \affiliation[label2]{organization={},
%             addressline={},
%             city={},
%             postcode={},
%             state={},
%             country={}}
%
%\affiliation{organization={},%Department and Organization
%            addressline={}, 
%            city={},
%            postcode={}, 
%            state={},
%            country={}}

\begin{abstract}
The Laplace eigenvalue problem on circular sectors has eigenfunctions with corner singularities. Standard methods may produce suboptimal approximation results. To address this issue, a novel numerical algorithm that enhances standard isogeometric analysis is proposed in this paper by using a single-patch graded mesh refinement scheme. Numerical tests demonstrate optimal convergence rates for both the eigenvalues and eigenfunctions. Furthermore, the results show that smooth splines possess a superior approximation constant compared to their $C^0$-continuous counterparts for the lower part of the Laplace spectrum. This is an extension of previous findings about excellent spectral approximation properties of smooth splines on rectangular domains to circular sectors. In addition, graded meshes prove to be particularly advantageous for an accurate approximation of a limited number of eigenvalues. The novel algorithm applied here has a drawback in the singularity of the isogeometric parameterization. It results in some basis functions not belonging to the solution space of the corresponding weak problem, which is considered a variational crime. However, the approach proves to be robust. Finally, a hierarchical mesh structure is presented to avoid anisotropic elements, omit redundant degrees of freedom and keep the number of basis functions contributing to the variational crime constant, independent of the mesh size. Numerical results validate the effectiveness of hierarchical mesh grading for the simulation of eigenfunctions with and without corner singularities.
\end{abstract}

%%Graphical abstract
%\begin{graphicalabstract}
%%\includegraphics{grabs}
%\end{graphicalabstract}

%%Research highlights
%\begin{highlights}
%\item Research highlight 1
%\item Research highlight 2
%\end{highlights}

\begin{keyword}
Isogeometric analysis \sep Eigenvalue problems \sep Spectral approximation properties \sep Circular sectors \sep Corner singularities \sep Local mesh refinement \sep Graded meshes \sep Hierarchical refinement \sep Singular parameterizations \sep Variational crimes
%% keywords here, in the form: keyword \sep keyword

%% PACS codes here, in the form: \PACS code \sep code

%% MSC codes here, in the form: 
\MSC[2010] 65N25 \sep 65D07 \sep 65N50 \sep 35A21 \sep 35C05
%% or \MSC[2008] code \sep code (2000 is the default)

\end{keyword}

\end{frontmatter}

%Führe aus mit pdflatex, biblatex, pdflatex, siehe https://tex.stackexchange.com/questions/632150/error-cannot-find-elsarticle-template-bcf-when-compiling-the-elsevier-templa

%TODO Bearbeite .bib - Datei mit nur nötigen Informationen	

%% \linenumbers

%% main text
\section{Introduction}
\label{sec: introduction}
The Laplace eigenvalue problem
\begin{align*}
- \Delta u = \lambda u \quad \text{ in } \Omega
\end{align*}
on arbitrary domains $\Omega \subset \R^d$ is of great significance in numerous areas of applied mathematics. Not only is the Laplace spectrum itself crucial for many applications, e.g., due to the geometric information it contains \cite{Kac1966,ReuterWolterPeinecke2006}, but it also appears during the analysis of other partial differential equations which describe everyday physical processes. For instance, the wave equation in two spatial dimensions describes the vibration of a membrane and its analytical solution can be derived by applying separation of variables and then solving the resulting Laplace eigenvalue problem \cite{BabuskaOsborn1991}. For some one-, two- or three-dimensional model domains such as lines, rectangles, circular sectors or balls, the exact Laplace eigenvalues and eigenfunctions have been determined in the literature \cite{Clebsch1862,Rayleigh1945,Strauss2008}. Yet, for more general domains, an analytical solution is usually not known and thus numerical methods are needed.

Isogeometric analysis (IGA) has been proven to provide powerful tools for an excellent approximation of the Laplace eigenvalues, especially in one dimension and for two-dimensional domains of rectangular nature \cite{CottrellRealiBazilevsHughes2006,HughesRealiSangalli2008,CottrellHughesBazilevs2009,CottrellHughesReali2007,HughesEvansReali2014,SandeManniSpeleers2019,SandeManniSpeleers2020,HiemstraHughesRealiSchillinger2021,ManniSandeSpeleers2022}.
However, the Laplace eigenfunctions of the underlying domains are always smooth and the question arises whether the spectrum of more complex domains, particularly in the presence of singularities and non-smooth eigenfunctions, can also be approximated well with IGA. For this purpose, circular sectors serve as a perfect model domain as they contain a typical corner singularity in the center and exact solutions of the Laplace eigenvalue problem are known such that spectral approximation properties can be verified numerically. In addition, one of the initial motivations of IGA can be explored during the discretization process: the exact representation of the computational domain geometry through the use of NURBS basis functions \cite{HughesCottrellBazilevs2005}. In this way, circular sectors can be parameterized exactly, which is not possible with classical finite elements. 

%For all these reasons, we consider the Laplace eigenvalue problem on circular sectors and its isogeometric approximation in extensive detail in this contribution.

Certainly, the corner singularity of circular sectors needs to be addressed during the numerical solution process to obtain a proper approximation of the Laplace spectrum. Many different approaches have been proposed in the literature and can be grouped into two categories. Either the mesh is locally refined towards the singularity or the approximation space is enriched by singular functions \cite{LiLu2000}. Sometimes, both ideas are combined. Within the framework of IGA, a method of the first category has been contributed where the parametric mapping is modified to obtain a grading of the mesh \cite{JeongOhKangKim2013}. In a related work, which belongs to the first and second category, this idea has been combined with an enrichment of the basis by singular functions \cite{OhKimJeong2014}. Another method of the first category is graded CutIGA where computational background meshes are cut arbitrarily to represent the singularity while retaining optimal approximation results \cite{JonssonLarsonLarsson2019}. Further possibilities include adaptive approaches or trimmed methods, both of which have gained significant attention in the field of isogeometric analysis recently, see for instance the overviews works \cite{BuffaGantnerGiannelliPraetoriusVazquez2022,MarussigHughes2018}. 

The existing results about powerful spectral approximation properties have been achieved using standard single-patch IGA or very closely related methods. Therefore, we aim for an approach that is as close as possible to the standard method to maintain comparability between our results and the existing ones. In this context, we decide to use a method of the first category and we prefer to work with an exact representation of circular sectors instead of a trimmed approach. Since we consider a model problem where the location of the singularity is known, the local refinement can be set up a priori and there is no need for adaptive strategies which require more computational effort.

From finite element methods, it is well known that mesh grading is a simple and powerful tool for local a priori refinement towards corner singularities \cite{ApelSaendigWhiteman1996,ApelNicaise1998,Babuska1970,OganesjanRukhovets1968}. This idea has already been transferred to IGA by using a multi-patch approach \cite{LangerMantzaflarisMooreToulopoulos2015}. However, the authors require the isogeometric mapping to be smooth, at least in the points towards which the mesh is locally refined. We contribute a new approach to overcome this restriction by using the single-patch polar-like parameterization of circular sectors which is singular at the conical point of the circular sector. Due to the singular parameterization, a few basis functions that span the standard isogeometric approximation space are singular \cite{TakacsJuettler2011,TakacsJuettler2012,Takacs2015}. Approximation properties in such cases have so far only been shown for smooth functions on singularly parameterized triangles in \cite{Takacs2015preprint}. Our contribution presents a numerical extension of this work to singularly parameterized circular sectors for a specific class of functions. Precisely, we consider numerical approximation properties for the set of eigenfunctions which possibly contains both smooth and singular functions depending on the angle of the circular sector.

The outline of this paper is as follows. In Section \ref{section: laplace evp on circular sectors} we introduce the model problem, derive an analytical solution and investigate crucial regularity properties of the eigenfunctions. Section \ref{section: preliminaries and notation} contains a short overview about the basics of splines and NURBS and the corresponding notation which is then used in Section \ref{section: Isogeometric mesh grading} to explain single-patch isogeometric mesh grading for circular sectors. In Section \ref{section: Numerical results} we provide numerical results showing that the proposed method guarantees optimal convergence rates for the Laplace eigenpairs and hence is a powerful approach to solve the considered model problem. Moreover, we demonstrate that smooth splines are particularly useful on graded meshes for computing multiple eigenvalues and consider a combination with a hierarchical refinement scheme. In Section \ref{section: Conclusion}, we finally conclude our main findings and list some ideas for further research.

In the sequel, the symbol $C$ is used for a generic positive constant, which may be different at each occurrence.

\section{The Laplace eigenvalue problem on circular sectors}
\label{section: laplace evp on circular sectors}
The main subject of this paper is the Laplace eigenvalue problem on circular sectors. Therefore, we provide the fundamental equations and boundary conditions, explain how an analytical solution can be derived and point out key regularity properties of the resulting eigenfunctions in this section. Furthermore, we introduce the weak form and the discrete form of the model problem which are essential for the numerical solution process with IGA.

\subsection{The model problem}
We consider the Laplace eigenvalue problem
\begin{alignat}{2}
-\Delta u &= \lambda u \qquad &&\text{in\ }\Omega\, , \nonumber \\
u &= 0 \qquad &&\text{on\ }\Gamma_D\, ,  \label{eq: model problem}\\
\frac{\partial u}{\partial n}&=0  \qquad && \text{on\ }\Gamma_N\,,  \nonumber
\end{alignat}
where the model domain $\Omega = \left\{(r\cos\varphi, r\sin\varphi) \in \R^2: 0<r<1, 0<\varphi<\omega\right\}$ is a circular sector of angle $\omega \in(0, 2\pi]$. The Dirichlet boundary  $\Gamma_D = \{(\cos \varphi, \sin \varphi): 0 \leq \varphi \leq \omega\}$ consists of the circular edge and the Neumann boundary is given by the angle legs, $\Gamma_N = \{(r \cos \varphi, r \sin \varphi): 0 \leq r < 1, \varphi \in \{0,\omega\}\}$. We illustrate the model domain and its boundary for $\omega = \frac{3}{2}\pi$ and $\omega = 2\pi$ in Figure \ref{fig: sketch of model domain}. 

From a physical point of view, the model problem \eqref{eq: model problem} describes the vibrations of a membrane stretched over a circular sector. The membrane is assumed to be fixed at the circular Dirichlet boundary $\Gamma_D$. For $\omega = 2\pi$, we can think of a circular drum with a straight crack which is represented by the two angle legs. In this context, it is natural to choose Neumann boundary conditions on $\Gamma_N$. For the sake of simplicity, we stick to this choice of boundary conditions and do not discuss further combinations although our choice is not essential for the main findings of this paper.

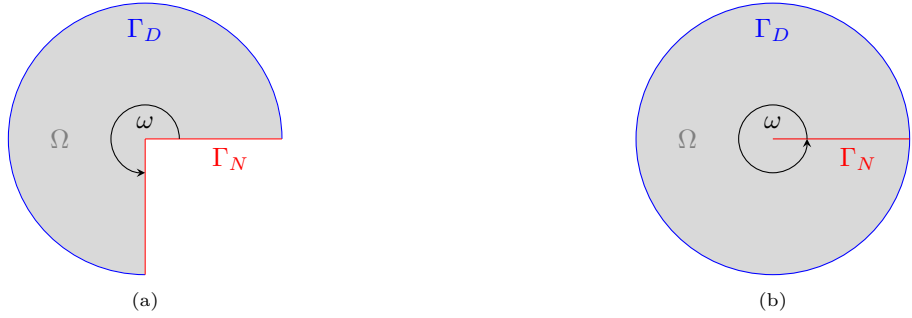
\begin{figure}
	\begin{center}
		\begin{subfigure}{0.495\textwidth}
			\begin{center}
				\begin{tikzpicture}[scale=0.9]	
				\filldraw[fill = gray, fill opacity = 0.3, draw = white] (2,0) arc (0:270:2) -- (0,0) -- cycle;
				\draw[blue] (2,0) arc (0:270:2);	
				\draw[red] (0,0) -- (2, 0);	
				\draw[red] (0,0) -- (0, -2);
				
				\draw[-stealth] (0.1:0.5) arc (0:270:0.5);
				\node at (90:0.25) {$\omega$};	
				\node[gray] at (180:1.25) {$\Omega$};
				\node[blue] at (90:1.6) {$\Gamma_D $};
				\node[red] at (1.25,-0.3) {$\Gamma_N $};
				\end{tikzpicture}
				\caption{}	
				\label{fig: sketch of model domain a}
			\end{center}
		\end{subfigure}
		\begin{subfigure}{0.495\textwidth}
			\begin{center}
				\begin{tikzpicture}[scale=0.9]	
				\draw[blue, fill = gray, fill opacity = 0.3] (0,0) circle (2);	
				\draw[red] (0,0) -- (2, 0);	
				\draw[-stealth] (0.1:0.5) arc (0:360:0.5);
				\node at (90:0.25) {$\omega$};	
				\node[gray] at (180:1.25) {$\Omega$};
				\node[blue] at (90:1.6) {$\Gamma_D $};
				\node[red] at (1.25,-0.3) {$\Gamma_N $};
				\end{tikzpicture}
				\caption{}	
				\label{fig: sketch of model domain b}
			\end{center}
		\end{subfigure}
	\end{center}
	\vspace{-5mm}
	\caption{Circular sectors $\Omega$ with angle $\omega$ and the corresponding boundaries. (a) $\omega = \frac{3}{2} \pi$. (b) $\omega=2\pi$.}
	\label{fig: sketch of model domain}
\end{figure}

\subsection{Analytical solution}
\label{subsec: analytical solution}
The analytical solution of our model problem \eqref{eq: model problem} is well known and there is a vast amount of literature about it. According to Rayleigh, the theory of vibrations of a circular membrane has first been introduced by Clebsch \cite{Clebsch1862,Rayleigh1945}. Later, Rayleigh himself deduced the eigenmodes of circular sectors. We can not provide an exhaustive list here, but mention a few other works in which the problem has also been discussed \cite{KuttlerSigillito1984,GrebenkovNguyen2013}. In this section, we recapitulate the main ideas to derive the exact Laplace eigenfunctions and eigenvalues of circular sectors. We adopt the notation used in the book \cite[Chapter 10.2]{Strauss2008}, where the vibrations of a circular membrane are considered in detail. First, we write down the equations \eqref{eq: model problem} in polar coordinates,
\begin{alignat*}{2}
-\frac{1}{r}\dd{u}{r} - \ddd{u}{r} - \frac{1}{r^2} \ddd{u}{\varphi} &= \lambda u \qquad &&\text{in } \Omega\, , \\
u &= 0 \qquad && \text{on } \Gamma_D \, , \\
\frac{\partial u}{\partial \varphi}&=0  \qquad && \text{on\ }\Gamma_N\, .
\end{alignat*}
Introducing the separation ansatz $u(r, \varphi) = R(r)\Phi(\vphi)$, we obtain
\begin{align*}
\frac{r^2 R''(r) +r R'(r)}{R(r)} + r^2 \lambda = - \frac{\Phi''(\varphi)}{\Phi(\varphi)} 
\end{align*}
and thus there is a $\sigma \in \R$ such that, by also considering the boundary conditions,
\begin{align}
r^2 R''(r) +r R'(r) + (\lambda r^2 - \sigma) R(r) = 0\, , \qquad R(1) = 0 
\label{EquationForR}
\end{align}
and
\begin{align}
\Phi''(\varphi) + \sigma \, \Phi(\varphi) = 0 \, , \qquad \Phi'(0) = \Phi'(\omega) = 0 \, .
\label{EquationForPhi}
\end{align}
The eigenvalue problem \eqref{EquationForPhi} has the solutions
\begin{align*}
\sigma_k = \nu_k^2 \, \quad \text{ and } \quad
\Phi_k(\varphi) = a_k \cos( \nu_k \, \varphi)  \quad \text{ with } \nu_k = k \, \frac{\pi}{\omega} \text{ and } a_k \in \R \text{ for } k \in \N_0 \, .
\end{align*}  
Solving the boundary value problem \eqref{EquationForR} for $R$ is a bit more technical. We know that the model problem \eqref{eq: model problem} only has positive eigenvalues, see for instance \cite[Chapter 10.1]{Strauss2008}. We substitute $\rho(r):= r \sqrt{\lambda}$, define $\tilde{R} (\rho(r)) : = R\left( \frac{\rho(r)}{\sqrt{\lambda}} \right) = R(r)$ and compute the relations
\begin{align*}
\DD{}{r} R(r) = \DD{}{r} \tilde{R} (\rho(r)) = \DD{}{\rho} \tilde{R} (\rho(r)) \. \sqrt{\lambda} \, ,  \\
\DDD{}{r}R(r) = \DDD{}{r}\tilde{R} (\rho(r)) = \DDD{}{\rho} \tilde{R} (\rho(r)) \. \lambda \, . 
\end{align*}
Inserting both expressions above into equation \eqref{EquationForR} leads to the Bessel differential equation
\begin{align}
\rho^2 \DDD{\tilde{R}}{\rho} + \rho \DD{\tilde{R}}{\rho} + (\rho^2-\nu_k^2) \tilde{R} = 0 \, .
\label{eq: Bessel's equation}
\end{align}
In general, the Bessel equation has solutions of the form $\tilde{R}(\rho) = \alpha J_{\nu_k}(\rho) + \beta Y_{\nu_k}(\rho)$, where $J_{\nu_k}$ and $Y_{\nu_k}$ are Bessel functions to the order $\nu_k$ of the first and second kind, respectively, and $\alpha,\beta \in \R$. However, the Bessel functions of the second kind have a pole at $\rho = 0$ and therefore result in unphysical solutions, i.e., it is $\beta=0$. Without loss of generality, we set $\alpha = 1$ and obtain the relevant solution of equation \eqref{eq: Bessel's equation} by
\begin{align*}
\tilde{R} (\rho(r)) = J_{\nu_k}(\rho(r)) = \sum_{j=0}^\infty \frac{(-1)^j}{j! \, \Gamma(\nu_k+j+1)} \left(\frac{\rho(r)}{2}\right)^{2j+ \nu_k} \,,
\end{align*} 
where $\Gamma$ denotes the Gamma function. Hence, the function
\begin{align*}
R(r) =  \tilde{R} (\rho(r)) = J_{\nu_k}(\rho(r)) = J_{\nu_k}(\mu \, r)  \quad \text{ with } \mu = \sqrt{\lambda} \text{ and } \nu_k =  k \frac{\pi}{\omega} \text{ for } k \in \N_0  
\end{align*}
solves the equation initially considered in problem \eqref{EquationForR}. Since we also require homogeneous Dirichlet boundary conditions on the circular boundary, that is, $R(1) = 0$, we obtain the condition $J_{\nu_k}(\mu) = 0$. Thus, $\mu$ has to be a root of the Bessel function $J_{\nu_k}$. We denote the $m$-th root of $J_{\nu_k}$ by $\mu_{\nu_k,m}$ for $m\in \N$ which enables us to express the exact eigenfunctions in polar coordinates as
\begin{align}
u_{\nu_k,m}(r, \varphi) = a_{\nu_k,m} \, J_{\nu_k}(\mu_{\nu_k,m} \, r) \cos(\nu_k \, \varphi)
\label{eq: formula eigenfunctions}
\end{align}
with $a_{\nu_k,m} \in \R $ and $\nu_k = k \frac{\pi}{\omega}$ for $k \in \N_0$. Throughout the paper, we set $a_{\nu_k,m}=1$ for all $k \in \N_0$ and $m \in \N$ for the sake of simplicity and we call the functions \eqref{eq: formula eigenfunctions} the Laplace eigenfunctions or eigenmodes of circular sectors. The corresponding exact eigenvalues and eigenfrequencies are then given by
\begin{align}
\lambda_{\nu_k,m} = \mu_{\nu_k,m}^2\, \quad \text{ and } \quad \omega_{\nu_k,m} = \sqrt{\lambda_{\nu_k,m}} = \mu_{\nu_k,m}
\label{eq: formula eigenvalues}
\end{align}
and will be called the Laplace eigenvalues and eigenfrequencies of circular sectors, respectively. Ordering the eigenvalues and eigenfrequencies in ascending order, we denote them by $\lambda_n$ and $\omega_n = \sqrt{\lambda_n} $ for $n = 1,2, \dots$, respectively. The zeros of Bessel functions and their asymptotic behavior have been investigated extensively in the literature. For example, from \cite[Chapter 9.5]{AbramowitzStegun1988}, we know that
\begin{align*}
\mu_{\nu_k,1} < \mu_{\nu_k+1,1} < \mu_{\nu_k,2} < \mu_{\nu_k+1,2} < \mu_{\nu_k,3} < \ldots  \quad \text{ for any } k \in \N_0 \,.
\end{align*}
In \cite{GrebenkovNguyen2013}, the authors provide a compact summary of two important asymptotic formulas. For fixed $\nu_k$ and large $m$, McMahon's expansion, $\mu_{\nu_k,m} \sim \pi (m+\nu_k/2-1/4)) + \mathcal{O}(m^{-1})$, holds \cite{Watson1995}. For fixed $m$ and large $\nu_k$, Olver's expansion, $\mu_{\nu_k,m} \sim \nu_k + \delta_m \nu_k^{1/3} + \mathcal{O}(\nu_k^{-1/3})$, holds with known coefficients $\delta_m$ \cite{Elbert2001,Olver1951,Olver1952}. Hence, it is
\begin{align*}
\mu_{\nu_k,m} \to \infty \text{ for } k \to \infty \quad \text{ and } \quad \mu_{\nu_k,m} \to \infty \text{ for } m \to \infty 
\end{align*}
and we obtain infinite sequences of eigenvalues $(\lambda_n)_{n\in\N}$ and eigenfrequencies $(\omega_n)_{n\in\N}$ with 
\begin{align}
0 < \lambda_1 \leq \lambda_2 \leq \ldots \to \infty  \quad \text{ and } \quad  0 < \omega_1 \leq \omega_2 \leq \ldots \to \infty \, .
\label{eq: sequence of eigenvalues}
\end{align}

\subsection{Regularity of the eigenfunctions}
\label{section: regularity of the eigenfunctions}
In this section, we delve into the regularity analysis of the Laplace eigenfunctions of circular sectors. To begin, we point out a useful representation of the considered eigenmodes.
\begin{remark}
	\label{remark: product form of eigenfunctions}
	The eigenfunctions \eqref{eq: formula eigenfunctions} can be written in separated variables as a product
	\begin{align}
	u_{\nu_k,m}(r, \varphi) = \, R(r) \, \Phi(\varphi)
	\label{eq: product representation of the eigenfunctions}
	\end{align}
	of a smooth function $\Phi(\varphi) = \cos(\nu_k \, \varphi)$ satisfying $\norm{\Phi^{(l)}}_{L^\infty(\Omega)} \leq \nu_k^l$ for all $l\in\N_0$, and a Bessel function $R(r) = J_{\nu_k}(\mu_{\nu_k,m} \, r)$, which contains all the information about the regularity of the eigenfunctions. Typically, the solutions of boundary value problems on circular sectors are given by a sum of such products. Thus, the product form \eqref{eq: product representation of the eigenfunctions} is a particular property of the considered eigenvalue problem \eqref{eq: model problem}.
\end{remark}

Motivated by Remark \ref{remark: product form of eigenfunctions}, we introduce some general properties of Bessel functions to investigate the regularity of the Laplace eigenfunctions \eqref{eq: formula eigenfunctions}. A useful asymptotic representation for small arguments of Bessel functions of the first kind to any order $\nu \notin \{-1,-2,-3, \dots\}$ is given by Abramowitz and Stegun \cite[Formula 9.1.7, p. 360]{AbramowitzStegun1988},
\begin{align}
J_{\nu}(z) \sim \frac{1}{\Gamma (\nu_k+1)} \left(\frac{z}{2}\right)^{\nu} \quad \text{ for } z \to 0 \, .
\label{eq: asymptotic representation of Bessel funtion to the first kind}
\end{align}
As $\nu_k = k \, \frac{\pi}{\omega}$ for $k \in \N_0$, the formula \eqref{eq: asymptotic representation of Bessel funtion to the first kind} yields that every eigenfunction $u_{\nu_k,m}$ for $k \in \N_0$, $m \in \N$ satisfies
\begin{align}
\label{eq: asymptotic representation of u}
u_{\nu_k,m}(r,\varphi) 
= J_{\nu_k}(\mu_{\nu_k,m} \, r) \cos(\nu_k \, \varphi)
\sim \frac{1}{\Gamma (\nu_k+1)} \left(\frac{\mu_{\nu_k,m} \, r}{2}\right)^{\nu_k}  \cos(\nu_k \, \varphi) = C r^{\nu_k} \quad \text{ for } r \to 0 \, .
\end{align}
Hence, the Laplace eigenfunctions of circular sectors behave like the functions $ r \mapsto r^{\nu_k}$ close to the conical point, which is the crucial observation to assess their regularity.

Furthermore, we have an integral representation of Bessel functions of the first kind to any order $\nu$ with $\text{Re}(\nu)>-\frac12$ \cite[Chapter 9.1]{AbramowitzStegun1988},
\begin{align*}
J_{\nu}(z) = \frac{\left(\frac{z}{2}\right)^\nu}{\pi^{1/2}\Gamma(\nu + \frac12)} \int_0^\pi \cos(z \cos \theta) \sin^{2\nu} \theta \, \mathrm{d} \theta \, .
\end{align*}
Thus, for $z \in \R$ with $z>0$, it is
\begin{align}
\abs{J_{\nu}(z)} & \leq \frac{\left(\frac{z}{2}\right)^\nu}{\pi^{1/2}\Gamma(\nu + \frac12)} \int_0^\pi \abs{\cos(z \cos \theta) \sin^{2\nu} \theta} \, \mathrm{d} \theta \nonumber \\
& \leq \frac{\left(\frac{z}{2}\right)^\nu}{\pi^{1/2}\Gamma(\nu + \frac12)} \int_0^\pi \abs{\sin^{2\nu} \theta} \, \mathrm{d} \theta \nonumber \\
& = C(\nu) \, z^{\nu}
\label{eq: estimate of Bessel functions}
\end{align}
with $C(\nu):= \displaystyle \frac{\left(\frac{1}{2}\right)^\nu}{\pi^{1/2}\Gamma(\nu + \frac12)} \int_0^\pi \abs{\sin^{2\nu} \theta} \, \mathrm{d} \theta $. Combined with the asymptotic behavior \eqref{eq: asymptotic representation of u}, we obtain
\begin{align}
\abs{u_{\nu_k,m}(r,\varphi)} 
= \abs{J_{\nu_k}(\mu_{\nu_k,m} \, r) \cos(\nu_k \, \varphi)} \leq C(\nu_k) \, (\mu_{\nu_k,m} \, r)^{\nu_k} \abs{\cos(\nu_k \, \varphi)}
\leq C \, r^{\nu_k}
\label{eq: estimate of eigenfunctions}
\end{align} 
in $\Omega$ for all $k \in \N_0$ and $m \in \N$.

% Alte Rechnung:
%Every eigenfunction $u = u_{\nu_k,m}$ for $k,m \in \N$ can be decomposed into an alternating series,
%\begin{align}
%u(r,\varphi) & = J_{\nu_k}(\mu_{\nu_k,m} \, r) \cos(\nu_k \, \varphi) \label{eq: Formula for u} \\ 
%& = \cos(\nu_k \, \varphi)  \sum_{j=0}^\infty \frac{(-1)^j}{j! \, \Gamma(\nu_k+j+1)} \left(\frac{\mu_{\nu_k,m} \, r}{2}\right)^{\nu_k+2j} \nonumber \\
%& =  \cos(\nu_k \, \varphi) \, r^{\nu_k} \sum_{j=0}^\infty (-1)^j \, C_j \, r^{2j} \, ,
%\label{eq: Alternating series}
%\end{align}
%where $\{C_j\}_{j \in \N}$ is a monotonously decreasing sequence of positive constants, i.e., $ 0 < C_{j+1} \leq C_j $ for all $j \in \N$, with $C_j \to 0$ for $j \to \infty$. Furthermore, for all $r \in (0,1]$ and $x,y  \in \R$ with $x\leq y$ it is $r^x \geq r^y$. Thus, the alternating series test, also known as Leibniz test, guarantees the convergence of the considered alternating series \eqref{eq: Alternating series}. Hence, for all $(r, \varphi) \in \Omega$, we have
%\todo{Gehe auf C ein, siehe Feedback Thomas. Nutze Formel 9.1.20 aus \cite{AbramowitzStegun1964}.}
%\begin{align*}
%\abs{u(r,\varphi)} & = \abs{\cos(\nu_k \, \varphi) \,  r^{\nu_k} \sum_{j=0}^\infty (-1)^j \, C_j \, r^{2j}} \nonumber\\
%& \leq r^{\nu_k} \abs{ \sum_{j=0}^\infty (-1)^j \, C_j \, r^{2j}} \nonumber\\
%& \leq C \, r^{\nu_k} \, .
%\end{align*}

Next, we consider the gradient $\nabla u = \left(\frac{\partial u}{\partial x_1}, \frac{\partial u}{\partial x_2}\right)^T $ of an eigenfunction $u = u_{\nu_k,m}$ for $k \in \N_0$, $m \in \N$ with $x_1$ and $x_2$ being the Cartesian coordinates in $\R^2$. The eigenmodes can be differentiated using the recursion relation for Bessel functions \cite[Chapter 10.5]{Strauss2008},
\begin{align}
J_\nu'(z) = \frac{\nu}{z} J_\nu(z)- J_{\nu+1}(z) \, .
\label{eq: differentiation of Bessel function}
\end{align}
By changing to polar coordinates and setting $\mu = \mu_{\nu_k,m}$, it follows
\begin{align}
\label{eq: estimate of gradient}
\abs{\nabla u} &\leq C \left(\abs{\dd{u}{r}} + \frac1r \abs{\dd{u}{\varphi}}\right) \nonumber \\
&=  C \left( \abs{ \dd{(J_{\nu_k}(\mu \, r))}{r} \cos(\nu_k \, \varphi)} +  \frac1r \abs{J_{\nu_k}(\mu \, r) \, \nu_k \, \sin(\nu_k \, \varphi)} \right) \nonumber \\
&=  C \left( \abs{\left(\frac{\nu_k}{r} \, J_{\nu_k}(\mu \, r) - \mu \, J_{\nu_k+1}(\mu \, r) \right) \cos(\nu_k \, \varphi)} + \frac1r \abs{J_{\nu_k}(\mu \, r)  \, \nu_k \, \sin(\nu_k \, \varphi)} \right) \\
& \leq C \left( \frac{2\nu_k}{r}  \abs{\, J_{\nu_k}(\mu \, r)} +\mu \abs{ J_{\nu_k+1}(\mu \, r)} \right) \nonumber \\
& \leq C \left( \frac{2\nu_k}{r} \,  r^{\nu_k} +\mu \, r^{\nu_k+1} \right) \nonumber \\
& \leq C \, r^{\nu_k-1} \, , \label{eq: final estimate of gradient}
\end{align}
where we use the boundedness of the Bessel functions \eqref{eq: estimate of Bessel functions}. We repeat the same argumentation for the higher derivatives
\begin{align*}
D^\alpha u = \frac{\partial ^{|\alpha|} u}{ \partial x_1 ^{\alpha_1} \partial x_2 ^{\alpha_2}}  \quad \text{ for multi-indices } \alpha = \begin{pmatrix}
\alpha_1 \\ \alpha_2\end{pmatrix}\in \N_0^2 
\end{align*}
and obtain
\begin{align}
\abs{D^\alpha u_{\nu_k,m}} \leq C r^{\nu_k - \abs{\alpha}} \quad \text{ in } \Omega
\label{eq: estimate for eigenfunctions and its derivatives with k}
\end{align}
for all $k \in \N_0, m \in \N$ and $\alpha \in \N_0^2$, where $|\alpha| = \alpha_1+\alpha_2$.

We are now able to determine the Sobolev regularity of the eigenmodes in the following lemma. Let $H^{s}(\Omega)$ for $s \in \N_0$ be the classical Sobolev spaces and $L^2(\Omega) = H^0(\Omega)$. For general $s\geq 0$ with $s \not \in \N_0$ the notation $H^{s}(\Omega)$ is used for the Sobolev-Slobodeckij spaces.

\begin{lemma}
	\label{lemma: regularity of the eigenfunctions}
	Let $k \in \N_0$ and $m \in \N$. Then, the Laplace eigenfunctions of circular sectors \eqref{eq: formula eigenfunctions} satisfy
	\begin{align*}
	u_{\nu_k,m} \in H^s(\Omega) \quad \text{ for all } s \geq 0 \text{ with } s<\nu_k + 1 \, .
	\end{align*}
\end{lemma}

\begin{proof}
	Estimate \eqref{eq: estimate for eigenfunctions and its derivatives with k} shows that the eigenfunction $u_{\nu_k,m}$ and its derivatives can be controlled by the function $r \mapsto r^{\nu_k}$ and its derivatives. It is well known that $r \mapsto r^{\nu_k}$ is in $H^s(\Omega)$ for all $s \geq 0$ with $\nu_k-s>-1$, that is, for $s<\nu_k+1$.		
\end{proof}

In general, the regularity result in Lemma \ref{lemma: regularity of the eigenfunctions} is not sharp. For instance, consider $k=0$. Then, it is $\nu_0=0$ and thus Lemma \ref{lemma: regularity of the eigenfunctions} only yields that $u_{\nu_0,m} \in H^s(\Omega)$ for all $s<1$. Indeed, the eigenfunctions $u_{\nu_0,m}$ are of a different nature in the sense that
\begin{align*}
u_{\nu_k,m}(0,0) = \begin{cases}
0 \quad & \text{ if } k>0  \, ,\\
1 \quad & \text{ if } k=0  \, .
\end{cases}
\end{align*}
Nevertheless, it is clear that the functions $u_{\nu_0,m}$ are at least in $H^1(\Omega)$, since they are solutions of the weak form of our model problem \eqref{eq: model problem}, which will be discussed more detailed in Section \ref{subsec: numerical solution process}. In fact, Lemma \ref{lemma: regularity of the eigenfunctions} is sharp for all eigenmodes $u_{\nu_k,m}$ with $\nu_k \not \in \N_0$, as the following lemma shows.

\begin{lemma}
	\label{lemma: smooth eigenfunctions}
	Let $k \in \N_0$ and $m \in \N$. The Laplace eigenfunctions of circular sectors \eqref{eq: formula eigenfunctions} satisfy
	\begin{align*}
	u_{\nu_k,m} \in H^s(\Omega) \quad \text{ for all } s \geq 0
	\end{align*}
	if and only if $\nu_k = k \frac{\pi}{\omega} \in \N_0$. For $\nu_k \not \in \N_0$, the regularity result shown in Lemma \ref{lemma: regularity of the eigenfunctions} is sharp.
\end{lemma}

\begin{proof}
	First, let $\nu_k = k \frac{\pi}{\omega} \not \in \N_0$. Exploiting the asymptotic representation \eqref{eq: asymptotic representation of u}, we observe that $u_{\nu_k,m}$ behaves like $r \mapsto r^{\nu_k}$ for $r \to 0$. Consequently, it is $u_{\nu_k,m} \not \in H^s(\Omega)$ for all $s\geq \nu_k+1$ and Lemma \ref{lemma: regularity of the eigenfunctions} is sharp.
	
	Now, let $n:= \nu_k = k \frac{\pi}{\omega} \in \N_0$. Estimate \eqref{eq: estimate of eigenfunctions} yields that
	\begin{align*}
	\abs{u_{n,m}(r,\varphi)} = \abs{J_{n}(\mu_{n,m} \, r) \cos(n \, \varphi) } \leq C r^n \abs{\cos(n \, \varphi)} \, .
	\end{align*}
	With the relations between Cartesian and polar coordinates, $(x,y) = (r \cos \varphi, r \sin \varphi)$ and $r^2=x^2 + y^2$, it holds
	\begin{align*}
	\abs{u_{n,m}(r,\varphi)} &  \leq C r^n \abs{\cos(n \, \varphi)} \\
	&= C r^n \abs{\sum_{j=0}^{ \lfloor \frac{n}{2} \rfloor} (-1)^j \binom{n}{2j} (\sin \varphi)^{2j}  (\cos \varphi)^{n-2j}} \\
	&= C \abs{\sum_{j=0}^{ \lfloor \frac{n}{2} \rfloor} (-1)^j \binom{n}{2j} \, r^{2j} (\sin \varphi)^{2j} \, r^{n-2j} (\cos \varphi)^{n-2j}} \\
	& = C \abs{\sum_{j=0}^{ \lfloor \frac{n}{2} \rfloor} (-1)^j \binom{n}{2j} \, y^{2j} \, x^{n-2j}} \\
	& = C \abs{P_0(x,y)}
	\end{align*}
	with the polynomial $P_0(x,y):= \displaystyle \sum_{j=0}^{ \lfloor \frac{n}{2} \rfloor} (-1)^j \binom{n}{2j} \, y^{2j} \, x^{n-2j}$, where we used a multiple-angle formula for the cosinus \cite[Chapter 2.6.2]{BronsteinSemendjajewMusiolMuhlig2001}. Hence, it is clear that $u \in L^2(\Omega)$. Next, we recall estimate \eqref{eq: estimate of gradient} and also apply a multiple-angle formula for the sinus \cite[Chapter 2.6.2]{BronsteinSemendjajewMusiolMuhlig2001} to obtain
	\begin{align*}
	|\nabla u_{n,m}| &
	\leq  C \left( \abs{\left(\frac{n}{r} \, J_{n}(\mu_{n,m} \, r) - \mu_{n,m} \, J_{n+1}(\mu_{n,m} \, r) \right) \cos(n \, \varphi)} + \frac1r \abs{J_{n}(\mu_{n,m} \, r)  \, n \, \sin(n \, \varphi)} \right) \\
	& \leq C \left( \left( r^{n-1} - r^{n+1} \right) \abs{\cos(n \, \varphi)} + r^{n-1} \abs{\sin(n \, \varphi)} \right) \\
	& \leq C r^{n-1} \left(\abs{\cos(n \, \varphi)} +  \abs{\sin(n \, \varphi)} \right) \\
	& = C r^{n-1} \left(\abs{\sum_{j=0}^{ \lfloor \frac{n}{2} \rfloor} (-1)^j \binom{n}{2j} (\sin \varphi)^{2j}  (\cos \varphi)^{n-2j}} + \abs{\sum_{j=0}^{ \lfloor \frac{n-1}{2} \rfloor} (-1)^j \binom{n}{2j+1} (\sin \varphi)^{2j+1}  (\cos \varphi)^{n-2j-1}} \right) \\
	& \leq C r^{n-1} \left(\sum_{j=0}^{ \lfloor \frac{n}{2} \rfloor} \binom{n}{2j} \abs{\sin \varphi}^{2j}  \abs{\cos \varphi}^{n-2j} + \sum_{j=0}^{ \lfloor \frac{n-1}{2} \rfloor} \binom{n}{2j+1} \abs{\sin \varphi}^{2j+1} \abs{\cos \varphi}^{n-2j-1} \right) \\
	& \leq C r^{n-1} \left(\sum_{j=0}^{ \lfloor \frac{n}{2} \rfloor} \binom{n}{2j} \abs{\sin \varphi}^{2j-1}  \abs{\cos \varphi}^{n-2j} + \sum_{j=0}^{ \lfloor \frac{n-1}{2} \rfloor} \binom{n}{2j+1} \abs{\sin \varphi}^{2j} \abs{\cos \varphi}^{n-2j-1} \right) \\
	& = C \left(\sum_{j=0}^{ \lfloor \frac{n}{2} \rfloor} \binom{n}{2j} \abs{r \sin \varphi}^{2j-1}  \abs{r \cos \varphi}^{n-2j} + \sum_{j=0}^{ \lfloor \frac{n-1}{2} \rfloor} \binom{n}{2j+1} \abs{r \sin \varphi}^{2j} \abs{r \cos \varphi}^{n-2j-1} \right) \\
	& = C \left(\sum_{j=0}^{ \lfloor \frac{n}{2} \rfloor} \binom{n}{2j} \abs{x}^{2j-1}  \abs{y}^{n-2j} + \sum_{j=0}^{ \lfloor \frac{n-1}{2} \rfloor} \binom{n}{2j+1} \abs{x}^{2j} \abs{y}^{n-2j-1} \right) \\
	& \leq C P_1(\abs{x}, \abs{y})
	\end{align*}
	with the polynomial $P_1(x,y)= \displaystyle \sum_{j=0}^{ \lfloor \frac{n}{2} \rfloor} \binom{n}{2j} x^{2j-1}  y^{n-2j} + \displaystyle \sum_{j=0}^{ \lfloor \frac{n-1}{2} \rfloor} \binom{n}{2j+1} x^{2j} y^{n-2j-1}$. Note that for some of the inequalities above it is essential that $r \in [0,1]$ and $\abs{\sin \varphi} \in [0,1]$. Thus, it follows that $\nabla u_{n,m} \in L^2(\Omega)$ and $u_{n,m}\in H^1(\Omega)$. Similarly, for higher derivatives $D^\alpha u_{n,m}$ with any multi-index $\alpha \in \R^2$, there exists a polynomial $P_{\abs{\alpha}}$ such that
	\begin{align*}
	\abs{D^\alpha u_{n,m}} \leq C P_{\abs{\alpha}}(\abs{x},\abs{y}) \, .
	\end{align*}
	Therefore, it is $D^\alpha u_{n,m} \in L^2(\Omega)$ for all $\alpha \in \R^2$ and we obtain
	\begin{align*}
	u_{n,m} \in H^s(\Omega) \quad \text{ for all } s \in \N_0  \, . 
	\end{align*}
	Using an interpolation argument for the Sobolev-Slobodeckij spaces, it immediately follows that $u_{n,m} \in H^s(\Omega)$ for all $s \geq 0$ and the lemma is proven.
\end{proof}

Throughout this paper, we will call the eigenfunctions $u_{\nu_k,m}$ with $\nu_k \in \N_0$ and $m\in \N$ smooth in the sense of Lemma \ref{lemma: smooth eigenfunctions}. To conclude the regularity analysis, we derive some simple but illustrative conclusions from the two lemmata above.

\begin{remark} 
	\
	\begin{enumerate}
		\item For all $k \in \N_0$ and $m \in \N$ it is $u_{\nu_k,m} \in H^1(\Omega)$.
		\item The eigenfunctions $u_{\nu_0,m}$ of circular sectors with arbitrary angles $\omega \in (0,2\pi]$ are smooth for all $m \in \N$.
		\item Consider a circular sector whose angle $\omega$ satisfies $\frac{\pi}{\omega} \in \N$, that is, $\omega = \frac{\pi}{N}$ for some $N \in \N$. Then all its Laplace eigenfunctions $u_{\nu_k,m}$ are smooth as $\nu_k = k \frac{\pi}{\omega} \in \N$ for all $k \in \N_0$.
		\item If the angle of a circular sector can not be expressed as a product of a rational number and $\pi$, i.e., $\omega \neq q \pi$ for all $q \in \Q$, it is  $\nu_k = k \frac{\pi}{\omega} \notin \N$ for all $k \in \N$, that is, none of the eigenfunctions $u_{\nu_k,m}$ for $k>0$ is smooth.
		\item The eigenfunctions $u_{\nu_1,m}$ of circular sectors with angles $\omega\in(\pi,2\pi]$ do not belong to $H^2(\Omega)$ for all $m \in \N$. They contain the strongest singularity among all eigenfunctions, which is of type $r^{\nu_1}$ near the conical point of the circular sector, where $\nu_1 = \frac{\pi}{\omega} \in [\frac12,1)$.
	\end{enumerate}
\end{remark}

\subsection{Eigenfunctions and eigenfrequencies of the unit disk with crack}
\label{subsection: regularity of eigenfunctions for omega = 2pi}
To gain a deeper understanding of the model problem \eqref{eq: model problem} and its exact solutions, we will now discuss a specific example: the circular sector with angle $\omega = 2 \pi$, which represents the unit disk with a crack located along the positive $x$-axis, as presented in Figure \ref{fig: sketch of model domain b}. All other circular sectors with smaller angles $\omega < 2 \pi$ can be interpreted as analogs or even simplifications of the case $\omega = 2 \pi$.

The Laplace eigenfunctions of the unit disk with crack are given by the formula \eqref{eq: formula eigenfunctions} with $\nu_k = \frac{k}{2}$ for $k \in \N_0$. Utilizing the regularity results outlined in Section \ref{section: regularity of the eigenfunctions}, we categorize the eigenmodes into two types: 
\begin{enumerate}[label = (\Alph*)]
	\item $\nu_k \in \N_0$, i.e., $k = 2n$ for $n \in \N_0$. As indicated in Lemma \ref{lemma: smooth eigenfunctions}, the corresponding eigenfunctions $u_{n,m}$ are smooth for all $m \in \N$. Exemplary eigenfunctions, namely $u_{\nu_0,m}$ and $u_{\nu_2,m}$ for $m \in \{1,2,3\}$, are showcased in the first and third columns of Figure \ref{fig:12eigenfunctions}. It is noteworthy that the crack is not visible in graphical representations of these eigenfunctions. This absence is due to the product form \eqref{eq: product representation of the eigenfunctions} of the eigenmodes and the $2\pi$-periodicity of the angular functions $\varphi \mapsto \cos(\nu_k \, \varphi) = \cos(n \, \varphi)$. 
	\item $\nu_k \notin \N_0$, i.e., $k = 2n+1$ for $n \in \N_0$. The corresponding eigenfunctions $u_{n + \frac12,m}$ have a singularity of type $r^{n + \frac12}$ and are only contained in $H^{s}(\Omega)$ for all $s<n + \frac32$, recall Lemma \ref{lemma: regularity of the eigenfunctions}. As an example, the eigenmodes $u_{\nu_1,m}$ and $u_{\nu_3,m}$ for $m \in \{1,2,3\}$ have singularities of type $r^{1/2}$ and $r^{3/2}$, respectively, and are presented in the second and fourth columns of Figure \ref{fig:12eigenfunctions}. Here, the crack is clearly visible since the functions $\varphi \mapsto \cos(\nu_k \varphi) = \cos((n+\frac12) \varphi)$ are not $2\pi$-periodic.
\end{enumerate}

\begin{figure}[t]
	\includegraphics[width=0.245\linewidth, trim=3cm 2cm 3cm 3cm, clip]{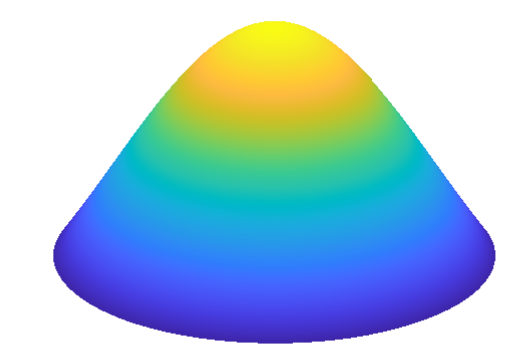}
	\includegraphics[width=0.245\linewidth, trim=3cm 2cm 3cm 3cm, clip]{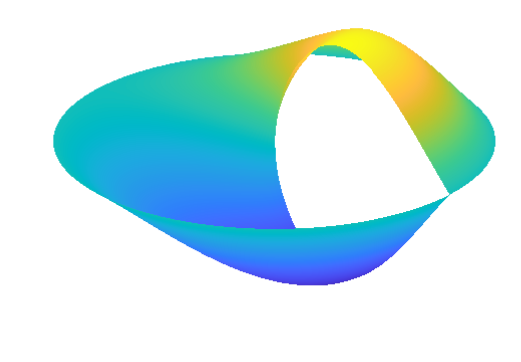}
	\includegraphics[width=0.245\linewidth, trim=3cm 2cm 3cm 3cm, clip]{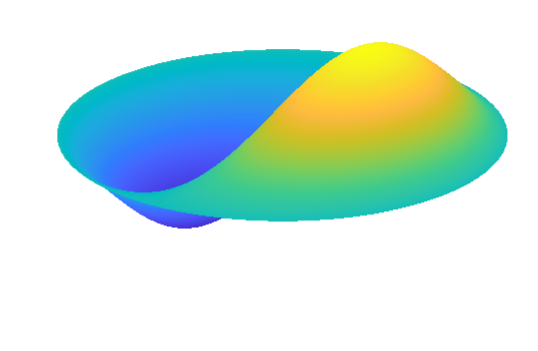}
	\includegraphics[width=0.245\linewidth, trim=3cm 2cm 3cm 3cm, clip]{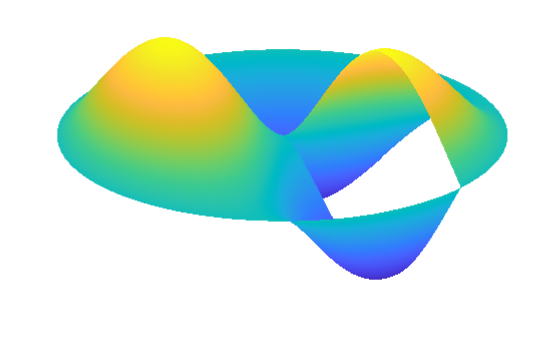}
	
	\includegraphics[width=0.245\linewidth, trim=3cm 2cm 3cm 3cm, clip]{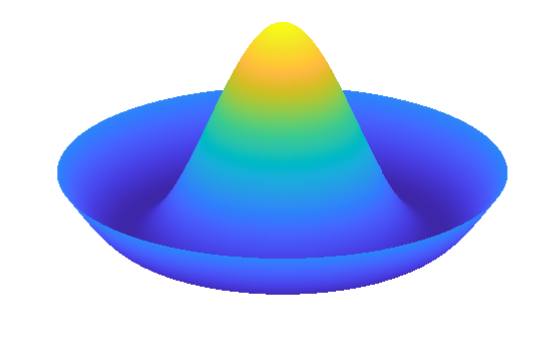}
	\includegraphics[width=0.245\linewidth, trim=3cm 2cm 3cm 3cm, clip]{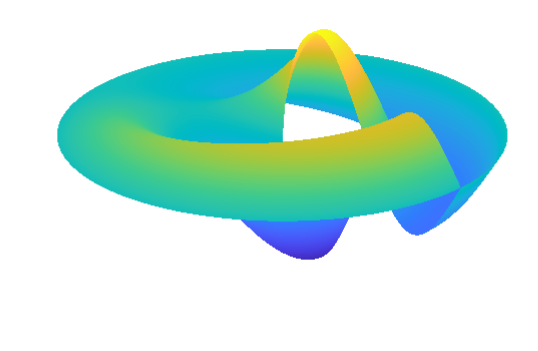}
	\includegraphics[width=0.245\linewidth, trim=3cm 2cm 3cm 3cm, clip]{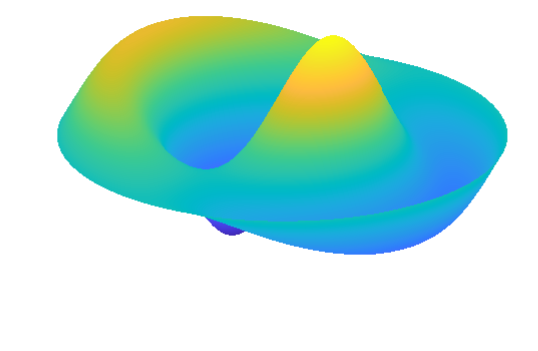}
	\includegraphics[width=0.245\linewidth, trim=3cm 2cm 3cm 3cm, clip]{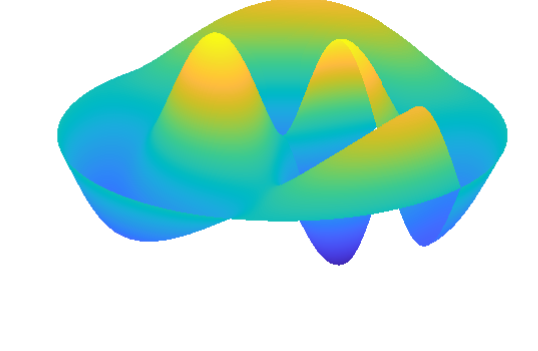}
	
	\includegraphics[width=0.245\linewidth, trim=3cm 2cm 3cm 3cm, clip]{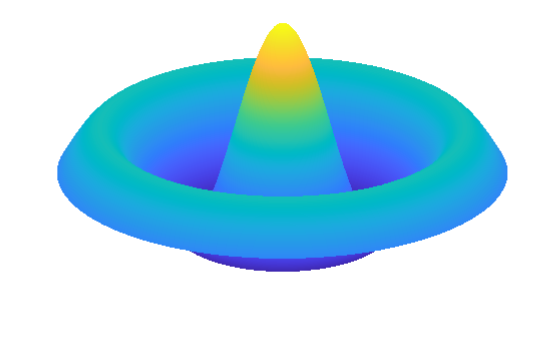}
	\includegraphics[width=0.245\linewidth, trim=3cm 2cm 3cm 3cm, clip]{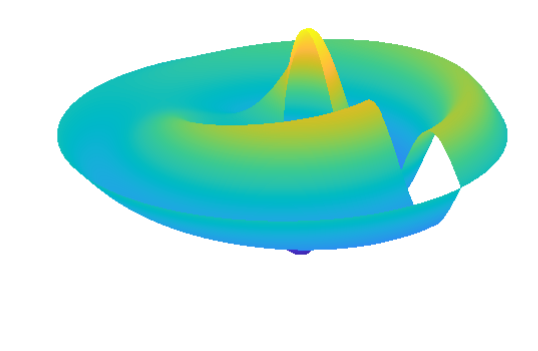}
	\includegraphics[width=0.245\linewidth, trim=3cm 2cm 3cm 3cm, clip]{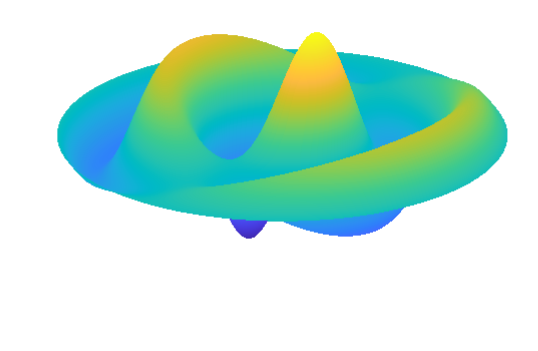}
	\includegraphics[width=0.245\linewidth, trim=3cm 2cm 3cm 3cm, clip]{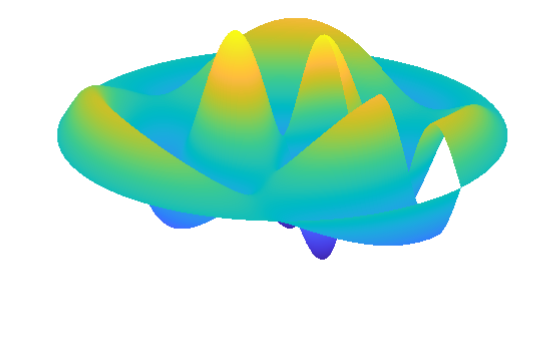}
	
	\vspace{-2mm}
	\caption{Illustration of $12$ Laplace eigenfunctions $u_{\nu_k,m}$ of the unit disk with crack. Columns from left to right: $\nu_k =0, \nu_k=\frac12, \nu_k =1, \nu_k=\frac32$. Rows from top to bottom: $m=1,m=2,m=3$.}
	\label{fig:12eigenfunctions}
\end{figure}

As indicated in Remark \ref{remark: product form of eigenfunctions}, the regularity properties of the Laplace eigenmodes of circular sectors are based on the corresponding Bessel functions. In Figure \ref{fig:besselfunszeros}, we visualize the Bessel functions $J_{\nu_k}$ to the orders $\nu_k \in \{0,\frac12,1, \frac32, \dots, \frac{11}{2}, 6\}$, which are all the half-integer and integer Bessel functions with zeros in the interval $[0,10]$. The asymptotic behavior \eqref{eq: asymptotic representation of Bessel funtion to the first kind} can be clearly observed, i.e., $J_{\nu_k}(z) \sim z^{\nu_k}$ for $z \to 0$, and $J_0$ differs from the other Bessel functions in the sense that $J_0(0)=1$ and $J_{\nu_k}(0)=0$ for $\nu_k \neq 0$. Bessel functions of half-integer order can be expressed explicitly \cite[Chapter 10.5]{Strauss2008}. For instance, it holds 
\begin{align*}
J_{\frac12}(z)=\sqrt{\frac{2}{\pi z}} \sin z \, ,
\end{align*}
which once again highlights the presence of a singularity of type $z^{1/2}$ in the Bessel function $J_{\frac12}$.

\begin{figure}[t]
	\centering
	\includegraphics[width=0.75\linewidth, trim=3cm 0.5cm 3cm 1.3cm, clip]{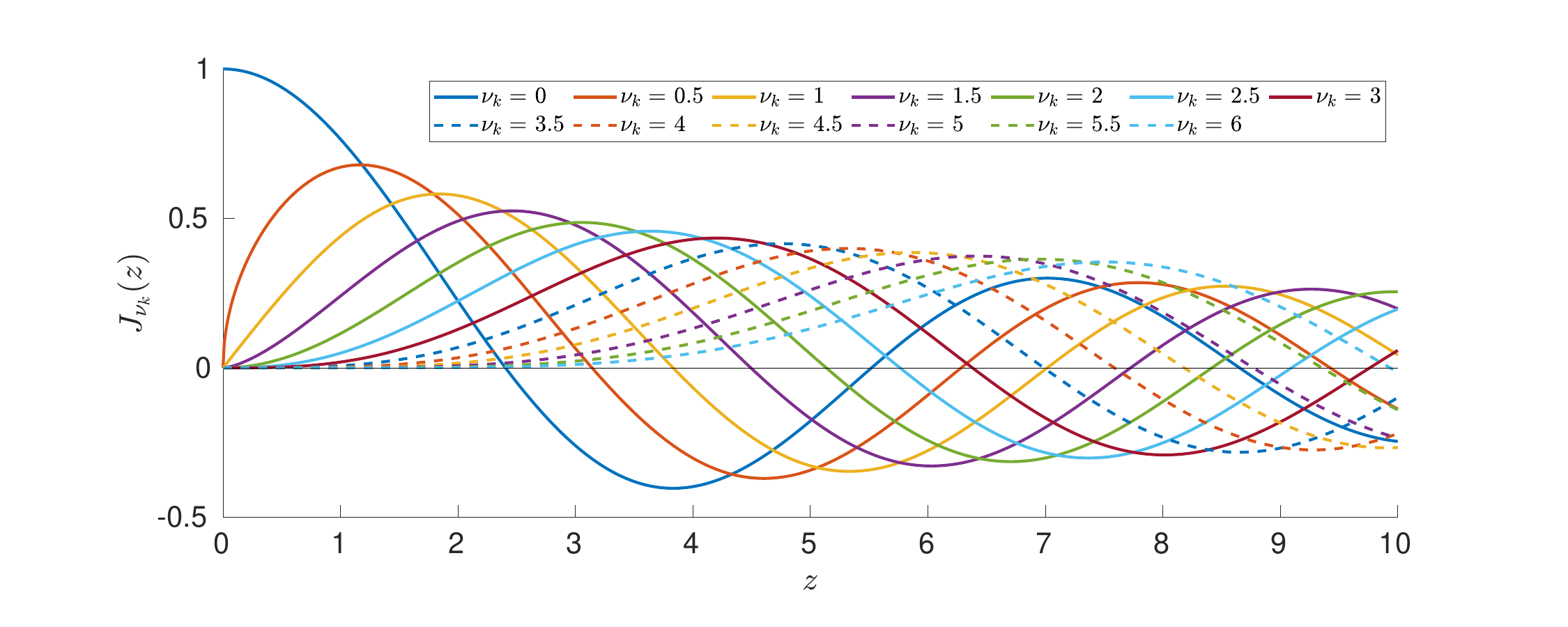}
	\vspace{-2mm}
	\caption{Visualization of the Bessel functions $J_{\nu_k}$ to the orders $\nu_k \in \{0,0.5,1, \dots, 5.5,6\}$ in the interval $[0,10]$.}
	\label{fig:besselfunszeros}
\end{figure}

Table \ref{tab:ZerosOfBesselfunctions} provides the eigenfrequencies of the unit disk with crack within the interval $[0,10]$, listed in ascending order. They are given by the zeros of the functions depicted in Figure \ref{fig:besselfunszeros}. Each eigenfrequency is associated with a specific root of a Bessel function, indicating the corresponding eigenfunction and its regularity. The eigenfrequencies $\omega_2$, $\omega_8$ and $\omega_{20}$, highlighted in dark orange, correspond to eigenfunctions with the strongest singularity, $u_{1/2,1}$, $u_{1/2,2}$ and $u_{1/2,3}$, respectively. Likewise, the frequencies $\omega_4$ and $\omega_{13}$, marked in light orange, belong to eigenfunctions with the second strongest singularity, $u_{3/2,1}$ and $u_{3/2,2}$, respectively, and so forth. Note that the occurrence of eigenfunctions characterized by identical singularities becomes rare as the frequencies increase. Loosely speaking, as $N \to \infty$, we observe that out of $N^2$ eigenfunctions, approximately $N$ contain a singularity of type $r^{1/2}$, $r^{3/2}$, and so on.

\begin{table}[t]
	\begin{center}
		\begin{tabular}{ c|c|c|c|c|c } 
			Eigenfrequency & Value &  $\nu_k = k/2$ & $m$ & Eigenfunction & Regularity \\ 
			\hline
			$\omega_1$ & $2.40$ & $0$ & $1$ & $u_{0,1}$ & smooth \\ 
			\rowcolor{orange!60}
			$\omega_2$ & $3.14$ & $0.5$ & $1$ & $u_{0.5,1}$ & $H^1(\Omega)$\\ 
			$\omega_3$ & $3.83$ & $1$ & $1$ & $u_{1,1}$ & smooth\\ 
			\rowcolor{orange!30}
			$\omega_4$ & $4.49$ & $1.5$ & $1$ & $u_{1.5,1}$ & $H^2(\Omega)$ \\ 
			$\omega_5$ & $5.14$ & $2$ & $1$ & $u_{2,1}$ & smooth \\ 
			$\omega_6$ & $5.52$ & $0$ & $2$ & $u_{0,2}$ & smooth \\ 
			$\omega_7$ & $5.76$ & $2.5$ & $1$ & $u_{2.5,1}$ & $H^3(\Omega)$ \\
			\rowcolor{orange!60} 
			$\omega_8$ & $6.28$ & $0.5$ & $2$ & $u_{0.5,2}$ & $H^1(\Omega)$\\ 
			$\omega_9$ & $6.38$ & $3$ & $1$ & $u_{3,1}$ & smooth \\ 
			$\omega_{10}$ & $6.99$ & $3.5$ & $1$ & $u_{3.5,1}$ & $H^4(\Omega)$ \\
			$\omega_{11}$ & $7.02$ & $1$ & $2$ & $u_{1,2}$ & smooth \\
			$\omega_{12}$ & $7.59$ & $4$ & $1$ & $u_{4,1}$ & smooth \\ 
			\rowcolor{orange!30}
			$\omega_{13}$ & $7.73$ & $1.5$ & $2$ & $u_{1.5,2}$ & $H^2(\Omega)$ \\
			$\omega_{14}$ & $8.18$ & $4.5$ & $1$ & $u_{4.5,1}$ & $H^5(\Omega)$\\  
			$\omega_{15}$ & $8.42$ & $2$ & $2$ & $u_{2,2}$ & smooth \\  
			$\omega_{16}$ & $8.65$ & $0$ & $3$ & $u_{0,3}$ & smooth \\  
			$\omega_{17}$ & $8.77$ & $5$ & $1$ & $u_{5,1}$ & smooth \\  
			$\omega_{18}$ & $9.10$ & $2.5$ & $2$ & $u_{2.5,2}$ & $H^3(\Omega)$ \\  
			$\omega_{19}$ & $9.36$ & $5.5$ & $1$ & $u_{5.5,1}$ & $H^6(\Omega)$\\  
			\rowcolor{orange!60}
			$\omega_{20}$ & $9.42$ & $0.5$ & $3$ & $u_{0.5,3}$ & $H^1(\Omega)$\\  
			$\omega_{21}$ & $9.76$ & $3$ & $2$ & $u_{3,2}$ & smooth \\  
			$\omega_{22}$ & $9.94$ & $6$ & $1$ & $u_{6,1}$ & smooth \\ 
			%			$\omega_{23}$ & $10.1735$ & $1$ & $3$ & $u_{1,3}$ \\ 
			%			$\omega_{24}$ & $10.4171$ & $3.5$ & $2$ & $u_{7/2,2}$ \\ 
			%			$\omega_{25}$ & $10.5128$ & $6.5$ & $1$ & $u_{13/2,1}$ \\ 
			%			$\omega_{26}$ & $10.9041$ & $1.5$ & $3$ & $u_{3/2,3}$ \\     
		\end{tabular}
	\end{center}
	\vspace{-2mm}
	\caption{The eigenfrequencies \eqref{eq: sequence of eigenvalues} of the unit disk with crack are listed in ascending order within the interval $[0,10]$ . As per equation \eqref{eq: formula eigenvalues}, each eigenfrequency corresponds to the $m$-th root of a Bessel function $J_{\nu_k}$, denoted by $\omega_n = \mu_{\nu_k,m}$ for some $\nu_k \in \{0, \frac12, 1, \frac32, \dots \}$ and $m \in \N$. Additionally, the corresponding eigenfunction and its regularity are provided. Rows containing eigenfunctions with the strongest and second strongest singularity are highlighted in dark orange and light orange, respectively.}
	\label{tab:ZerosOfBesselfunctions}
\end{table}

\subsection{Numerical solution process}
\label{subsec: numerical solution process}
Having explored the analytical solution of the Laplace eigenvalue problem on circular sectors, we now describe how it can be approached numerically. Therefore, we introduce the corresponding variational form, also known as weak form, of the model problem \eqref{eq: model problem}, see \cite[Chapter I.3]{BabuskaOsborn1991} for more details. We choose
\begin{align*}
V:= H^1(\Omega) 
\end{align*}
as the underlying Hilbert space for the test and basis functions. Taking into account the boundary conditions, we further define
\begin{align*}
V_0 := \{v \in  H^1(\Omega) : v = 0 \text{ on } \Gamma_D \} \, .
\end{align*} 
Multiplying the eigenvalue equation with test functions $v \in V_0$ and integrating over the domain $\Omega$ leads to the equation
\begin{align*}
\int_\Omega \nabla u \cdot \nabla v \, \mathrm{d} \bs x = \lambda \int_\Omega u \, v \, \mathrm{d} \bs x \, .
\end{align*}
Hence, we define the bilinear forms
\begin{align*}
a : V \times V \to \R, \quad & a( u,  v)  := \int_\Omega \nabla u \cdot \nabla v \, \mathrm{d} \bs x\, ,  \\
m : V \times V \to \R, \quad & m( u,  v)  := \int_\Omega u \, v \,  \mathrm{d} \bs x
\end{align*}
and obtain the weak formulation of the eigenvalue problem: Find $(\lambda, u) \in \R^+ \times V_0 \setminus \{0\}$ such that
\begin{align}
a( u,  v)	= \lambda \,  m( u,  v)  \quad \forall v \in V_0 \, .
\label{eq: weak formulation}
\end{align}
Knowing that both bilinear forms are bounded,
\begin{align*}
a( u,  v) &\leq C_1 \norm{u}_V  \norm{v}_V \,  \quad \forall u,v \in V \, , \\
m( u,  v) &\leq C_2 \norm{u}_V  \norm{v}_V \,  \quad \forall u,v \in V
\end{align*}
and $a$ is coercive on $V_0$,
\begin{align*}
a( u,  u) \geq \alpha \norm{u}^2_V  \quad \forall u \in V_0 \, ,
\end{align*}
it can be shown \cite[Chapter 4]{BabuskaOsborn1991} that the eigenvalue problem \eqref{eq: weak formulation} has an infinite sequence of eigenvalues
\begin{align*}
0<\lambda_1 \leq \lambda_2 \leq \cdots \ \nearrow + \infty 
\end{align*}
and corresponding eigenfunctions $u_1, u_2, \ldots \in V_0$.

Numerical approaches like finite element methods (FEM) and isogeometric analysis (IGA) are based on Galerkin's principle. Let $V_{0h} \subset V_0$ be a $l$-dimensional subspace. Instead of the infinite-dimensional problem \eqref{eq: weak formulation}, we solve the following discrete eigenvalue problem: Find $u_h \in V_{0h}$ such that
\begin{align}
\label{eq: discrete weak formulation}
a( u_h,  v_h)	= \lambda_h \,  m( u_h,  v_h)  \quad \forall v_h \in V_{0h} \, ,
\end{align}
which yields a finite sequence of eigenvalues \cite[Chapter 5]{BabuskaOsborn1991}
\begin{align*}
0<\lambda_{1,h} \leq \lambda_{2,h}  \leq \cdots \leq  \lambda_{k,h} 
\end{align*}
and corresponding eigenfunctions $u_{1,h}, u_{2,h}, \dots, u_{k,h} \in V_{0h}$. Further, let $\{\varphi_1, \dots, \varphi_l\}$ be a basis of $V_{0h}$. Then, each eigenfunction $u_h$ can be represented as a linear combination of the basis functions, $u_h = \sum_{j=1}^l u_h^j \, \varphi_j$,
with coefficients $u_h^j \in \R$ for all $j \in \{1,...,l\}$. Inserting this into the discrete weak formulation \eqref{eq: discrete weak formulation} leads to the discrete matrix eigenvalue problem
\begin{align}
A \, \vec{u}_h  = \lambda_h \, M  \, \vec{u}_h  \,, 
\label{eq: matrix eigenvalue problem}
\end{align}
where 
\begin{align*}
A = \left(a(\varphi_j, \varphi_i)\right)_{i,j \in  \{1,2, \dots, l \} }, \ 
\vec{u}_h  = \left( u_h^j \right)_{j \in  \{1,2, \dots, l \} } \ \text{ and } 
M = \left(m(\varphi_j, \varphi_i) \right)_{i,j \in  \{1,2, \dots, l \} } \, ,
\end{align*}
which can be solved numerically. The approximated eigenfunctions and eigenvalues converge to the exact ones for $h \to 0$ if the spaces $V_{0h}$ and basis functions $\{\varphi_1, \dots, \varphi_l\}$ are chosen appropriately. This choice varies throughout different numerical methods. While FEM typically employ continuous piecewise polynomials, IGA makes use of B-splines and NURBS, a class of functions that will be introduced in the following section.

\section{Basics of isogeometric analysis}
\label{section: preliminaries and notation}
We present the essential basics of isogeometric analysis and establish necessary notation, largely following the framework presented in the review paper \cite{BeiraodaVeigaBuffaSangalliVazquez2014}. A more detailed introduction can be found in the books \cite{CottrellHughesBazilevs2009} on isogeometric analysis and \cite{Schumaker2007} on spline theory.

\subsection{Univariate B-splines and NURBS}
First, we introduce the concept of B-splines and NURBS in the univariate case and explain their application for the modeling of curves. Additionally, we provide a brief overview of common refinement procedures.

\subsubsection{Definition and properties}
\label{subsec: univariate splines}
Let $p \in \N_0$ and $n \in \N$. We call $\Xi := \{\xi_1, \xi_2, \dots, \xi_{n+p+1}\}$ a $p$-open knot vector if
\begin{align*}
\xi_1 = \xi_2 = \dots = \xi_{p+1} < \xi_{p+2} \leq \xi_{p+3} \leq \dots \leq \xi_{n-1} \leq \xi_{n} < \xi_{n+1} = \xi_{n+2} = \xi_{n+p+1} \, ,
\end{align*}
where $\xi_i \in \R$ for $i =1, \dots, n+p+1$ is called the $i$-th knot which is allowed to occur repeatedly. Without loss of generality, we assume that $\xi_1=0$ and $\xi_{n+p+1}=1$ and hence all the knots are contained in the unit interval $[0,1]$. Furthermore, we define the vector $Z = \{\zeta_1, \dots, \zeta_N\}$ of knots without repetitions, also called breakpoints, with
\begin{align*}
\Xi = \{\underbrace{\zeta_1, \dots, \zeta_1}_{m_1}, \underbrace{\zeta_2, \dots, \zeta_2}_{m_2}, \dots, \underbrace{\zeta_N, \dots, \zeta_N}_{m_N}\} \, ,
\end{align*}
where $N \in \N$ is the total number of pairwise different knots and $m_j \in \N$ for $j =1, \dots, N$ denotes the multiplicity of the breakpoint $\zeta_j$ such that $\sum_{j=1}^N m_j = n + p + 1$. For $p$-open knot vectors, $m_1=m_N=p+1$ always holds. Further, we assume $m_j \leq p+1$, $j=2,3,\dots,N-1$, for the multiplicities of the internal knots. 

The entries of $Z$ define a mesh on the unit interval $[0,1]$ and the local mesh size of the element $I_i = (\zeta_i, \zeta_{i+1})$ is called $h_i= \zeta_{i+1}- \zeta_i$ for $i=1,\dots,N-1$.
The B-Spline functions of degree $p$ are denoted by
\begin{align*}
\widehat{B}_{i,p} \colon [0,1] \to \R , \quad \zeta \mapsto \widehat{B}_{i,p}(\zeta) \, , \quad  i= 1, 2, \dots, n \, ,
\end{align*}
and can be defined recursively from the knot vector $\Xi$ using the Cox--de Boor formula, 
\begin{align*}
&\widehat{B}_{i,0}(\zeta) = \begin{cases} 1 \quad &\text{ if } \xi_i \leq \zeta < \xi_{i+1}\, , \\
0 \quad &\text{ else}  \, ,
\end{cases}
\end{align*}
and for $p \geq 1$,
\begin{align*}
&\widehat{B}_{i,p}( \zeta) = \frac{\zeta - \xi_i}{\xi_{i+p}-\xi_i} \widehat{B}_{i,p-1}(\zeta) + \frac{\xi_{i+p+1} - \zeta}{\xi_{i+p+1} - \xi_{i+1}}  \widehat{B}_{i+1,p-1}(\zeta) \, ,
\end{align*}
with $0/0:= 0$. The B-spline functions, besides other characteristics, are non-negative and form a partition of unity. Furthermore, they build a basis of the space of splines, that is, piecewise polynomials of degree $p$ with $k_j : = p-m_j$ continuous derivatives at the breakpoints $\zeta_j$, $j=1, \dots, N$. As $m_j \leq p+1$, we obtain $-1 \leq k_j \leq p-1$, where $k_j=-1$ occurs for the maximum knot multiplicity of $m_j=p+1$ and stands for a discontinuity at $\zeta_j$. We define the regularity vector $\bs k = \{k_1, \dots, k_N\}$ to collect the regularities at the breakpoints, remarking that $k_1=k_N=-1$ always holds for open knot vectors.

In Figure \ref{fig: basisfunp3k2}, we illustrate some cubic B-spline functions to showcase their possibly varying regularity properties at internal knots. The B-splines in Figure \ref{fig: basisfunp3k2 a} are constructed with the open knot vector $\Xi=\{0,0,0,0,0.25,0.5, \allowbreak 0.75,1,1,1,1\}$ and have $C^2$-regularity at the breakpoints $0.25$, $0.5$ and $0.75$. In contrast, Figure \ref{fig: basisfunp3k2 b} demonstrates B-spline functions defined by the open knot vector $\Xi=\{0,0,0,0,\allowbreak 0.25,0.25,0.25,\allowbreak 0.5,0.5,0.75,\allowbreak 1,1,1,1\}$, where repeated interior knot values lead to $C^0$, $C^1$ and $C^2$-regularity at the breakpoints $0.25$, $0.5$ and $0.75$, respectively.

\begin{figure}
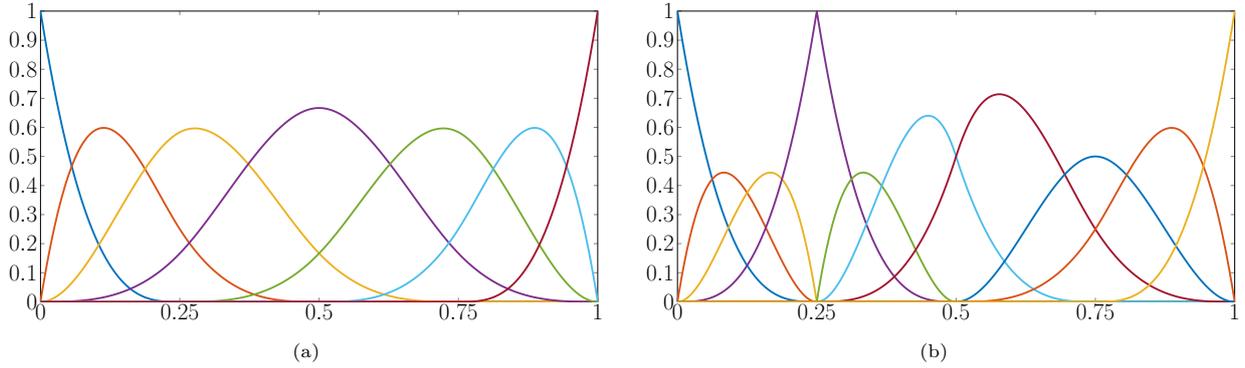

	\begin{center}
		\begin{subfigure}{0.495\textwidth}
			\begin{center}
				\resizebox{0.97\linewidth}{!}{\input{basisfun_1D_k=2_thicklines.tikz}}
			\end{center}
			\vspace{-3mm}
			\caption{}
			\label{fig: basisfunp3k2 a}
		\end{subfigure}
		\begin{subfigure}{0.495\textwidth}
			\begin{center}
				\hfill
				\resizebox{0.97\linewidth}{!}{\input{basisfun_1D_k=0-2_thicklines.tikz}}
			\end{center}
			\vspace{-3mm}
			\caption{}
			\label{fig: basisfunp3k2 b}
			\label{fig: test}
		\end{subfigure}
		\vspace{-2mm}
		\caption{Examples of cubic B-spline functions with distinct regularity properties, defined by different knot vectors $\Xi$. (a): $\Xi=\{0,0,0,0,0.25,0.5,0.75,1,1,1,1\}$. (b): $\Xi=\{0,0,0,0,0.25,0.25,0.25,0.5,0.5,0.75,1,1,1,1\}$.}
		\label{fig: basisfunp3k2}
	\end{center}
\end{figure}

Finally, we denote the univariate spline space spanned by the B-spline functions corresponding to $\Xi$ by
\begin{align*}
S_p(\Xi) = \text{span}\left\{\widehat{B}_{i,p} : i=1,2,\dots,n \right\} \, .
\end{align*}
The definition of each B-spline $\widehat{B}_{i,p}$ depends only on $p+2$ knots, which motivates the definition of a local knot vector $\Xi_{i,p} := \{\xi_i, \dots, \xi_{i+p+1}\}$. In this context, the notation will sometimes be adapted to $\widehat{B}_{i,p} (\zeta) = \widehat{B}[\Xi_{i,p}](\zeta)$. 

\subsubsection{B-spline and NURBS curves}
By associating each basis function $\widehat{B}_{i,p}$, $i=1,\dots,n$ with a control point $\bs c_i \in \R^m$, $m=2,3$, we define a spline curve as the linear combination of B-splines,
\begin{align*}
\bs C(\zeta) = \sum_{i=1}^n \bs c_i \, \widehat{B}_{i,p}(\zeta) \, .
\end{align*}
However, classical splines face restrictions in representing important geometries like conic sections. To overcome this limitation, non-uniform rational B-splines (NURBS) are introduced, see \cite{PieglTiller1995} for more details. Therefore, the weight function
\begin{align*}
W(\zeta) = \sum_{l=1}^{n} w_l \, \widehat{B}_{l,p} (\zeta)
\end{align*}
is determined by choosing positive constants $w_l>0$, $l=1,\dots, n$, which are called weights. The NURBS basis functions are then defined by
\begin{align*}
\widehat{N}_{i,p} (\zeta) = \frac{w_i  \widehat{B}_{i,p} (\zeta)}{\sum_{l=1}^n w_l \, \widehat{B}_{l,p} (\zeta)} = \frac{w_i  \widehat{B}_{i,p} (\zeta)}{W(\zeta)} \, , \quad i = 1, \dots, n \, .
\end{align*}
We denote the corresponding NURBS space by
\begin{align*}
N_p(\Xi, W) = \text{span}\left\{ \widehat{N}_{i,p} : i = 1, \dots, n \right \} \, .
\end{align*}
A NURBS curve can now be introduced in the same way as a spline curve,
\begin{align*}
\bs C(\zeta) = \sum_{i=1}^n \bs c_i \, \widehat{N}_{i,p}(\zeta) \, ,
\end{align*}
where again $\{\bs c_i\}_{i=1, \dots,n}$ is the set of control points $ \bs c_i \in \R^m$, $m=2,3$, that are chosen depending on the geometry to be parameterized.

\subsubsection{Refinement procedures}
\label{sec: refinement procedures}
It is a major principle of isogeometric analysis that the parameterization $\bs C(\zeta)$ does not change during refinement. Therefore, the refined spaces have to contain the space of the parameterization. We say that $S_p(\Xi)$ is a refinement of a given spline space $S_{p^0}(\Xi^0)$ if 
\begin{align}
\label{eq: refined spline spaces}
S_{p^0}(\Xi^0) \subset S_p(\Xi) \, .
\end{align}
When NURBS are used, it is also required that the weight function $W(\zeta)$ remains unchanged. We say that $N_p(\Xi, W)$ is a refinement of a given NURBS space $N_{p^0}(\Xi^0,W)$ if
\begin{align}
\label{eq: refined NURBS spaces}
N_{p^0}(\Xi^0,W) \subset N_p(\Xi,W) \, ,
\end{align}
which, as $W$ does not change, comes down to the condition \eqref{eq: refined spline spaces} for the corresponding spline spaces.

Spline and NURBS curves can be refined through knot insertion and degree elevation. Both procedures offer formulas to recalculate control points and weights such that the parameterization does not change. In total, three refinement types can be constructed by combining the algorithms. Knot insertion results in the classical mesh refinement known as $h$-refinement, while degree elevation enables $p$-refinement. Consecutive application of degree elevation and knot insertion is called $k$-refinement. This approach allows to maintain the regularity at the internal knots of $\Xi_0$ while enhancing differentiability at all other knots. For a more detailed description of these refinement procedures we refer the interested reader to \cite{HughesCottrellBazilevs2005}, where they were initially introduced.

\subsection{Bivariate B-splines and NURBS}
Multivariate B-splines and NURBS are constructed as tensor products of their univariate counterparts. This section provides a summary of the main concepts and notation. The focus of our paper is on the analysis of circular sectors, so we set the spatial dimension $d=2$ and specifically consider bivariate B-splines and NURBS. However, all definitions and notations can be easily extended to higher dimensions.

\subsubsection{Definition and properties}
\label{sec: Definition multivariate splines}
Let $n_l \in \N$, the degrees $p_l\in \N$ and the $p_l$-open knot vectors $\Xi_l = \{\xi_{l,1}, \xi_{l,2} \dots , \xi_{l,n_l+p+1}\}$ be given for $l=1,2$. We define the polynomial degree vector $\bs p = (p_1, p_2)$ and the bivariate knot vector $\bs \Xi = \Xi_1 \times \Xi_2$. Further, let $N_l \in \N$ be the number of knots without repetition for each direction such that the corresponding univariate knot vectors of breakpoints are given by $Z_l = \{\zeta_{l,1}, \zeta_{l,2} \dots , \zeta_{l,N_l}\}$ for $l=1,2$. They form a Cartesian grid in the parametric domain $\widehat{\Omega} = (0,1)^2$ which defines the parametric B\'ezier mesh $\widehat{\mathcal{M}}$, 
\begin{align}
\label{eq: parametric Bezier mesh}
\widehat{\mathcal{M}} := \left \{Q_{\bs j} \subset \widehat{\Omega} : Q_{\bs j} = Q_{(j_1,j_2)} = (\zeta_{1,j_1} , \zeta_{1,j_1 + 1}) \times (\zeta_{2,j_2} , \zeta_{2,j_2 + 1}), \bs j \in \bs J \right \} \, , 
\end{align}
where we introduce the set of multi-indices $\bs J= \{\bs j = (j_1, j_2) : 1 \leq j_l \leq N_l - 1,  l=1,2\}$. Moreover, we set $\bs I = \{\bs i = (i_1, i_2) : 1 \leq i_l \leq n_l,  l=1,2\}$ and denote the local bivariate knot vector for each multi-index $\bs i = (i_1,i_2)$ by
\begin{align*}
\bs \Xi_{\bs i, \bs p} = \Xi_{i_1,p_1} \times \Xi_{i_2,p_2} \, .
\end{align*}
Bivariate B-spline functions are then defined for every $\bs i \in \bs I$ via
\begin{align*}
\widehat{B}_{\bs i, \bs p} : [0,1]^2 \to \R , \quad \widehat{B}_{\bs i, \bs p} (\bs \zeta) = \widehat{B}[\Xi_{i_1,p_1}](\zeta_1) \,  \widehat{B}[\Xi_{i_2,p_2}](\zeta_2) 
\end{align*}
and the tensor product spline space in the parametric domain $\widehat{\Omega}$ is denoted by
\begin{align*}
S_{\bs p} (\bs \Xi ) 
=  \text{span} \left \{\widehat{B}_{\bs i, \bs p} (\bs \zeta), \bs i \in \bs I \right \} 
= S_{p_1} (\Xi_1) \times S_{p_2} (\Xi_2) \, .
\end{align*}
Further, we define multivariate NURBS basis functions
\begin{align*}
\widehat{N}_{\bs i, \bs p} (\bs \zeta) = \frac{w_{\bs i} \, \widehat{B}_{\bs i, \bs p} (\bs \zeta)} {W(\bs \zeta)}
\end{align*}
using the weight function
\begin{align}
W(\bs \zeta) = \sum_{\bs l \in\bs I}  w_{\bs l} \, \widehat{B}_{\bs l, \bs p} (\bs \zeta)\, ,
\label{eq: weight function}
\end{align}
where we choose weights $w_{\bs l} > 0$ for all $\bs l \in \bs I$. The space of NURBS on the parametric domain is finally denoted by
\begin{align*}
N_{\bs p}(\bs \Xi, W) = \text{span} \left \{\widehat{N}_{\bs i, \bs p} (\bs \zeta) , \bs i \in \bs I \right \} \, .
\end{align*}
Moreover, the refinement algorithms introduced in Section \ref{sec: refinement procedures} can be generalized to bivariate spline and NURBS spaces. 

\subsubsection{B-spline and NURBS surfaces}
\label{subsec: B-spline and NURBS surfaces}
B-spline and NURBS surfaces in $\R^m$, $m=2,3$, are defined as linear combinations of the tensor product functions introduced above,
\begin{align}
\bs F(\bs \zeta) = \sum_{\bs i \in \bs I} \bs c_{\bs i} \widehat{B}_{\bs i, \bs p} (\zeta) \quad \text{ and } \quad
\bs F(\bs \zeta) = \sum_{\bs i \in \bs I} \bs c_{\bs i} \widehat{N}_{\bs i, \bs p} (\zeta) \, ,
\label{eq: formula parameterization}
\end{align}
respectively, where, similar to the univariate case, each basis function is associated with a control point $\bs c_{\bs i} \in \R^m$, $\bs i \in \bs I$. The $\bs F$-image $\Omega = \bs F(\hat{\Omega})$ of the parametric domain $\hat{\Omega} = (0,1)^2$ is commonly referred to as the physical domain. Using this construction, exact parameterizations
\[\bs F \colon  \hat{\Omega} \to \Omega \]
of various types of domains $\Omega \subset \R^2$, including circular sectors, can be obtained. For further details, the interested reader may consult the books \cite{Farin1995,PieglTiller1995,CottrellHughesBazilevs2009}. 

To define a mesh in $\Omega$, we consider the image under $\bs F$ of the partition given by the knot vectors without repetitions, i.e., each element $Q_{\bs j} \in \widehat{\mathcal{M}}$ of the parametric B\'ezier mesh from \eqref{eq: parametric Bezier mesh} is mapped to an element $K_{\bs j} =\bs F(Q_{\bs j} )$ in the physical domain. We set
\begin{align*}
\mathcal{M}:= \left \{ K_{\bs j}  \subset \Omega : K_{\bs j}  = \bs F(Q_{\bs j}), \bs j \in \bs J \right\} \, ,
\end{align*}
which is commonly known as the physical B\'ezier mesh, or simply B\'ezier mesh. The meshes for the coarsest knot vector $\Xi^0$ will be denoted by $\widehat{\mathcal{M}_0}$ and $\mathcal{M}_0$.

\section{Isogeometric mesh grading for circular sectors}
\label{section: Isogeometric mesh grading}
In the following, we demonstrate explicitly how the constructions presented in Section \ref{section: preliminaries and notation} can be used to address the numerical solution of our model problem \eqref{eq: model problem} with isogeometric analysis. We begin by describing the single-patch polar-like NURBS parameterization of circular sectors and discuss the resulting spline spaces on the computational domain. These spaces go beyond the solution space of the weak problem \eqref{eq: weak formulation} due to the low regularity of the parameterization which is considered a variational crime. Finally, we propose a graded mesh refinement scheme to tackle the singularities of the eigenfunctions observed in Section \ref{section: regularity of the eigenfunctions} and explore the choice of the associated grading parameter.

Here, we concentrate again on the circular sector $\Omega$ with angle $\omega = 2\pi$, as illustrated in Figure \ref{fig: sketch of model domain b}, which represents a unit disk with a crack on the positive $x$-axis. Just as in Section \ref{subsection: regularity of eigenfunctions for omega = 2pi}, it serves as a prototype for all other circular sectors with smaller angles $\omega < 2\pi$, where the same procedure can be carried out in an analogous or even simplified manner.

\subsection{Single-patch parameterization and coarse meshes}
\label{subsection: Parameterization}

First, we construct the single-patch polar-like isogeometric parameterization $\bs F \colon \hat{\Omega} \to \Omega$ of the unit disk with crack. The approach can be described intuitively by the deformations shown in Figure \ref{fig: deformations} and has been employed similarly in \cite{OhKimJeong2014}. We combine the typical $9$-point NURBS discretization of circles with a linear component in radial direction. Various other possibilities for parameterizing circular sectors and disks have been studied in the literature \cite{Lu2009,TakacsJuettler2011,Takacs2012,NguyenKarciauskasPeters2014,HughesSangalli2017}. In particular, we use the open knot vectors
\[\Xi^0_{1} = \left\{0,0,1,1\right\} \quad  \text{ and } \quad \Xi^0_2 = \left\{0,0, 0,\frac14,\frac14, \frac12, \frac12, \frac34, \frac34,1,1,1 \right\} \] 
along with the control points $\bs c_{\bs i}$ and weights $w_{\bs i}$, $\bs i \in \bs I_0 = \{(i_1,i_2): 1 \leq i_1 \leq 2, 1\leq i_2 \leq 9\}$, as defined in Table \ref{tab: Control points and weights for unit disk with crack} and illustrated in Figure \ref{fig: control points}. Note that all the control points $\bs c_{(1,i_2)}$ for $i_2=1,\dots,9$ lie at the conical point $\bs S=(0,0)$ of the circular sector, causing the parametric edge $\{(0,\zeta_2): \zeta_2 \in [0,1]\}$ to collapse to this point in the physical domain. The corresponding edge is depicted in cyan in Figure \ref{fig: deformations} and will be referred to as the singular edge of $\widehat{\Omega}$ in this paper. We further direct the interested reader to \ref{sec: appendix explicit representation} for an explicit representation of the parameterization $\bs F$ and to \ref{sec: appendix comparison} for a comparison of $\bs F$ to the classic transformation between Cartesian and polar coordinates.

\begin{figure}[t]
	\begin{center}
		\begin{subfigure}{0.7\linewidth}
			\begin{center}
				\begin{tikzpicture}[decoration = {snake, pre length=3pt, post length=3pt,},thick, scale = 0.8]
				
				\fill[gray!30] (0,0) rectangle (3,3);
				
				\draw[orange] (0,3) -- (3, 3);
				\draw[cyan] (0,0) -- (0, 3);
				\draw[green] (0,0) -- (3, 0);
				\draw[blue] (3,0) -- (3, 3);
				\draw[dashed, line width=0.05mm] (0,0.75) -- (3, 0.75);
				\draw[dashed, line width=0.05mm] (0,1.5) -- (3, 1.5);
				\draw[dashed, line width=0.05mm] (0, 2.25) -- (3, 2.25);
				
				\node at (1.5,3.5) {\large $\hat{\Omega}$};
				\node[left,yshift=2] at(0,0) {\small $0$};
				\node[left] at (0, 0.75) {\small $\frac{1}{4}$};
				\node[left] at (0 ,1.5) {\small $\frac{1}{2}$};
				\node[left] at (0, 2.25) {\small $\frac{3}{4}$};
				\node[left] at(0, 3) {\small $1$};
				
				\node at (0,-0.3) {\small $0$};
				\node at (3,-0.3) {\small $1$};
				
				\draw[-stealth] (0.6, 0.15) -- (1.05,0.15) node[above, yshift = -1] {\small $\zeta_1$};
				\draw[-stealth] (0.6, 0.15) -- (0.6, 0.6) node[left, yshift = -4] {\small $\zeta_2$};
				
				\draw [-stealth, decorate](3.3,1.5) -- (4.2,1.5);
				
				\fill[gray!30] (4.5,1.5) -- (6,1.5) -- (6,3);
				
				\draw[orange] (4.5,1.5) -- (6, 3);
				\draw[blue] (6,1.5) -- (6, 3);
				\draw[green] (4.5,1.5) -- (6, 01.5);
				\fill[cyan] (4.5,1.5) circle (2pt);
				\draw[dashed, line width=0.05mm] (4.5,1.5) -- (6,1.875);
				\draw[dashed, line width=0.05mm] (4.5,1.5) -- (6,2.25);
				\draw[dashed, line width=0.05mm] (4.5,1.5) -- (6,2.625);
				
				\draw [-stealth, decorate](6.3,1.5) -- (7.2,1.5);
				
				\fill[gray!30] (8,1.5) -- (9.5,1.5) arc (0:90:1.5);
				
				\draw [->, very thick] (8,2) arc (90:360:0.5);
				
				\draw[orange] (8,1.5) -- (8, 3);
				\draw[green] (8,1.5) -- (9.5, 1.5);
				\draw[blue] (9.5,1.5) arc (0:90:1.5);
				\fill[cyan] (8,1.5) circle (2pt);
				\draw[dashed, line width=0.05mm] (8,1.5) [pin=330:a] {} -- ++ (22.5:1.5);
				\draw[dashed, line width=0.05mm] (8,1.5) [pin=330:a] {} -- ++ (45:1.5);
				\draw[dashed, line width=0.05mm] (8,1.5) [pin=330:a] {} -- ++ (67.5:1.5);
				
				\draw [-stealth, decorate] (9.8,1.5) -- (10.7,1.5);
				
				\fill[gray!30] (12.5,1.5) circle (1.5cm);
				
				\draw [blue] (12.5,1.5) circle (1.5cm);
				\draw[-,green] (12.5,1.52) -- (14,1.52);
				\draw[-,orange] (12.5,1.48) -- (14,1.48);		
				\fill [cyan] (12.5,1.5) circle (2pt);
				\draw[dashed, line width=0.05mm] (12.5,1.5) -- (12.5,3);
				\draw[dashed, line width=0.05mm] (11,1.5) -- (12.5,1.5);
				\draw[dashed, line width=0.05mm] (12.5,1.5) -- (12.5,0);
				
				\node at (12.5,3.5) {\large $\Omega$};
				\end{tikzpicture}
			\end{center}
			\vspace{-5mm}
			\subcaption{}
			\label{fig: deformations}
		\end{subfigure}
		\begin{subfigure}{0.29 \linewidth}
			\begin{center}
				\begin{tikzpicture}[scale = 0.8]
				\draw[blue, thick, fill = gray, fill opacity = 0.3] (1.5, 1.5) circle (1.5cm);
				\draw[orange,thick] (1.5,1.48) -- (3,1.48);
				\draw[green,thick] (1.5,1.52) -- (3,1.52);
				
				\draw[dashed] (1.5,1.5) -- (1.5,3);
				\draw[dashed] (1.5,1.5) -- (1.5,0);
				\draw[dashed] (1.5,1.5) -- (0,1.5);
				
				\fill[red] (3,1.5) circle (1.5pt) node[above right, black] {\small $\bs c_{(2,1)}$};
				\fill[red] (2.56066,2.56066) circle (1.5pt) node[above right, black] {\small $\bs c_{(2,2)}$} ;
				\fill[red] (1.5,3) circle (1.5pt) node[above, black] {\small $\bs c_{(2,3)}$};
				\fill[red] (0.43933,2.56066) circle (1.5pt)  node[above left, black] {\small $\bs c_{(2,4)}$};
				\fill[red] (0,1.5) circle (1.5pt) node[left,black] {\small $\bs c_{(2,5)}$};
				\fill[red] (0.43933,0.43933) circle (1.5pt) node[below left, black] {\small $\bs c_{(2,6)}$};
				\fill[red] (1.5,0) circle (1.5pt)  node[below, black] {\small $\bs c_{(2,7)}$};
				\fill[red] (2.56066,0.43933) circle (1.5pt) node[below right, black] {\small $\bs c_{(2,8)}$};
				\fill[red] (3,1.5) circle (1.5pt) node[below right, black] {\small $\bs c_{(2,9)}$};
				\fill[cyan] (1.5,1.5) circle (1.5pt);
				
				\node at (1.3,1.75) {\small $\bs c_{(1,i_2)}$ $\forall i_2$};
				\end{tikzpicture}
			\end{center}
			\vspace{-5mm}
			\subcaption{}
			\label{fig: control points}
		\end{subfigure}
	\end{center}
	\vspace{-5mm}
	\caption{(a) Deforming a unit square to a unit disk with crack. (b) Illustration of the control points $\bs c_{\bs i}$ for $\bs i \in \bs I_0 = \{\bs i = (i_1,i_2): 1 \leq i_1 \leq 2, 1\leq i_2 \leq 9\}$.}
	\label{fig: deformations and control points}
\end{figure}
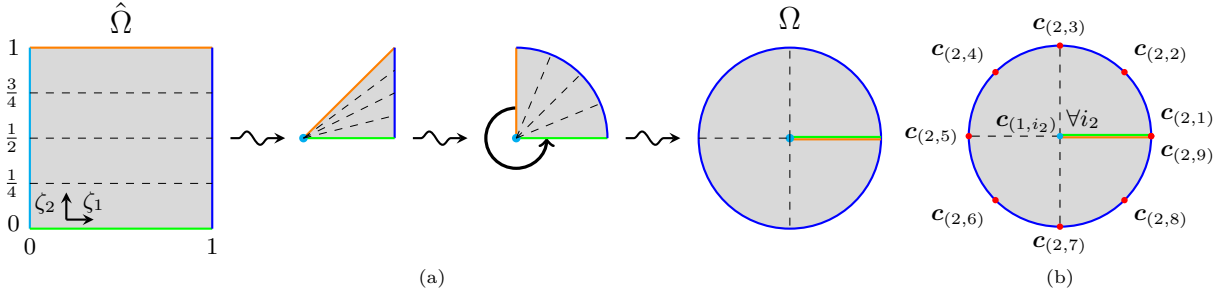

\begin{table}[t]
	\begin{center}
		\bgroup
		\def\arraystretch{1.1}%
		\setlength\tabcolsep{5pt}
		\begin{tabular}{ c | c c c c c c c c c } 
			$i_1$ & $1$ & $1$ & $1$ & $1$ & $1$ & $1$ & $1$ & $1$ & $1$ \\
			$i_2$ & $1$ & $2$ & $3$ & $4$ & $5$ & $6$ & $7$ & $8$ & $9$ \\ 
			$\bs c_{(i_1,i_2)}$ & $(0,0)$ & $(0,0)$ & $(0,0)$ & $(0,0)$ & $(0,0)$ & $(0,0)$ & $(0,0)$ & $(0,0)$ & $(0,0)$ \\
			$w_{(i_1,i_2)}$ & $1$ & $\frac{1}{\sqrt{2}}$ & $1$ & $\frac{1}{\sqrt{2}}$ & $1$ & $\frac{1}{\sqrt{2}}$ & $1$ & $\frac{1}{\sqrt{2}}$ & $1$ \\
			\newline \\
			\hline \\
			$i_1$ & $2$ & $2$ & $2$ & $2$ & $2$ & $2$ & $2$ & $2$ & $2$ \\ 
			$i_2$ & $1$ & $2$ & $3$ & $4$ & $5$ & $6$ & $7$ & $8$ & $9$ \\ 
			$\bs c_{(i_1,i_2)}$ & $(1,0)$ & $(\frac{1}{\sqrt{2}},\frac{1}{\sqrt{2}})$ & $(0,1)$ & $(-\frac{1}{\sqrt{2}},\frac{1}{\sqrt{2}})$ & $(-1,0)$ & $(-\frac{1}{\sqrt{2}},-\frac{1}{\sqrt{2}})$ & $(0,-1)$ & $(\frac{1}{\sqrt{2}},-\frac{1}{\sqrt{2}})$ & $(1,0)$ \\
			$w_{(i_1,i_2)}$ & $1$ & $\frac{1}{\sqrt{2}}$ & $1$ & $\frac{1}{\sqrt{2}}$ & $1$ & $\frac{1}{\sqrt{2}}$ & $1$ & $\frac{1}{\sqrt{2}}$ & $1$ 
		\end{tabular}
		\egroup
		\caption{Definition of control points $\bs c_{\bs i}$ and weights $w_{\bs i}$ for $\bs i \in \bs I_0 = \{\bs i = (i_1,i_2): 1 \leq i_1 \leq 2, 1\leq i_2 \leq 9\}$.}
		\label{tab: Control points and weights for unit disk with crack}
	\end{center}
\end{table}

The vectors of knots without repetitions,
\[
Z_1=\{0,1\} \quad  \text{ and } \quad Z_2 = \left \{0, \frac14, \frac12, \frac34, 1 \right \} \, ,
\]
form a Cartesian grid which defines the parametric B\'ezier mesh
\begin{align}
\label{eq: coarse parametric mesh}
\widehat{\mathcal{M}}_0 := \left\{Q_{(1,j_2)} =(0,1) \times \left(\frac{j_2-1}{4}, \frac{j_2}{4}\right) : 1 \leq j_2 \leq 4 \right\} \, .
\end{align}
With identity \eqref{eq: equality polar and NURBS coordinates} from \ref{sec: appendix comparison}, it follows that
\begin{align*}
K_{(1,j_2)} = \bs F(Q_{(1,j_2)}) = \left\{(r\cos\varphi, r\sin\varphi) \in \R^2: r \in (0,1), \varphi \in \left(\frac{j_2-1}{2} \, \pi , \frac{j_2}{2} \, \pi \right) \right \}
\end{align*}
for $j_2=1,\dots,4$. We obtain the physical B\'ezier mesh
\begin{align}
\label{eq: coarse physical mesh}
\mathcal{M}_0 := \left\{ K_{(1,j_2)} : 1 \leq j_2 \leq 4 \right\} \, .
\end{align}
Both meshes are illustrated in Figure \ref{fig: Coarse mesh}. Since they are constructed from the coarse knot vectors $\Xi_1^0$ and $\Xi_2^0$, we call them the coarse parametric and coarse physical B\'ezier meshes, respectively.

\begin{figure}
	\begin{center}
		\begin{tikzpicture}[baseline, scale = 0.65]
		
		\node at (2,4.5) {\large $\widehat{\mathcal{M}}_0 $};
		
		\fill[gray!30] (0,0) rectangle (4,4);
		
		\draw[-stealth] (4,0) -- (5, 0)  node[below, xshift = -2] {$\zeta_1$};
		\draw[-stealth] (0,4) -- (0, 5) node[left, yshift = -2] {$\zeta_2$};

		\draw[red] (0,0) -- (4, 0);
		\draw[red] (0,4) -- (4, 4);
		
		\draw[cyan] (0,0) -- (0, 4);
		\draw[blue] (4,0) -- (4, 4);
		
		\foreach \i in {1,...,3}{
			\draw[dashed] (0,\i) -- (4,\i);
		}
		
		\foreach \i in {1,...,4}{
			\node at (2,\i-0.5) {$Q_{(1,\i)}$};
		}
		
		\node[left, yshift = 2] at (0,0) {$0$};
		\node[left] at (0,1) {$\frac14$};
		\node[left] at (0,2) {$\frac24$};
		\node[left] at (0,3) {$\frac34$};
		\node[left] at (0,4) {$1$};
		
		\node[below, xshift = 2] at (0,0) {$0$};
		\node[below] at (4,0) {$1$};
		
		\draw [->, very thick] (6.3,2) to [out=30,in=150] (8.3,2) node[above] at (7.3,2.5) {\large $ \bs F$};
		
		%	\draw [->, very thick] (8.3,1.5) to [out=210,in=330] (6.3,1.5) node[below] at (7.3,1) {$F^{-1}$};
		\end{tikzpicture}
		\hspace*{12mm}
		\begin{tikzpicture}[baseline=-1.25cm, scale =0.725]
		
		\node at (0,2.5) {\large $\mathcal{M}_0$};
		
		\fill[gray!30] (0,0) circle (2);
		
		\foreach \i in {0,1,...,4}{
			\draw[blue] (\i*90:2) arc (\i*90:\i+90 + 90:2);
		}
		
		\foreach \i in {1,...,3}{
			\draw[dashed] (0:0) -- (\i*90:2);
		}
		\draw[red] (0,0) -- (2, 0);
		
		\foreach \i in {1,...,4}{
			\node at (\i*90-45:1) {$K_{(1,\i)}$};
		}	
		\fill[cyan] (0,0) circle (2pt);
		
		\end{tikzpicture}
		\caption{Illustration of the coarse parametric B\'ezier mesh $\widehat{\mathcal{M}}_0 = \left \{Q_{(1,j_2)}: j_2 = 1,2,3,4 \right \} $ and the physical B\'ezier mesh $\mathcal{M}_0 = \left \{K_{(1,j_2)} = \bs F(Q_{(1,j_2)}):j_2  = 1,2,3,4 \right \}$.}
		\label{fig: Coarse mesh}
	\end{center}
\end{figure}
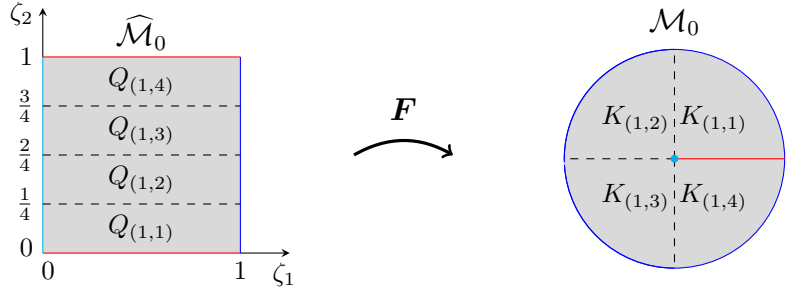

\subsection{Isogeometric approximation spaces}
\label{subsection: Approximation spaces in the physical domain}	

To establish a comprehensive framework for the numerical solution of model problem \eqref{eq: model problem}, we proceed to describe the isogeometric approximation spaces. Adopting an isoparametric approach, we use the discrete function space that defines the parameterization $\bs F$ also for the approximation space.

We begin with the bivariate spaces in the parametric domain, as defined in Section \ref{sec: Definition multivariate splines}. The tensor product spline space is given by
\begin{align*}
S_{\bs p^0} (\bs \Xi^0 ) 
=  \text{span} \left \{\widehat{B}_{\bs i, \bs p^0} (\bs \zeta), \bs i \in \bs I_0 \right \} \, , 
\end{align*}
where $\bs \Xi^0 = \Xi_1^0 \times \Xi_2^0$ and $\bs p^0 = (p_1^0,p_2^0)=(1,2)$. Likewise, the bivariate NURBS space reads
\begin{align*}
N_{\bs p^0}(\bs \Xi^0, W) = \text{span} \left \{\widehat{N}_{\bs i, \bs p^0} (\bs \zeta) , \bs i \in \bs I_0 \right \} \, .
\end{align*}
Further, we can apply the refinement algorithms from Section \ref{sec: refinement procedures} to obtain finer spline and NURBS spaces $S_{\bs p} (\bs \Xi )$ and $ N_{\bs p}(\bs \Xi, W)$ in the parametric domain which satisfy the relations \eqref{eq: refined spline spaces} and \eqref{eq: refined NURBS spaces}.

In standard literature \cite{BazilevsBeiraoDaVeigaCottrellHughesSangalli2006,BeiraodaVeigaChoSangalli2012,BeiraodaVeigaBuffaSangalliVazquez2014}, the space of NURBS in the physical domain is defined by the $\bs F$-image of the parametric spaces. More precisely, if $
\widehat{V}_h = N_{\bs p}(\bs \Xi, W) = \text{span} \{\widehat{N}_{\bs i, \bs p} (\bs \zeta) , \bs i \in \bs I \} $
is a refinement of the coarse parametric NURBS space $N_{\bs p^0}(\bs \Xi^0, W)$, the approximation space in the physical domain is defined by
\begin{align*}
V_h = \left \{f \circ \bs F^{-1} : f \in \widehat{V}_h \right \} \, .
\end{align*}
A basis of this space can be provided by a push-forward of the parametric basis, 
\begin{align}
\label{eq: basis functions regular case}
V_h = \text{span} \left \{N_{\bs i, \bs p} (\bs x) := \widehat{N}_{\bs i, \bs p} \circ \bs F^{-1}(\bs x) , \bs i \in \bs I \right \} \, .
\end{align}
However, this approach requires certain regularity assumptions on the parameterization, for more details see \cite{BeiraoDaVeigaBuffaRivasSangalli2011}. These requirements are not fulfilled by the polar-like mapping that we constructed in Section \ref{subsection: Parameterization}. Hence, using the approximation space \eqref{eq: basis functions regular case} leads to a variational crime which will be investigated in detail in the next section.

\subsection{Variational crimes}
\label{subsec: variational crime}	
It has been shown in \cite{TakacsJuettler2011,TakacsJuettler2012} that the lack of regularity in the isogeometric parameterization results in some of the basis functions \eqref{eq: basis functions regular case} not belonging to $H^1(\Omega)$, which implies $V_h \not \subset H^1(\Omega)$. As outlined in Section \ref{subsec: numerical solution process}, this constitutes a variational crime since the Galerkin principle requires the discrete space to be a subset of the solution space, i.e., $V_h \subset V = H^1(\Omega)$. However, the numerical results presented in Section \ref{section: Numerical results} are fully satisfying. Therefore, we intentionally do not take any steps to repair this variational crime and focus on demonstrating the numerical robustness of IGA with respect to the singular parameterization of circular sectors. Besides, we remark that no theoretical error estimates are developed in this contribution, but our results can still be validated as we consider a model problem with known exact solutions.

Nevertheless, it needs to be discussed why a reasonable computational method is achieved despite committing a variational crime. The matrix $A$ in the eigenvalue problem \eqref{eq: matrix eigenvalue problem} is not well-defined, since for all basis functions $N_{\bs i, \bs p}, N_{\bs j, \bs p} \in V_h \setminus V$, it holds
\begin{align}
\label{eq: full integral stiffness matrix}
a (N_{\bs i, \bs p}, N_{\bs j, \bs p}) = \int_\Omega \nabla N_{\bs i, \bs p} \nabla N_{\bs j, \bs p} \, \mathrm{d} \bs x = \infty \, .
\end{align}
However, in the computing package GeoPDEs \cite{Vazquez2016,deFalcoRealiVazquez2011} and other standard IGA or FEM software, the values of the bilinear form are not computed by exact integration but by numerical approximation with a suitable quadrature rule. Therefore, we assume that on each B\'ezier element $K_{\bs k} \in \mathcal{M}$, a quadrature rule is defined by a set of $n_{\bs k}$ quadrature points and weights,
\[ \bs x_{l,{\bs k}} \in K_{\bs k} \quad \text{and} \quad  w_{l,{\bs k}} \in \R\, , \quad  l=1, 2, \dots, n_{\bs k}\, ,\bs k \in \bs J \, .\]
An approximated bilinear form is then used to generate the entries of the stiffness matrix,
\begin{align}
\label{eq: approximate stiffness matrix entries}
\widetilde{a}(N_{\bs i,\bs p},  N_{\bs j,\bs p}) = \sum_{\bs k \in \bs J}\sum_{l=1}^{n_{\bs k}} \nabla N_{\bs i,\bs p} (\bs x_{l,\bs k}) \, \nabla  N_{\bs j,\bs p} (\bs x_{l,\bs k}) \, w_{l,\bs k} < \infty \quad \text{for all } \bs i, \bs j \in \bs I\, .
\end{align}
Thus, only the values of $ \nabla N_{\bs i,\bs p}$ are computed at a limited number of points and not the full integral \eqref{eq: full integral stiffness matrix}. Analogously, an approximate bilinear form $\widetilde{m}$ replaces the integral-based form $m$. Hence, instead of system \eqref{eq: matrix eigenvalue problem}, the matrix eigenvalue problem
\begin{align*}
\widetilde{A} \, \vec{u}_h  = \lambda_h \, \widetilde{M} \, \, \vec{u}_h 
\end{align*}
is solved, where
\begin{align*}
\widetilde{A} = \left(\widetilde{a}(N_{\bs i,\bs p}, N_{\bs j,\bs p})\right)_{\bs i, \bs j \in \bs I} \quad \text{ and } \quad
\widetilde{M}=  \left(\widetilde{m}(N_{\bs i,\bs p}, N_{\bs j,\bs p})\right)_{\bs i, \bs j \in \bs I} \, .
\end{align*}
However, the use of quadrature integrals is a very common variational crime in FEM or IGA. Hence, in some sense, the commonly committed variational crime hides the specific crime $V_h \not \subset V$, which is caused by the singular parameterization here.

\begin{remark}
	\label{remark: idea to treat singular basis functions}
	As mentioned in Section \ref{subsection: Parameterization}, there exist several other ways to parameterize circular sectors than using the polar-like mapping $\bs F$. In particular, a multi-patch approach based on biquadratic NURBS can be employed, which has been proposed in the numerical examples of \cite{LangerMantzaflarisMooreToulopoulos2015}. The multi-patch parameterization contains another type of singularities which are located on the circular boundary \cite{TakacsJuettler2011}. Hence, the singularity of the Laplace eigenfunctions does not coincide with the singularities of the isogeometric mapping, which may simplify the problem. However, we prefer the presented single-patch approach for the following reasons: 
	\begin{itemize}
		\item As indicated in Section \ref{sec: introduction}, previous results about spectral approximation properties of IGA in the literature have been shown for single-patch domains. 
		\item A polar-like discretization seems to be more natural and intuitive when considering circular sectors in an isogeometric context.
		\item The multi-patch approach leads to redundant control points at the patch junctions which are absent when using only one singularly mapped patch.
		\item Introducing a numerical framework becomes more complicated when multiple patches are used instead of a single one. Besides, standard approximation theory for IGA can not be applied. Even though the singularities in the multi-patch parameterization of circular sectors are weaker than the singularity in our single-patch mapping $\bs F$, the necessary regularity assumptions, as discussed in \cite{BazilevsBeiraoDaVeigaCottrellHughesSangalli2006,BeiraodaVeigaBuffaSangalliVazquez2014}, are not fulfilled. Up to our knowledge, the multi-patch discretization of circular sectors has only been considered numerically so far, but not in the error analysis \cite{LangerMantzaflarisMooreToulopoulos2015}.
	\end{itemize}
\end{remark}

\subsection{Single-patch mesh grading}
\label{subsection: Construction and properties of graded mesh}
After setting up the isogeometric framework for the numerical solution of our model problem \eqref{eq: model problem}, we now propose a modified refinement algorithm to tackle the singularities of the Laplace eigenfunctions discussed in Section \ref{section: regularity of the eigenfunctions}.

\subsubsection{Graded mesh refinement scheme}
\label{subsec: Mesh refinement scheme}
Isogeometric $h$-refinement is achieved through knot insertion, recall Section \ref{sec: refinement procedures}. In standard IGA, the univariate knot vectors are refined uniformly, with newly inserted knots being equidistantly distributed. We propose a more flexible knot insertion method that allows local mesh refinement towards known singularities. The key concept is to abandon uniformity in favor of knots concentrated around singularities, creating what we refer to as \textit{graded $h$-refinement}. The idea of graded meshes is not new; it has been proven to be a powerful tool for local a priori refinement of finite elements towards corner singularities \cite{ApelSaendigWhiteman1996,ApelNicaise1998,Babuska1970,OganesjanRukhovets1968}. More recently, a similar concept has been proposed for IGA in a multi-patch context, but the isogeometric parameterization is assumed to be smooth, at least in the points towards which the mesh is locally refined \cite{LangerMantzaflarisMooreToulopoulos2015}. However, this requirement is not satisfied by the single-patch parameterization $\bs F$ that we employ in this paper.

We begin with a detailed description of the standard uniform refinement procedure for single-patch circular sectors and subsequently introduce the proposed modification to achieve graded $h$-refinement. As a starting point, we recall the coarse parametric and physical B\'ezier meshes \eqref{eq: coarse parametric mesh} and \eqref{eq: coarse physical mesh}, respectively, both of which are illustrated in Figure \ref{fig: Coarse mesh}. It is crucial to ensure that the parameterization $\bs F$ and the weight function $W$ remain unchanged during refinement, as outlined in the conditions \eqref{eq: refined spline spaces} and \eqref{eq: refined NURBS spaces}. This can be done in practice by bisecting each interval of the coarse mesh a certain number of times. For instance, let $J_1, J_2 \in \N$ such that $J_2/4 \in \N$ and define $h_1:=\frac{1}{J_1}$ and $h_2:=\frac{1}{J_2}$. Then, the knot vectors
\begin{align}
\label{eq: uniform refinement knot vector 1}
\Xi_1^{h_1} &= \left \{ 0,0, \frac{1}{J_1}, \frac{2}{J_1}, \dots, \frac{J_1-1}{J_1},1,1 \right \} \, , \\
\Xi_2^{h_2} &= \left \{0,0,0, \frac{1}{J_2}, \frac{2}{J_2}, \dots, \frac{J_2/4-1}{J_2}, \frac{J_2/4}{J_2}, \frac{J_2/4}{J_2}, \frac{J_2/4+1}{J_2}, \frac{J_2/4+2}{J_2}, \dots, \frac{J_2/2-1}{J_2},\frac{J_2/2}{J_2}, \frac{J_2/2}{J_2}, \dots, 1,1,1 \right \} \nonumber
\end{align}
are uniform refinements of the coarse knot vectors $\Xi_1^0$ and $\Xi_2^0$, respectively. The vectors of knots without repetitions are simply given by
\begin{align*}
Z_1^{h_1} &= \left \{0, \frac{1}{J_1}, \frac{2}{J_1}, \dots, \frac{J_1-1}{J_1},1 \right \} 
= \{ \zeta_{1,j_1} := (j_1-1) h_1 : 1 \leq j_1 \leq J_1+1\}\, , \\
Z_2^{h_2} &= \left \{0, \frac{1}{J_2}, \frac{2}{J_2}, \dots, \frac{J_2-1}{J_2},1 \right \} 
= \{ \zeta_{2,j_2} := (j_2-1) h_2 : 1 \leq j_2 \leq J_2+1\} \, .
\end{align*}
We obtain the refined parametric B\'ezier mesh 
\begin{align}
\label{eq: uniform refined parametric mesh}
\widehat{\mathcal{M}}
& = \left\{Q_{\bs j} \subset \hat{\Omega}: Q_{\bs j} = Q_{(j_1,j_2)} =\left(\zeta_{1,j_1}, \zeta_{1,j_1+1}\right) \times \left(\zeta_{2,j_2}, \zeta_{2,j_2+1}\right) ,\bs j \in \bs J \right\}
\end{align}
with $\bs J := \{ \bs j = (j_1,j_2): 1 \leq j_1 \leq J_1, 1 \leq j_2 \leq J_2\}$. Each element of the coarse mesh $\widehat{\mathcal{M}_0}$, given by \eqref{eq: coarse parametric mesh}, is uniformly split up $J_1$ times in $\zeta_1$-direction and $J_2/4$ times in $\zeta_2$-direction. The same effect holds for the refined physical B\'ezier mesh 
\begin{align*}
\mathcal{M}
& =  \{ K_{\bs j} \subset \Omega: K_{\bs j} = \bs F(Q_{\bs j}) , \bs j \in \bs J\}\, .
\end{align*}
In total, the resulting meshes contain $J_1 \cdot J_2$ uniformly refined elements. An exemplary illustration of such parametric and physical meshes after uniform refinement with $J_1=6$ and $J_2 = 16$ is presented in Figure \ref{fig: refined meshes uniform}.

Now, our objective is to adjust the refinement process to achieve a physical B\'ezier mesh that is locally refined around the conical point of the circular sector. When the isogeometric mapping $\bs F$ is applied to the closure of the parametric domain, the singular edge $\{(0,\zeta_2) : \zeta_2 \in [0,1]\}$ is mapped onto the singular point $\bs S= (0,0)$. Hence, we need to grade the knots in the parametric domain towards the singular edge to obtain a finer mesh in the physical domain locally near the singularity. To this end, we introduce a \textit{grading parameter} $\mu \in (0,1]$ and modify the knot insertion process in $\zeta_1$-direction. Instead of the vector \eqref{eq: uniform refinement knot vector 1}, we use the graded knot vector
\begin{align}
\label{eq: knot vector graded mesh refinement}
\Xi_1^{h_1,\mu} = \{0,0, (h_1)^{\frac{1}{\mu}}, (2 h_1)^{\frac{1}{\mu}} , \dots, ((J_1-1)h_1)^{\frac{1}{\mu}},1,1\}
\end{align}
The graded vector of knots without repetitions is then given by
\begin{align}
\label{eq: graded vector of breakpoints}
Z_1^{h_1,\mu} = \{ \zeta_{1,j_1}^\mu := \left((j_1-1) h_1\right)^{\frac{1}{\mu}} : 1 \leq j_1 \leq J_1+1\} \, .
\end{align}
Note that the choice $\mu = 1$ leads to uniform refinement. Thus, the graded refinement scheme is a generalization of the uniform version. In $\zeta_2$-direction, no adjustments are needed, and we employ the same knot vector $\Xi_2^{h_2}$ and vector of breakpoints $Z_2^{h_2}$ as in the uniform refinement procedure. By combining the two directions, we obtain the graded parametric mesh
\begin{align*}
\widehat{\mathcal{M}}^\mu
= \left\{Q_{\bs j}^\mu \subset \hat{\Omega}: Q_{\bs j}^\mu =Q_{(j_1,j_2)}^\mu =\left(\zeta_{1,j_1}^\mu, \zeta_{1,j_1+1}^\mu\right) \times \left(\zeta_{2,j_2}, \zeta_{2,j_2+1}\right) , \bs j \in \bs J \right\} \, 
\end{align*}
which is locally refined towards the singular edge. The corresponding physical B\'ezier mesh 
\begin{align*}
\mathcal{M}^\mu
& = \left\{K_{\bs j}^\mu \subset \Omega : K_{\bs j}^\mu = \bs F (Q_{\bs j}^\mu), \bs j \in \bs J \right\}
\end{align*}
is locally refined towards the conical point of the circular sector. In Figure \ref{fig: refined meshes graded}, we illustrate this effect by depicting the graded parametric and physical B\'ezier meshes after refining with $J_1=6$ and $J_2=16$ and a grading parameter of $\mu=1/2$.

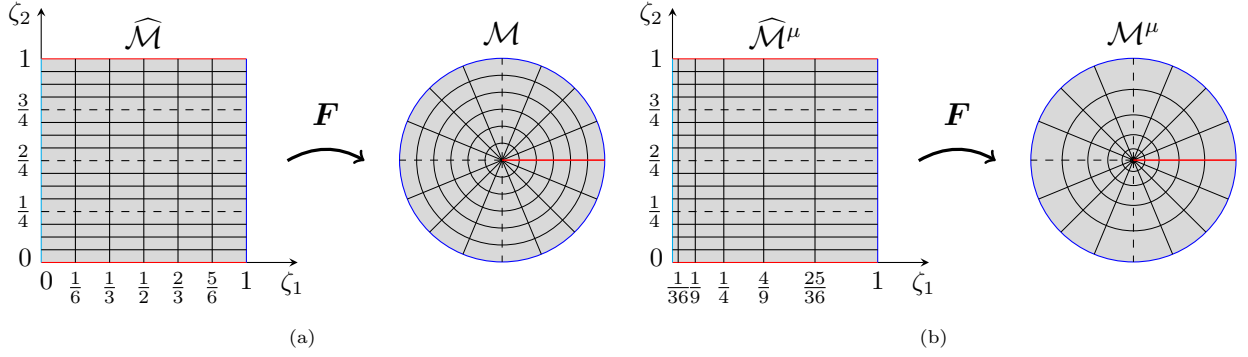
\begin{figure}[t]
	\begin{center}
		\begin{subfigure}{0.495\linewidth}
			\begin{center}
				\begin{tikzpicture}[baseline, scale = 0.675]			
				\def \N {5}			% Anzahl der Unterteilungen in xi-Richtung
				\def \hxi{1/6}		% Elementgröße in xi-Richtung
				
				\def \M {3} 		% Anzahl der Elemente in eta-Richtung
				\def \heta{1/4}		% Elementgröße in eta-Richtung
				
				\node at (2,4.5) {\large $\widehat{\mathcal{M}}$};
				
				\fill[gray!30] (0,0) rectangle (4,4);
				
				\draw[-stealth] (4,0) -- (5, 0)  node[below, xshift = -2] {$\zeta_1$};
				\draw[-stealth] (0,4) -- (0, 5) node[left, yshift = -2] {$\zeta_2$};
				
				\draw[red] (0,0) -- (4, 0);
				\draw[red] (0,4) -- (4, 4);
				
				\foreach \i in {1,...,\N}{
					\draw (\i*\hxi*4,0) -- (\i*\hxi*4,4);
				}
				
				\draw[cyan] (0,0) -- (0, 4);
				\draw[blue] (4,0) -- (4, 4);
				
				\foreach \i in {1,...,3}{
					\draw[dashed] (0, \i) -- (4, \i);
				}
				
				\foreach \j in {0,...,3}{
					\foreach \i in {1,...,\M}{
						\draw (0, \i*\heta + \j) -- (4, \i*\heta + \j);
					}
				}
				
				\node[left, yshift = 2] at (0,0) {$0$};
				\node[left] at (0,1) {$\frac14$};
				\node[left] at (0,2) {$\frac24$};
				\node[left] at (0,3) {$\frac34$};
				\node[left] at (0,4) {$1$};
				
				\node[below, xshift = 2] at (0,0) {$0$};
				\node[below] at (\hxi*1*4,0) {$\frac16$};
				\node[below] at (\hxi*2*4,0) {$\frac13$};
				\node[below] at (\hxi*3*4,0) {$\frac12$};
				\node[below] at (\hxi*4*4,0) {$\frac23$};
				\node[below] at (\hxi*5*4,0) {$\frac56$};
				\node[below] at (4,0) {$1$};
				
				\draw [->, very thick] (4.8,2) to [out=30,in=150] (6.3,2) node[above] at (5.55,2.5) {\large $\bs F$};
				\end{tikzpicture}
				\hspace*{2mm}
				\begin{tikzpicture}[scale = 1.35]		
				
				\def \N {5}			% Anzahl der Unterteilungen in xi-Richtung
				\def \hxi{1/6}		% Elementgröße in xi-Richtung
				
				\node at (0,1.25) {\large $\mathcal{M}$};
				
				\fill[gray!30] (0,0) circle (1);
				
				\foreach \i in {0,...,\N}{
					\draw [line width=0.05mm] (0:0) circle (\i*\hxi);
				}
				\draw [blue] (0:0) circle (1);
				
				% formula: \x = sin(\heta*j*pi/2)
				\foreach \x in {{0.382},0.707, 0.924}{
					\draw [line width=0.05mm] (0,0) -- (\x,{sqrt(1-\x*\x)});
					\draw [line width=0.05mm] (0,0) -- (-\x,{sqrt(1-\x*\x)});
					\draw [line width=0.05mm] (0,0) -- (-\x,-{sqrt(1-\x*\x)});
					\draw [line width=0.05mm] (0,0) -- (\x,-{sqrt(1-\x*\x)});
				}	
				
				\draw [line width=0.05mm, dashed] (0,0) -- (0,1);
				\draw [line width=0.05mm, dashed] (0,0) -- (-1,0);
				\draw [line width=0.05mm, dashed] (0,0) -- (0,-1);		
				\draw [-,red, line width=0.2mm] (0:0) -- (0:1);
				
				\end{tikzpicture}
				\subcaption{}
				\label{fig: refined meshes uniform}
			\end{center}
		\end{subfigure}
		\hfill		
		\begin{subfigure}{0.495\linewidth}
			\begin{center}		
				\begin{tikzpicture}[baseline, scale = 0.675]
				
				\def \N {5}			% Anzahl der Unterteilungen in xi-Richtung
				\def \hxi{1/6}		% Elementgröße in xi-Richtung
				
				\def \M {3} 		% Anzahl der Elemente in eta-Richtung
				\def \heta{1/4}		% Elementgröße in eta-Richtung
				
				\node at (2,4.5) {\large $\widehat{\mathcal{M}}^\mu$};
				
				\fill[gray!30] (0,0) rectangle (4,4);
				
				\draw[-stealth] (4,0) -- (5, 0)  node[below, xshift = -2] {$\zeta_1$};
				\draw[-stealth] (0,4) -- (0, 5) node[left, yshift = -2] {$\zeta_2$};

				\draw[red] (0,0) -- (4, 0);
				\draw[red] (0,4) -- (4, 4);
				
				\foreach \i in {1,...,\N}{
					\draw (\i*\i*\hxi*\hxi*4,0) -- (\i*\i*\hxi*\hxi*4,4);
				}
				
				\draw[cyan] (0,0) -- (0, 4);
				\draw[blue] (4,0) -- (4, 4);
				
				\foreach \i in {1,...,3}{
					\draw[dashed] (0, \i) -- (4, \i);
				}
				
				\foreach \j in {0,...,3}{
					\foreach \i in {1,...,\M}{
						\draw (0, \i*\heta + \j) -- (4, \i*\heta + \j);
					}
				}
				
				\node[left, yshift = 2] at (0,0) {$0$};
				\node[left] at (0,1) {$\frac14$};
				\node[left] at (0,2) {$\frac24$};
				\node[left] at (0,3) {$\frac34$};
				\node[left] at (0,4) {$1$};
				
				\node[below] at (\hxi*\hxi*1*1*4,0) {$\frac{1}{36}$};
				\node[below] at (\hxi*\hxi*2*2*4,0) {$\frac{1}{9}$};
				\node[below] at (\hxi*\hxi*3*3*4,0) {$\frac{1}{4}$};
				\node[below] at (\hxi*\hxi*4*4*4,0) {$\frac{4}{9}$};
				\node[below] at (\hxi*\hxi*5*5*4,0) {$\frac{25}{36}$};
				\node[below] at (4,0) {$1$};
				
				\draw [->, very thick] (4.8,2) to [out=30,in=150] (6.3,2) node[above] at (5.55,2.5) {\large $\bs F$};
				\end{tikzpicture}
				\hspace*{2mm}	
				\begin{tikzpicture}[scale = 1.35]
				
				\def \N {5}			% Anzahl der Unterteilungen in xi-Richtung
				\def \hxi{1/6}		% Elementgröße in xi-Richtung
				
				\node at (0,1.25) {\large $\mathcal{M}^\mu$};
				
				\fill[gray!30] (0,0) circle (1);
				
				\foreach \i in {0,...,\N}{
					\draw [line width=0.05mm] (0:0) circle (\i*\i*\hxi*\hxi);
				}
				\draw [blue] (0:0) circle (1);
				
				\foreach \x in {0.382,0.707, 0.924}{
					\draw [line width=0.05mm] (0,0) -- (\x,{sqrt(1-\x*\x)});
					\draw [line width=0.05mm] (0,0) -- (-\x,{sqrt(1-\x*\x)});
					\draw [line width=0.05mm] (0,0) -- (-\x,-{sqrt(1-\x*\x)});
					\draw [line width=0.05mm] (0,0) -- (\x,-{sqrt(1-\x*\x)});
				}	
				
				\draw [line width=0.05mm, dashed] (0,0) -- (0,1);
				\draw [line width=0.05mm, dashed] (0,0) -- (-1,0);
				\draw [line width=0.05mm, dashed] (0,0) -- (0,-1);		
				\draw [-,red,line width=0.2mm] (0:0) -- (0:1);
				
				\end{tikzpicture}
				\subcaption{}
				\label{fig: refined meshes graded}
			\end{center}
		\end{subfigure}	
		\vspace{-2mm}
		\caption{Parametric and physical B\'ezier meshes $\widehat{\mathcal{M}}$ and $\mathcal{M}$ after refining with $J_1=6$ and $J_2= 16$. Solid black lines are used for new subdivisions and dashed black lines for the initial coarse meshes $\widehat{\mathcal{M}_0}$ and $\mathcal{M}_0$. (a) Uniform refinement (b) Graded refinement with grading parameter $\mu =1/2$.}
		\label{fig: refined meshes}		
	\end{center}
\end{figure}

\subsubsection{Anisotropic elements in the graded meshes}
\label{subsec: anisotropic meshes}
Both the graded parametric and physical B\'ezier meshes contain anisotropic elements. The aspect ratio of the rectangular parametric elements $Q_{\bs j} \in \widehat{\mathcal{M}}^\mu$ close to the singular edge of $\widehat{\Omega}$ is given by
\begin{align}
\label{eq: aspect ratio parametric domain}
\frac{h_2}{h_1 ^{1/\mu}} = C h^{1-\frac{1}{\mu}} \to \infty \quad \text{ for } h_1=h_2=h \to 0  \text{ and } \mu <1 \, .
\end{align}
Thus, the elements are highly stretched in $\zeta_2$-direction. This effect is fully induced by the mesh grading and depends on the grading parameter $\mu$. For $\mu=1$ and $\frac{1}{C} h_2 \leq h_1 \leq C h_2$, that is, for evenly uniform refinement in both directions, the aspect ratio \eqref{eq: aspect ratio parametric domain} is constant and the mesh is isotropic.

In contrast, the latter does not hold for the physical B\'ezier mesh. We compute the aspect ratio of the elements $K_{\bs j} \in \mathcal{M}^\mu$ close to the singular point $\bs S$ using the lengths of their largest edges,
\begin{align}
\label{eq: aspect ratio physical domain}
\frac{h_1^{1/\mu}}{h_1^{1/\mu}h_2} = \frac{1}{h} \to \infty  \quad \text{ for } h_1=h_2=h \to 0 \, .
\end{align}
The elements are highly stretched in radial direction, corresponding to the $\zeta_1$-direction in the parametric domain. Hence, the stretching direction is reversed from the parametric to the physical mesh. Moreover, the physical mesh elements are anisotropic regardless of the mesh grading as the aspect ratio \eqref{eq: aspect ratio physical domain} does not depend on the grading parameter $\mu$. Consequently, the anisotropy in the physical mesh is exclusively induced by the singular parameterization $\bs F$. In Section \ref{subsec: hierarchical refinement}, we explore a slight modification of our refinement approach to avoid this effect.

\subsubsection{Choice of the grading parameter}
\label{subsec: choice of grading parameter}

It remains to be discussed how to choose the grading parameter $\mu$ such that optimal convergence orders for the singular solutions and corresponding eigenvalues are achieved. Here, we capitalize on existing knowledge about graded finite element meshes. While there is extensive literature on this topic, we only list a few important works here \cite{ApelNicaise1998,ApelSaendigWhiteman1996,ApelMelenk2017,Babuska1970,OganesjanRukhovets1968}.

According to the results from Section \ref{section: regularity of the eigenfunctions}, the strongest singularity that Laplace eigenfunctions of circular sectors can possibly have is of type $r^{\nu_1} $ with $\nu_1 = \frac{\pi}{\omega} \geq \frac12$. Therefore, if we consider an arbitrary eigenfunction $u = u_{\nu_k,m}$, where $k \in \N_0$ and $m \in \N$ are not known a priori, we can only expect that $u \in H^{s}(\Omega)$ for all $s<\nu_1+1$. If we restrict ourselves to classical Sobolev spaces, we can assume no more than $u \in H^1(\Omega)$ for circular sectors with angles $\omega > \pi$. Besides, it has been shown that the convergence of the discrete eigenvalues can be described in terms of the $H^1(\Omega)$-error of the corresponding Galerkin approximations of the eigenfunctions \cite{BabuskaAziz1972}. Hence, we incorporate a grading parameter that has been proven effective to achieve optimal convergence of the $H^1(\Omega)$-error, which is of order $p$ for finite elements of degree $p$. The required conditions are \cite{ApelMelenk2017}
\begin{align}
\label{eq: grading parameter simplified}
\mu < \frac{\nu_1}{p}  \quad \text{ and } \quad \mu \leq 1 \, .
\end{align}
We transfer these conditions to IGA for NURBS of degree $\bs p=(p_1,p_2)$ with $p =\min(p_1,p_2)$. However, there is still some freedom how to precisely choose $\mu$ for a numerical computation and several guidelines have been discussed in the literature, see for instance \cite{ApelMelenk2017}. Our experience has shown that
\begin{align}
\label{eq: grading parameter concrete}
\mu = 0.9 \cdot \frac{\nu_1}{p}
\end{align}
is a good choice, provided that $\frac{\nu_1}{p} < 1$.

In general, it is crucial to be aware of varying regularity properties of the Laplace eigenmodes of circular sectors. Eigenfunctions of low regularity are approximated more accurately by graded meshes due to the optimal convergence which is not achieved by uniform meshes. In contrast, smooth eigenfunctions are actually approximated better by uniform meshes. Unnecessary grading of the mesh has a negative effect on the approximation constant, although the optimal convergence rate is not affected. We will demonstrate this behavior later in the numerical examples. Of course, the choice of the grading parameter also depends on the degree of the basis functions and the objective of the computation. If the goal is to approximate a specific eigenfunction $u = u_{\nu_k,m}$ for a given $k \in \N_0$ and $m \in \N$, it is better to choose 
\begin{align}
\mu = \begin{cases}
1 \quad & \text{ if } \nu_k \in \N_0 \text{ or } \nu_k \geq p  \, , \\
0.9 \cdot \displaystyle \frac{\nu_k}{p} \quad & \text{ else}\, ,
\end{cases}
\label{eq: choice of grading parameter for known eigenfunction}
\end{align}
where we recall the regularity results of Lemma \ref{lemma: regularity of the eigenfunctions} and Lemma \ref{lemma: smooth eigenfunctions}. If multiple eigenfunctions have to be approximated at once with a guaranteed overall accuracy, the strongest grading parameter \eqref{eq: grading parameter concrete} should be chosen. In contrast, if poor approximations of the few eigenfunctions of very low regularity are not considered problematic, it may be advisable to use uniform meshes. To gain a better understanding, all the described effects of mesh grading will be illustrated numerically in the next section.

\section{Numerical results}
\label{section: Numerical results}
Finally, we perform some numerical tests showing the efficiency of our proposed approach to approximate the Laplace eigenvalue problem on circular sectors \eqref{eq: model problem}. First, we illustrate optimal convergence orders for the eigenpairs with respect to graded $h$-refinement for appropriately chosen grading parameters. In a second experiment, we demonstrate the advantages of our method for the computation of multiple eigenvalues, where we further examine spectral approximation properties of maximally smooth spline spaces on circular sectors. Finally, we combine our method with a hierarchical approach to avoid anisotropic elements in the physical B\'ezier mesh and save unnecessary degrees of freedom. In accordance with the main parts of the paper, we provide numerical results for the most complex circular sector with an inner angle of $2\pi$, which serves as a prototype for smaller angles.

All numerical experiments are carried out using the Matlab computing package GeoPDEs 3.2.2 \cite{Vazquez2016,deFalcoRealiVazquez2011} where our algorithm can be implemented easily by adapting one line of the standard knot refining routine \textit{kntrefine} that is provided in the package. 

\subsection{Approximation error quantities}
Before showing our computational results, we explain the quantities that are used to measure the error of the eigenfunctions and eigenvalues. Let $u=u_{\nu_k,m}$ be an eigenfunction for some $k \in \N_0$ and $m\in\N$ with corresponding eigenvalue $\lambda = \lambda_{\nu_k,m}$ and let $u_h$ and $\lambda_h$ be their approximations using the approach presented in this paper. The absolute and relative eigenvalue errors are given by $|\lambda- \lambda_h|$ and $|\lambda- \lambda_h|/ \lambda$, respectively. However, the eigenfunction error needs to be discussed in more detail. As outlined in Section \ref{subsec: variational crime}, the approximated eigenfunctions $u_h\in V_h$ are not in $H^1(\Omega)$ and thus the $H^1(\Omega)$-error $\norm{u-u_h}_{H^1(\Omega)}$ is not well defined. Instead, we work with quadrature based error quantities, which are commonly used in standard IGA or FEM software. Let $\bs x_{l,{\bs k}} \in K_{\bs k}$ and $w_{l,{\bs k}} \in \R$, $l=1, 2, \dots, n_{\bs k}\, ,\bs k \in \bs J$, be the quadrature points and weights for each element $K_{\bs k} \in \mathcal{M}$, as introduced in Section \ref{subsec: variational crime}. Then, for any $v\in V_h \cup H^1(\Omega)$, we define
\begin{align}
\label{eq: l2-norm quadrature}
\norm{v}_{L^2_h(\Omega)} := \left(\sum_{\bs k \in \bs J}\sum_{l=1}^{n_{\bs k}} v (\bs x_{l,\bs k})^2 \, w_{l,\bs k} \right)^{1/2}
\end{align}
and 
\begin{align}
\label{eq: h1-norm quadrature}
\norm{v}_{H^1_h(\Omega)} := \left( \abs{v}_{L^2_h(\Omega)} ^2 + \left(\sum_{\bs k \in \bs J}\sum_{l=1}^{n_{\bs k}} \nabla v (\bs x_{l,\bs k})^2 \, w_{l,\bs k} \right) \right) ^{1/2} \, .
\end{align}
GeoPDEs uses Gauss-Legendre quadrature for both parametric directions and combines them with a tensor product approach. Therefore, the quadrature weights $w_{l,{\bs k}} \in \R$, $l=1, 2, \dots, n_{\bs k}$, are non-negative \cite[Chapter 25.4]{AbramowitzStegun1988} and the expressions \eqref{eq: l2-norm quadrature} and \eqref{eq: h1-norm quadrature} define seminorms in $L^2(\Omega)$ and $H^1(\Omega)$, respectively. In the following numerical tests, the error quantities
\begin{align*}
\norm{u-u_h}_{L^2_h(\Omega)} \quad \text{ and } \quad \norm{u-u_h}_{H^1_h(\Omega)}
\end{align*}
will be called $L^2_h(\Omega)$- and $H^1_h(\Omega)$-error of the approximated eigenfunctions $u_h$, respectively. All the computations of this section are carried out using $n_{\bs k} = 36$ quadrature points for all $\bs k \in \bs J$.

\subsection{Optimal convergence to the analytical eigenfunctions and eigenvalues}
\label{subsec: optimal convergence}
In the first examples, we show that our method produces optimal convergence rates for the approximation of the Laplace eigenpairs of circulars sectors. As discussed in Section \ref{subsection: regularity of eigenfunctions for omega = 2pi}, the sequence of eigenmodes consists of both singular and smooth functions which are referred to as eigenfunctions of Type (A) and Type (B), respectively. We exemplarily pick one eigenfunction out of each group and evaluate the resulting approximation errors.

\begin{example}[Eigenfunction of Type (A)]
	\label{ex: approximation of singular eigenfunction}
	We approximate the eigenvalue $\lambda_{\nu_1,1}$ of the unit disk with crack and its corresponding eigenfunction $u_{\nu_1,1}$, which has the strongest possible singularity of type $r^{\nu_1}$, where $\nu_1=1/2$. In Figure \ref{fig: approximated second eigenfunction p=(1,2) with variational crime}, we illustrate the discrete eigenfunction throughout the gradual refinement process using NURBS of the lowest possible degree $\bs p=(1,2)$ and regularity $\bs k = (0,1)$. The solution is showcased for various graded mesh refinements, ranging from relatively coarse to further refined ones. The variational crime outlined in Section \ref{subsec: variational crime} can be observed in the coarse solutions as multiple function values are plotted at the conical point $\bs S=(0,0)$. However, these values all converge to zero with further refinement such that the finer solutions visually look correct. Figure \ref{fig: approximation errors for second eigenfunction variational crime} presents the $H^1_h(\Omega)$- and $L^2_h(\Omega)$-error of the eigenfunction and the absolute eigenvalue error using NURBS of polynomial degree $\bs p =(p,p)$, $p\in\{2,3,4\}$, and regularity $ \bs k=(p-1,p-1)$ with varying grading parameters $\mu$. Precisely, we set $\mu=1$ for uniform and $\mu = 0.9 \cdot \nu_1 / p$ for graded refinement, as proposed in formula \eqref{eq: grading parameter concrete}. On uniform meshes, the convergence rate with respect to the $H^1_h(\Omega)$- and $L^2_h(\Omega)$-error of the eigenfunction is $1/2$ and $1$, respectively, for all $\bs p$. In contrast, graded meshes yield optimal convergence orders $p$ and $p+1$, respectively. Consequently, the discrete eigenvalue converges to the exact one with order $1$ for any $\bs p$ during uniform refinement, whereas mesh grading recovers the optimal convergence rate of $2p$. 
	
	\begin{figure}
		\includegraphics[width=0.245\linewidth, trim=3cm 2.5cm 3cm 3cm, clip]{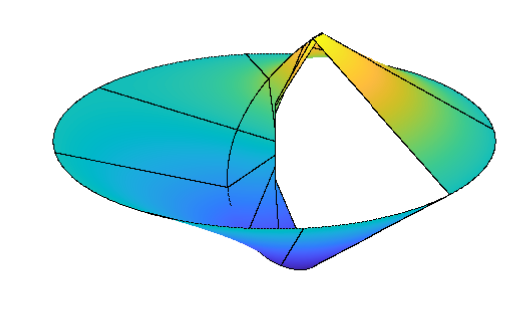}
		\includegraphics[width=0.245\linewidth, trim=3cm 2.5cm 3cm 3cm, clip]{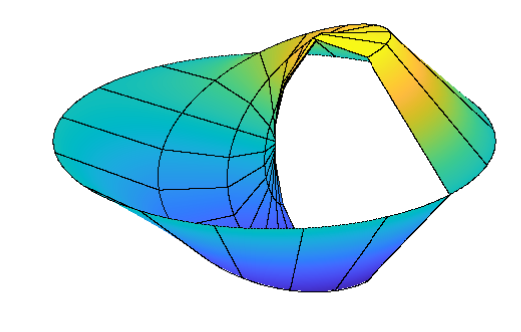}
		\includegraphics[width=0.245\linewidth, trim=3cm 2.5cm 3cm 3cm, clip]{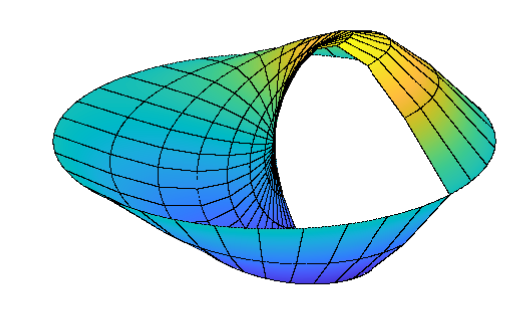}
		\includegraphics[width=0.245\linewidth, trim=3cm 2.5cm 3cm 3cm, clip]{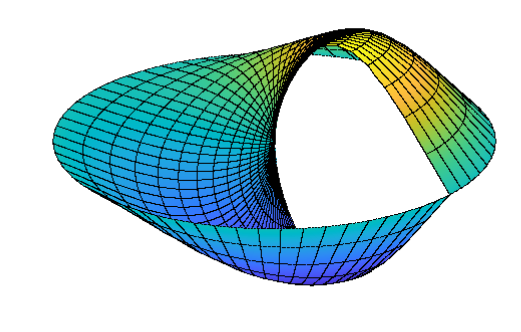}
		
		\caption{Approximation of the eigenfunction $u_{\nu_1,1}$ using NURBS of lowest possible degree $\bs p =(p_1,p_2) = (1,2)$ and regularity $\bs k = (0,1)$ while gradually refining $J_1=\frac{J_2}{4}=2,4,8,16$ times with grading parameter $\mu = 0.9 \cdot \frac{\nu_1}{p_1} = 0.45 $, recall \eqref{eq: grading parameter simplified}.}
		\label{fig: approximated second eigenfunction p=(1,2) with variational crime}
	\end{figure}
	
	\begin{figure}
		\resizebox{0.325\linewidth}{!}{% This file was created by matlab2tikz.
%
%The latest updates can be retrieved from
%  http://www.mathworks.com/matlabcentral/fileexchange/22022-matlab2tikz-matlab2tikz
%where you can also make suggestions and rate matlab2tikz.
%
\definecolor{mycolor1}{rgb}{0.00000,0.44700,0.74100}%
\definecolor{mycolor2}{rgb}{0.85000,0.32500,0.09800}%
\definecolor{mycolor3}{rgb}{0.92900,0.69400,0.12500}%
\definecolor{mycolor4}{rgb}{0.49400,0.18400,0.55600}%
\definecolor{mycolor5}{rgb}{0.46600,0.67400,0.18800}%
\definecolor{mycolor6}{rgb}{0.30100,0.74500,0.93300}%
\begin{tikzpicture}
\tikzstyle{every node}=[font=\Large]
\begin{axis}[%
width=4.521in,
height=3.566in,
at={(0.758in,0.481in)},
scale only axis,
xmode=log,
xmin=52,
xmax=69300,
xminorticks=true,
xlabel={ndof},
ymode=log,
ymin=1.54537479724358e-06,
ymax=0.759408641673929,
yminorticks=true,
ylabel={$||u-u_h||_{H^1_h(\Omega)}$},
axis background/.style={fill=white},
legend style={at={(0.03,0.03)}, anchor=south west, legend cell align=left, align=left, draw=white!15!black}
]
\addplot [color=mycolor1, line width=2.0pt, mark size=3.5pt, mark=square, mark options={solid, mycolor1}]
  table[row sep=crcr]{%
52	0.759408641673929\\
126	0.418408441906673\\
370	0.161190464295871\\
1242	0.0411407118390105\\
4522	0.0101208190548811\\
17226	0.00252345592204149\\
67210	0.000631554089597145\\
};
\addlegendentry{p=2, k=1, $\mu$=0.225}

\addplot [color=mycolor2, dotted, line width=2.0pt, mark size=3.3pt, mark=triangle, mark options={solid, mycolor2}]
  table[row sep=crcr]{%
52	0.324393354869122\\
126	0.208614434634026\\
370	0.148682416094953\\
1242	0.105657729932014\\
4522	0.0748311828716879\\
17226	0.0529377652603949\\
67210	0.0374372131662614\\
};
\addlegendentry{p=2, k=1, $\mu$=1}

\addplot [color=mycolor3, line width=2.0pt, mark size=3.5pt, mark=square, mark options={solid, mycolor3}]
  table[row sep=crcr]{%
85	0.184516750824273\\
175	0.078965430733476\\
451	0.0258736275994486\\
1387	0.00389999496165686\\
4795	0.000667575351627581\\
17755	8.313021856124e-05\\
68251	1.08574530062193e-05\\
};
\addlegendentry{p=3, k=2, $\mu$=0.15}

\addplot [color=mycolor4, dotted, line width=2.0pt, mark size=3.3pt, mark=triangle, mark options={solid, mycolor4}]
  table[row sep=crcr]{%
85	0.211154576294829\\
175	0.159492870773542\\
451	0.113145007365059\\
1387	0.0801080075474656\\
4795	0.0566649351371352\\
17755	0.0400720872849423\\
68251	0.0283360699345079\\
};
\addlegendentry{p=3, k=2, $\mu$=1}

\addplot [color=mycolor5, line width=2.0pt, mark size=3.5pt, mark=square, mark options={solid, mycolor5}]
  table[row sep=crcr]{%
126	0.229281678861689\\
232	0.0754342751360269\\
540	0.0200095072742913\\
1540	0.00335216807103452\\
5076	0.000252685487834078\\
18292	2.12849803913814e-05\\
69300	1.54537479724358e-06\\
};
\addlegendentry{p=4, k=3, $\mu$=0.1125}

\addplot [color=mycolor6, dotted, line width=2.0pt, mark size=3.3pt, mark=triangle, mark options={solid, mycolor6}]
  table[row sep=crcr]{%
126	0.185311637249168\\
232	0.134040818565191\\
540	0.0951088217395636\\
1540	0.0672848830138218\\
5076	0.0475819248954451\\
18292	0.0336459838833139\\
69300	0.0237912743643556\\
};
\addlegendentry{p=4, k=3, $\mu$=1}

\logLogSlopeTriangle{0.825}{0.15}{0.71}{0.25}{gray, line width = 2pt}{0.5};
\logLogSlopeTriangle{0.85}{0.1}{0.46}{01}{gray, line width = 2pt}{2};
\logLogSlopeTriangleUpsideDown{0.95}{0.1}{0.3}{1.5}{gray, line width = 2pt}{3};
\logLogSlopeTriangle{0.825}{0.1}{0.05}{2}{gray, line width = 2pt}{4};

\end{axis}
\end{tikzpicture}%
}
		\hspace*{1mm}
		\resizebox{0.325\linewidth}{!}{% This file was created by matlab2tikz.
%
%The latest updates can be retrieved from
%  http://www.mathworks.com/matlabcentral/fileexchange/22022-matlab2tikz-matlab2tikz
%where you can also make suggestions and rate matlab2tikz.
%
\definecolor{mycolor1}{rgb}{0.00000,0.44700,0.74100}%
\definecolor{mycolor2}{rgb}{0.85000,0.32500,0.09800}%
\definecolor{mycolor3}{rgb}{0.92900,0.69400,0.12500}%
\definecolor{mycolor4}{rgb}{0.49400,0.18400,0.55600}%
\definecolor{mycolor5}{rgb}{0.46600,0.67400,0.18800}%
\definecolor{mycolor6}{rgb}{0.30100,0.74500,0.93300}%
\begin{tikzpicture}
\tikzstyle{every node}=[font=\Large]
\begin{axis}[%
width=4.521in,
height=3.566in,
at={(0.758in,0.481in)},
scale only axis,
xmode=log,
xmin=52,
xmax=69300,
xminorticks=true,
xlabel={ndof},
ymode=log,
ymin=1.44013005681052e-08,
ymax=0.112968781416064,
yminorticks=true,
ylabel={$||u-u_h||_{L^2_h(\Omega)}$},
axis background/.style={fill=white},
legend style={at={(0.03,0.03)}, anchor=south west, legend cell align=left, align=left, draw=white!15!black}
]
\addplot [color=mycolor1, line width=2.0pt, mark size=3.5pt, mark=square, mark options={solid, mycolor1}]
  table[row sep=crcr]{%
52	0.112968781416064\\
126	0.0451132264577812\\
370	0.014134718623586\\
1242	0.00168028486335575\\
4522	0.000191929328794319\\
17226	2.32039461175333e-05\\
67210	2.87768285201812e-06\\
};
\addlegendentry{p=2, k=1, $\mu$=0.225}

\addplot [color=mycolor2, dotted, line width=2.0pt, mark size=3.3pt, mark=triangle, mark options={solid, mycolor2}]
  table[row sep=crcr]{%
52	0.023394908508514\\
126	0.00470290793147184\\
370	0.00164328820859685\\
1242	0.000575606688527693\\
4522	0.000203876336182562\\
17226	7.22181828736764e-05\\
67210	2.56140086824138e-05\\
};
\addlegendentry{p=2, k=1, $\mu$=1}

\addplot [color=mycolor3, line width=2.0pt, mark size=3.5pt, mark=square, mark options={solid, mycolor3}]
  table[row sep=crcr]{%
85	0.0113398409947163\\
175	0.00389316098322747\\
451	0.00249558136509603\\
1387	0.000225403322530688\\
4795	1.61808466738888e-05\\
17755	8.23353433266986e-07\\
68251	5.5341288062453e-08\\
};
\addlegendentry{p=3, k=2, $\mu$=0.15}

\addplot [color=mycolor4, dotted, line width=2.0pt, mark size=3.3pt, mark=triangle, mark options={solid, mycolor4}]
  table[row sep=crcr]{%
85	0.009422486132242\\
175	0.00353227283917371\\
451	0.00116755949243138\\
1387	0.000502589337139104\\
4795	0.000227070570877388\\
17755	0.000107053243432372\\
68251	5.1831463393525e-05\\
};
\addlegendentry{p=3, k=2, $\mu$=1}

\addplot [color=mycolor5, line width=2.0pt, mark size=3.5pt, mark=square, mark options={solid, mycolor5}]
  table[row sep=crcr]{%
126	0.0132974477328271\\
232	0.00496870886157153\\
540	0.000848442967250195\\
1540	0.000187426872371374\\
5076	6.35051170223276e-06\\
18292	3.78230714594596e-07\\
69300	1.44013005681052e-08\\
};
\addlegendentry{p=4, k=3, $\mu$=0.1125}

\addplot [color=mycolor6, dotted, line width=2.0pt, mark size=3.3pt, mark=triangle, mark options={solid, mycolor6}]
  table[row sep=crcr]{%
126	0.00573813110459382\\
232	0.00250651384405356\\
540	0.0011513684682419\\
1540	0.000554650180114532\\
5076	0.000267484611753565\\
18292	0.00013171970727905\\
69300	6.53546576454968e-05\\
};
\addlegendentry{p=4, k=3, $\mu$=1}

\logLogSlopeTriangleUpsideDown{0.95}{0.15}{0.61}{0.5}{gray, line width = 2pt}{1};
\logLogSlopeTriangle{0.825}{0.1}{0.365}{1.5}{gray, line width = 2pt}{3};
\logLogSlopeTriangleUpsideDown{0.95}{0.1}{0.24}{2}{gray, line width = 2pt}{4};
\logLogSlopeTriangle{0.825}{0.1}{0.06}{2.5}{gray, line width = 2pt}{5};

\end{axis}
\end{tikzpicture}%
}
		\hspace*{1mm}
		\resizebox{0.325\linewidth}{!}{% This file was created by matlab2tikz.
%
%The latest updates can be retrieved from
%  http://www.mathworks.com/matlabcentral/fileexchange/22022-matlab2tikz-matlab2tikz
%where you can also make suggestions and rate matlab2tikz.
%
\definecolor{mycolor1}{rgb}{0.00000,0.44700,0.74100}%
\definecolor{mycolor2}{rgb}{0.85000,0.32500,0.09800}%
\definecolor{mycolor3}{rgb}{0.92900,0.69400,0.12500}%
\definecolor{mycolor4}{rgb}{0.49400,0.18400,0.55600}%
\definecolor{mycolor5}{rgb}{0.46600,0.67400,0.18800}%
\definecolor{mycolor6}{rgb}{0.30100,0.74500,0.93300}%
\begin{tikzpicture}
\tikzstyle{every node}=[font=\Large]
\begin{axis}[%
width=4.521in,
height=3.566in,
at={(0.758in,0.481in)},
scale only axis,
xmode=log,
xmin=52,
xmax=69300,
xminorticks=true,
xlabel={ndof},
ymode=log,
ymin=3.5473846082823e-12,
ymax=0.521252513420217,
yminorticks=true,
ylabel={$|\lambda-\lambda_h|$},
axis background/.style={fill=white},
legend style={at={(0.03,0.03)}, anchor=south west, legend cell align=left, align=left, draw=white!15!black}
]
\addplot [color=mycolor1, line width=2.0pt, mark size=3.5pt, mark=square, mark options={solid, mycolor1}]
  table[row sep=crcr]{%
52	0.521252513420217\\
126	0.211702837884971\\
370	0.0350934711795041\\
1242	0.00240953844320657\\
4522	0.000147809199116722\\
17226	9.21679747101223e-06\\
67210	5.77764598830299e-07\\
};
\addlegendentry{p=2, k=1, $\mu$=0.225}

\addplot [color=mycolor2, dotted, line width=2.0pt, mark size=3.3pt, mark=triangle, mark options={solid, mycolor2}]
  table[row sep=crcr]{%
52	0.0243882659227701\\
126	0.00209221582068153\\
370	0.000902009621066568\\
1242	0.000340433680447561\\
4522	0.00014996219013419\\
17226	7.20064830819211e-05\\
67210	3.56021605636414e-05\\
};
\addlegendentry{p=2, k=1, $\mu$=1}

\addplot [color=mycolor3, line width=2.0pt, mark size=3.5pt, mark=square, mark options={solid, mycolor3}]
  table[row sep=crcr]{%
85	0.0334821995337471\\
175	0.00856336904388222\\
451	0.000873357462451096\\
1387	2.12223049480542e-05\\
4795	6.41444536597646e-07\\
17755	1.00015107307172e-08\\
68251	1.70780722896779e-10\\
};
\addlegendentry{p=3, k=2, $\mu$=0.15}

\addplot [color=mycolor4, dotted, line width=2.0pt, mark size=3.3pt, mark=triangle, mark options={solid, mycolor4}]
  table[row sep=crcr]{%
85	0.0460596642296718\\
175	0.0201821906271356\\
451	0.00999247448419815\\
1387	0.00498274528922416\\
4795	0.00248983123661262\\
17755	0.00124477923402821\\
68251	0.000622387209901376\\
};
\addlegendentry{p=3, k=2, $\mu$=1}

\addplot [color=mycolor5, line width=2.0pt, mark size=3.5pt, mark=square, mark options={solid, mycolor5}]
  table[row sep=crcr]{%
126	0.0441237780188093\\
232	0.00749798125813328\\
540	0.000566736310995353\\
1540	1.57322599765308e-05\\
5076	9.18891061019167e-08\\
18292	6.54280185585776e-10\\
69300	3.5473846082823e-12\\
};
\addlegendentry{p=4, k=3, $\mu$=0.1125}

\addplot [color=mycolor6, dotted, line width=2.0pt, mark size=3.3pt, mark=triangle, mark options={solid, mycolor6}]
  table[row sep=crcr]{%
126	0.0525242399140193\\
232	0.0258477264506212\\
540	0.0128919758569044\\
1540	0.00644543758420646\\
5076	0.00322305266021772\\
18292	0.00161166684911151\\
69300	0.000805875668858036\\
};
\addlegendentry{p=4, k=3, $\mu$=1}

\logLogSlopeTriangle{0.825}{0.15}{0.71}{0.5}{gray, line width = 2pt}{1};
\logLogSlopeTriangle{0.85}{0.1}{0.472}{2}{gray, line width = 2pt}{4};
\logLogSlopeTriangleUpsideDown{0.95}{0.1}{0.3}{3}{gray, line width = 2pt}{6};
\logLogSlopeTriangle{0.825}{0.1}{0.055}{4}{gray, line width = 2pt}{8};

\end{axis}
\end{tikzpicture}%
}
		
		\caption{Approximation errors for the eigenfunction $u=u_{\nu_1,1}$ and its corresponding eigenvalue $\lambda=\lambda_{\nu_1,1}$ using gradually and uniformly refined NURBS of degree $\bs p =(p,p)$, $p\in \{2,3,4\}$, and regularity $\bs k = (p-1,p-1)$. We set $\mu = 0.9 \frac{\nu_1}{p} =  \frac{0.9}{2p}$ for graded meshes and $\mu = 1$ corresponds to uniform refinement.}
		\label{fig: approximation errors for second eigenfunction variational crime}
	\end{figure}
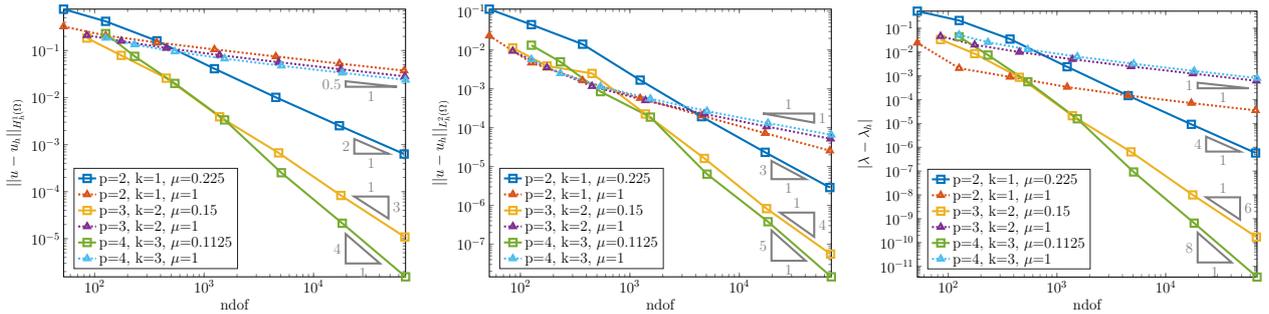
\end{example}

\begin{example}[Eigenfunction of Type (B)]
	\label{ex: approximation of smooth eigenfunction}
	Next, we approximate the eigenvalue $\lambda_{\nu_2,1}$ of the unit disk with crack and its corresponding eigenfunction $u=u_{\nu_2,1}$, where $\nu_2=1$. As shown in Lemma \ref{lemma: smooth eigenfunctions}, this is a smooth function. Figure \ref{fig: approximation errors for third eigenfunction variational crime} depicts the $H^1_h(\Omega)$- and $L^2_h(\Omega)$-error of the eigenfunction and the absolute eigenvalue error using NURBS of polynomial degree $\bs p =(p,p)$, $p\in\{2,3,4\}$, and regularity $\bs k= (p-1,p-1)$ with different grading parameters $\mu$. Here, the convergence orders with uniform and graded refinement are equal, but the approximation constant of the uniform meshes is superior. This result justifies the recommended grading parameter \eqref{eq: choice of grading parameter for known eigenfunction} if a specific eigenfunction is approximated.
	
	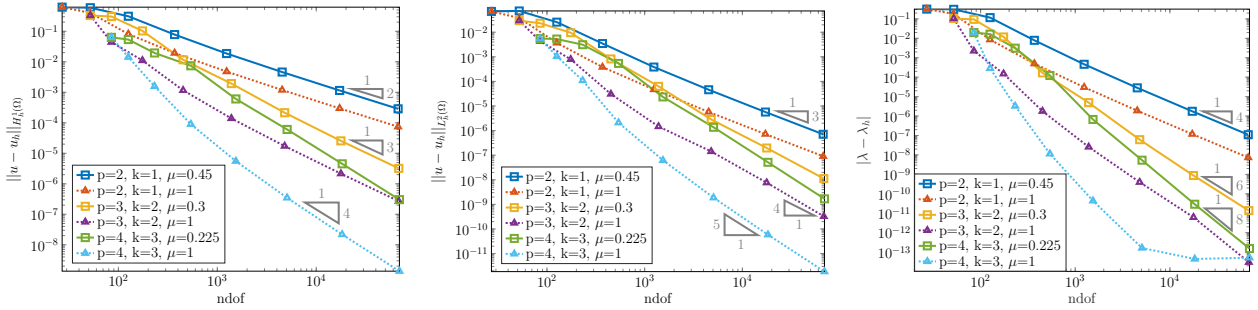
\begin{figure}[t]
		\resizebox{0.32\linewidth}{!}{% This file was created by matlab2tikz.
%
%The latest updates can be retrieved from
%  http://www.mathworks.com/matlabcentral/fileexchange/22022-matlab2tikz-matlab2tikz
%where you can also make suggestions and rate matlab2tikz.
%
\definecolor{mycolor1}{rgb}{0.00000,0.44700,0.74100}%
\definecolor{mycolor2}{rgb}{0.85000,0.32500,0.09800}%
\definecolor{mycolor3}{rgb}{0.92900,0.69400,0.12500}%
\definecolor{mycolor4}{rgb}{0.49400,0.18400,0.55600}%
\definecolor{mycolor5}{rgb}{0.46600,0.67400,0.18800}%
\definecolor{mycolor6}{rgb}{0.30100,0.74500,0.93300}%
\begin{tikzpicture}
\tikzstyle{every node}=[font=\Large]
\begin{axis}[%
width=4.568in,
height=3.603in,
at={(0.766in,0.486in)},
scale only axis,
xmode=log,
xmin=27,
xmax=69300,
xminorticks=true,
xlabel={ndof},
ymode=log,
ymin=1.36869504869903e-09,
ymax=0.619579240357889,
yminorticks=true,
ylabel={$||u-u_h||_{H^1_h(\Omega)}$},
axis background/.style={fill=white},
legend style={at={(0.03,0.03)}, anchor=south west, legend cell align=left, align=left, draw=white!15!black}
]
\addplot [color=mycolor1, line width=2.0pt, mark size=3.5pt, mark=square, mark options={solid, mycolor1}]
  table[row sep=crcr]{%
27	0.619579240357889\\
52	0.593406361632431\\
126	0.3118918250066\\
370	0.0782246002993857\\
1242	0.0186983766561087\\
4522	0.00461779625711514\\
17226	0.00115091932858985\\
67210	0.000287509556384843\\
};
\addlegendentry{p=2, k=1, $\mu$=0.45}

\addplot [color=mycolor2, dotted, line width=2.0pt, mark size=3.3pt, mark=triangle, mark options={solid, mycolor2}]
  table[row sep=crcr]{%
27	0.619579240357889\\
52	0.40208674426521\\
126	0.081684743053873\\
370	0.0193931151896345\\
1242	0.00478744481409925\\
4522	0.00119292415186803\\
17226	0.000297978429753435\\
67210	7.44783456613554e-05\\
};
\addlegendentry{p=2, k=1, $\mu$=1}

\addplot [color=mycolor3, line width=2.0pt, mark size=3.5pt, mark=square, mark options={solid, mycolor3}]
  table[row sep=crcr]{%
52	0.329928839875558\\
85	0.3059823755074\\
175	0.103059544152511\\
451	0.0116687574287159\\
1387	0.00193719072680267\\
4795	0.000214735681820283\\
17755	2.56445109784671e-05\\
68251	3.21034920436612e-06\\
};
\addlegendentry{p=3, k=2, $\mu$=0.3}

\addplot [color=mycolor4, dotted, line width=2.0pt, mark size=3.3pt, mark=triangle, mark options={solid, mycolor4}]
  table[row sep=crcr]{%
52	0.329928839875558\\
85	0.0441810804322156\\
175	0.0110742611036587\\
451	0.00116459103715946\\
1387	0.000139392361532326\\
4795	1.72734997115378e-05\\
17755	2.1558773051958e-06\\
68251	2.69526365024239e-07\\
};
\addlegendentry{p=3, k=2, $\mu$=1}

\addplot [color=mycolor5, line width=2.0pt, mark size=3.5pt, mark=square, mark options={solid, mycolor5}]
  table[row sep=crcr]{%
85	0.0633823339081802\\
126	0.0537437639049071\\
232	0.0195278652947162\\
540	0.00742491663452471\\
1540	0.000606562999893627\\
5076	6.05563669971165e-05\\
18292	4.52977325308729e-06\\
69300	3.02423322081481e-07\\
};
\addlegendentry{p=4, k=3, $\mu$=0.225}

\addplot [color=mycolor6, dotted, line width=2.0pt, mark size=3.3pt, mark=triangle, mark options={solid, mycolor6}]
  table[row sep=crcr]{%
85	0.0633823339081802\\
126	0.0142728705614051\\
232	0.00157750394781962\\
540	8.99425529648439e-05\\
1540	5.53625226893008e-06\\
5076	3.47308617810386e-07\\
18292	2.18216891122536e-08\\
69300	1.36869504869903e-09\\
};
\addlegendentry{p=4, k=3, $\mu$=1}

\logLogSlopeTriangleUpsideDown{0.95}{0.09}{0.69}{1}{gray, line width = 2pt}{2};
\logLogSlopeTriangleUpsideDown{0.95}{0.09}{0.495}{1.5}{gray, line width = 2pt}{3};
\logLogSlopeTriangleUpsideDown{0.82}{0.1}{0.26}{2}{gray, line width = 2pt}{4};

\end{axis}
\end{tikzpicture}%
}
		\hspace*{1mm}
		\resizebox{0.32\linewidth}{!}{% This file was created by matlab2tikz.
%
%The latest updates can be retrieved from
%  http://www.mathworks.com/matlabcentral/fileexchange/22022-matlab2tikz-matlab2tikz
%where you can also make suggestions and rate matlab2tikz.
%
\definecolor{mycolor1}{rgb}{0.00000,0.44700,0.74100}%
\definecolor{mycolor2}{rgb}{0.85000,0.32500,0.09800}%
\definecolor{mycolor3}{rgb}{0.92900,0.69400,0.12500}%
\definecolor{mycolor4}{rgb}{0.49400,0.18400,0.55600}%
\definecolor{mycolor5}{rgb}{0.46600,0.67400,0.18800}%
\definecolor{mycolor6}{rgb}{0.30100,0.74500,0.93300}%
\begin{tikzpicture}
\tikzstyle{every node}=[font=\Large]
\begin{axis}[%
width=4.568in,
height=3.603in,
at={(0.766in,0.486in)},
scale only axis,
xmode=log,
xmin=27,
xmax=69300,
xminorticks=true,
xlabel={ndof},
ymode=log,
ymin=1.8497194333741e-12,
ymax=0.0734312066413871,
yminorticks=true,
ylabel={$||u-u_h||_{L^2_h(\Omega)}$},
axis background/.style={fill=white},
legend style={at={(0.03,0.03)}, anchor=south west, legend cell align=left, align=left, draw=white!15!black}
]
\addplot [color=mycolor1, line width=2.0pt, mark size=3.5pt, mark=square, mark options={solid, mycolor1}]
  table[row sep=crcr]{%
27	0.0702831084890611\\
52	0.0734312066413871\\
126	0.0257531145863176\\
370	0.00347568804236712\\
1242	0.000384799304398632\\
4522	4.58868908644786e-05\\
17226	5.6927666824735e-06\\
67210	7.10480351286862e-07\\
};
\addlegendentry{p=2, k=1, $\mu$=0.45}

\addplot [color=mycolor2, dotted, line width=2.0pt, mark size=3.3pt, mark=triangle, mark options={solid, mycolor2}]
  table[row sep=crcr]{%
27	0.0702831084890611\\
52	0.0378354367804528\\
126	0.00369607854073533\\
370	0.000382671873296083\\
1242	4.70635100532554e-05\\
4522	5.76127396117307e-06\\
17226	7.21994517503527e-07\\
67210	9.06237082132988e-08\\
};
\addlegendentry{p=2, k=1, $\mu$=1}

\addplot [color=mycolor3, line width=2.0pt, mark size=3.5pt, mark=square, mark options={solid, mycolor3}]
  table[row sep=crcr]{%
52	0.0296725439480925\\
85	0.0232915069243082\\
175	0.00956480310289609\\
451	0.000823776249175645\\
1387	6.31190842701838e-05\\
4795	2.84570988859389e-06\\
17755	1.96691067263248e-07\\
68251	1.12175036275752e-08\\
};
\addlegendentry{p=3, k=2, $\mu$=0.3}

\addplot [color=mycolor4, dotted, line width=2.0pt, mark size=3.3pt, mark=triangle, mark options={solid, mycolor4}]
  table[row sep=crcr]{%
52	0.0296725439480925\\
85	0.00434164323930073\\
175	0.000791397978819673\\
451	3.07348454232327e-05\\
1387	1.46763047369369e-06\\
4795	1.41960786062165e-07\\
17755	7.6336991882219e-09\\
68251	3.2684542204552e-10\\
};
\addlegendentry{p=3, k=2, $\mu$=1}

\addplot [color=mycolor5, line width=2.0pt, mark size=3.5pt, mark=square, mark options={solid, mycolor5}]
  table[row sep=crcr]{%
85	0.00555264680203518\\
126	0.00528779477460469\\
232	0.00311185819807807\\
540	0.000539547018733666\\
1540	2.33542884488691e-05\\
5076	1.36289672675293e-06\\
18292	5.14210759118573e-08\\
69300	1.68794180593417e-09\\
};
\addlegendentry{p=4, k=3, $\mu$=0.225}

\addplot [color=mycolor6, dotted, line width=2.0pt, mark size=3.3pt, mark=triangle, mark options={solid, mycolor6}]
  table[row sep=crcr]{%
85	0.00555264680203518\\
126	0.00105234918596993\\
232	0.000108533149651327\\
540	2.09919052100588e-06\\
1540	6.21250300758878e-08\\
5076	1.86633239841443e-09\\
18292	5.84655691974127e-11\\
69300	1.8497194333741e-12\\
};
\addlegendentry{p=4, k=3, $\mu$=1}

\logLogSlopeTriangleUpsideDown{0.95}{0.09}{0.615}{1.5}{gray, line width = 2pt}{3};
\logLogSlopeTriangle{0.88}{0.09}{0.215}{2}{gray, line width = 2pt}{4};
\logLogSlopeTriangle{0.7}{0.1}{0.14}{2.5}{gray, line width = 2pt}{5};

\end{axis}
\end{tikzpicture}%
}
		\hspace*{1mm}
		\resizebox{0.32\linewidth}{!}{% This file was created by matlab2tikz.
%
%The latest updates can be retrieved from
%  http://www.mathworks.com/matlabcentral/fileexchange/22022-matlab2tikz-matlab2tikz
%where you can also make suggestions and rate matlab2tikz.
%
\definecolor{mycolor1}{rgb}{0.00000,0.44700,0.74100}%
\definecolor{mycolor2}{rgb}{0.85000,0.32500,0.09800}%
\definecolor{mycolor3}{rgb}{0.92900,0.69400,0.12500}%
\definecolor{mycolor4}{rgb}{0.49400,0.18400,0.55600}%
\definecolor{mycolor5}{rgb}{0.46600,0.67400,0.18800}%
\definecolor{mycolor6}{rgb}{0.30100,0.74500,0.93300}%
\begin{tikzpicture}
\tikzstyle{every node}=[font=\Large]
\begin{axis}[%
width=4.568in,
height=3.603in,
at={(0.766in,0.486in)},
scale only axis,
xmode=log,
xmin=20,
xmax=69300,
xminorticks=true,
xlabel={ndof},
ymode=log,
ymin=1.01980662698043e-14,
ymax=0.318027416616536,
yminorticks=true,
ylabel={$|\lambda-\lambda_h|$},
label style={font=\Large\color{white!15!black}},
tick label style={font=\Large},
axis background/.style={fill=white},
legend style={at={(0.0,0.0)}, anchor=south west, legend cell align=left, align=left, draw=white!15!black, font=\Large}
]
\addplot [color=mycolor1, line width=2.0pt, mark size=3.5pt, mark=square, mark options={solid, mycolor1}]
  table[row sep=crcr]{%
27	0.318027416616536\\
52	0.311141355079107\\
126	0.116254464979399\\
370	0.00789199050806566\\
1242	0.000461529458432253\\
4522	2.828868482041e-05\\
17226	1.75930133394786e-06\\
67210	1.09819682236889e-07\\
};
\addlegendentry{p=2, k=1, $\mu$=0.45}

\addplot [color=mycolor2, dotted, line width=2.0pt, mark size=3.3pt, mark=triangle, mark options={solid, mycolor2}]
  table[row sep=crcr]{%
27	0.318027416616536\\
52	0.179649463714172\\
126	0.00855306635173569\\
370	0.000496640251217428\\
1242	3.04070309287141e-05\\
4522	1.89009761442094e-06\\
17226	1.17963276480282e-07\\
67210	7.36990202199195e-09\\
};
\addlegendentry{p=2, k=1, $\mu$=1}

\addplot [color=mycolor3, line width=2.0pt, mark size=3.5pt, mark=square, mark options={solid, mycolor3}]
  table[row sep=crcr]{%
52	0.104125202171245\\
85	0.0942446288016363\\
175	0.0117771877540722\\
451	0.000166617516695666\\
1387	4.90901718563919e-06\\
4795	6.11242008119461e-08\\
17755	8.73033201287399e-10\\
68251	1.36459732402727e-11\\
};
\addlegendentry{p=3, k=2, $\mu$=0.3}

\addplot [color=mycolor4, dotted, line width=2.0pt, mark size=3.3pt, mark=triangle, mark options={solid, mycolor4}]
  table[row sep=crcr]{%
52	0.104125202171245\\
85	0.00225281297417723\\
175	0.000150011158181584\\
451	1.78202697220797e-06\\
1387	2.57698804517759e-08\\
4795	3.95999677493819e-10\\
17755	6.17994544427347e-12\\
68251	3.01980662698043e-14\\
};
\addlegendentry{p=3, k=2, $\mu$=1}

\addplot [color=mycolor5, line width=2.0pt, mark size=3.5pt, mark=square, mark options={solid, mycolor5}]
  table[row sep=crcr]{%
85	0.0197036887358539\\
126	0.0164096667762159\\
232	0.00311682058013218\\
540	0.000122334096065657\\
1540	6.6684811983464e-07\\
5076	5.37607469652812e-09\\
18292	2.91517920913975e-11\\
69300	1.52766688188422e-13\\
};
\addlegendentry{p=4, k=3, $\mu$=0.225}

\addplot [color=mycolor6, dotted, line width=2.0pt, mark size=3.3pt, mark=triangle, mark options={solid, mycolor6}]
  table[row sep=crcr]{%
85	0.0197036887358539\\
126	0.000287161251179668\\
232	3.27961671686694e-06\\
540	1.14351266233825e-08\\
1540	4.3648640257743e-11\\
5076	1.63424829224823e-13\\
18292	4.44089209850063e-14\\
69300	5.15143483426073e-14\\
};
\addlegendentry{p=4, k=3, $\mu$=1}

\logLogSlopeTriangleUpsideDown{0.947}{0.085}{0.61}{2}{gray, line width = 2pt}{4};
\logLogSlopeTriangleUpsideDown{0.947}{0.085}{0.36}{3}{gray, line width = 2pt}{6};
\logLogSlopeTriangleUpsideDown{0.947}{0.08}{0.24}{4}{gray, line width = 2pt}{8};

\end{axis}
\end{tikzpicture}%
}
		
		\caption{Approximation errors for the eigenfunction $u=u_{\nu_2,1}$ and its corresponding eigenvalue $\lambda=\lambda_{\nu_2,1}$ using gradually and uniformly refined NURBS of degree $\bs p =(p,p)$, $p\in \{2,3,4\}$, and regularity $\bs k = (p-1,p-1)$. We set $\mu = 0.9 \frac{\nu_2}{p}  =\frac{0.9}{p}$ for graded meshes and $\mu = 1$ corresponds to uniform refinement.}
		\label{fig: approximation errors for third eigenfunction variational crime}
	\end{figure}
\end{example}

\subsection{Spectral approximation properties of smooth splines on circular sectors}
\label{subsec: spectral approximation properties}	
Naturally, the solution of an eigenvalue problem is not just given by single eigenvalues and eigenfunctions but by an infinite sequence of such. Therefore, we use our approach to compute multiple solutions of the model problem \eqref{eq: model problem} in the next numerical experiment. The different regularity properties of the exact Laplace eigenfunctions, outlined in Section \ref{section: regularity of the eigenfunctions}, will be crucial to understand the numerical results.

In this context, it is of interest whether some of the spectral approximation properties of smooth splines, which have been shown for rectangular domains in the literature \cite{CottrellHughesBazilevs2009,SandeManniSpeleers2019,HiemstraHughesRealiSchillinger2021,ManniSandeSpeleers2022} can be extended to circular sectors. In most of the related works, the full discrete spectrum is approximated, that is, one discrete eigenvalue is computed per computational degree of freedom. The power of smooth splines in this setting is twofold. First, finite elements produce so called spectral branches, i.e., the eigenvalue error jumps for higher frequencies, a phenomenon which does not occur with maximally smooth splines. Second, the overall approximation constant of smooth splines is better over the whole discrete spectrum. In the following, we conduct numerical studies to investigate comparable properties of the discrete spectrum of circular sectors.

At this point, smooth splines on circular sectors should be considered cautiously. In $\zeta_2$-direction, the geometry mapping $\bs F$ is based on the classical quadratic NURBS parameterizations of circles. Typically, circles are constructed from multiple arcs smaller than $180^\circ$, and the resulting bases are no more than $C^0$-continuous at the junctions of the arcs \cite[Chapter 2.4.1.1]{CottrellHughesBazilevs2009}. In particular, we use four arcs of $90^\circ$ for the discretization of the unit disk with crack. Therefore, the univariate NURBS basis functions in $\zeta_2$-direction are only $C^{0}$-continuous at the points $\left\{1/4,1/2,3/4\right\}$. We explicitly compute these functions in \ref{sec: appendix explicit representation}.  Consequently, the coarse tensor product basis functions \eqref{eq: basis functions regular case} are $C^0$-continuous along the lines $\{(r \cos \varphi, r \sin\varphi ) : r \in (0,1), \varphi \in \{\pi/2, \pi, 3\pi/2\}\}$. After refinement, these $C^0$-lines will still be present. However, this is considered a minor limitation for the smoothness of the NURBS basis functions since, in all the regions between the specified lines of $C^0$-continuity, the regularity can be increased as desired by using the $k$-refinement algorithm introduced in Section \ref{sec: refinement procedures}.
\begin{example}
	\label{ex: full discrete spectrum}
	We compute the full discrete spectrum of the unit disk with crack using different parameters, all of which result in systems with a total of $976$ degrees of freedom. We compare bicubic $C^2$-continuous NURBS and biquintic $C^4$-continuous NURBS on uniform and graded meshes with their $C^0$-continuous counterparts. As outlined in Section \ref{subsec: choice of grading parameter}, we choose the strong grading parameter $\mu=0.9 \cdot \nu_1 /p$ from \eqref{eq: grading parameter concrete} for the graded meshes since we want to compute all discrete eigenvalues at once and some of them have corresponding eigenfunctions with the strong singularity of type $r^{\nu_1}$. Figure \ref{fig: full discrete spectrum} presents the relative eigenvalue errors, ordered by the discrete eigenvalue size. For both types of mesh, it becomes evident that the approximation accuracy of the displayed spectrum increases with higher NURBS regularity, except for the very upper part of the spectrum. Hence, the positive findings concerning the approximation constant of smooth splines, which have been illustrated in the literature for rectangular domains, can be extended to the case of circular sectors, at least in the lower part of the spectrum. However, we cannot identify the appearance of spectral branches. Our results on uniform meshes optically resemble those in \cite{BartezzaghiDedeQuarteroni2015}, where the Laplace-Beltrami eigenvalue problem has been considered on the unit sphere. In general, higher eigenvalues are approximated better by uniform meshes, which is attributable to their superior approximation constant on coarse refinement levels, as it can also be observed in Example \ref{ex: approximation of singular eigenfunction} and \ref{ex: approximation of smooth eigenfunction}.
	\begin{figure}
		\centering
		\includegraphics[width=0.495\linewidth, trim=1.3cm 0.8cm 1.6cm 1.4cm, clip]{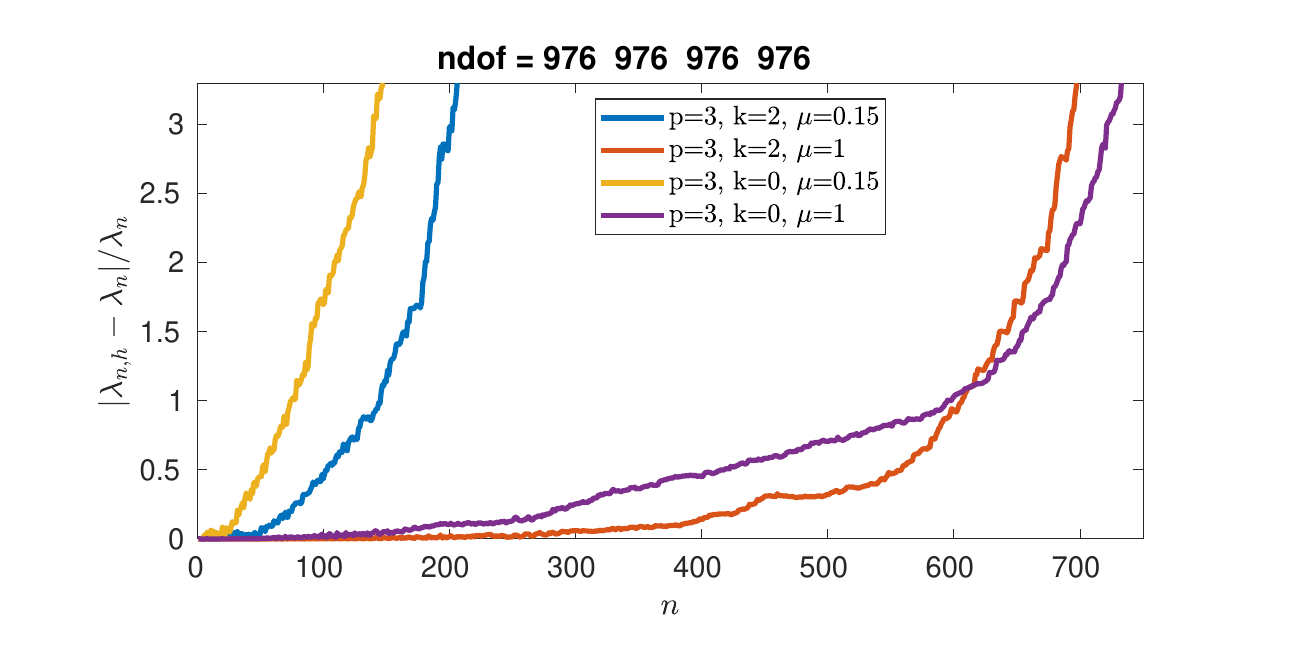}
		\hfill
		\includegraphics[width=0.495\linewidth, trim=1.3cm 0.8cm 1.6cm 1.4cm, clip ]{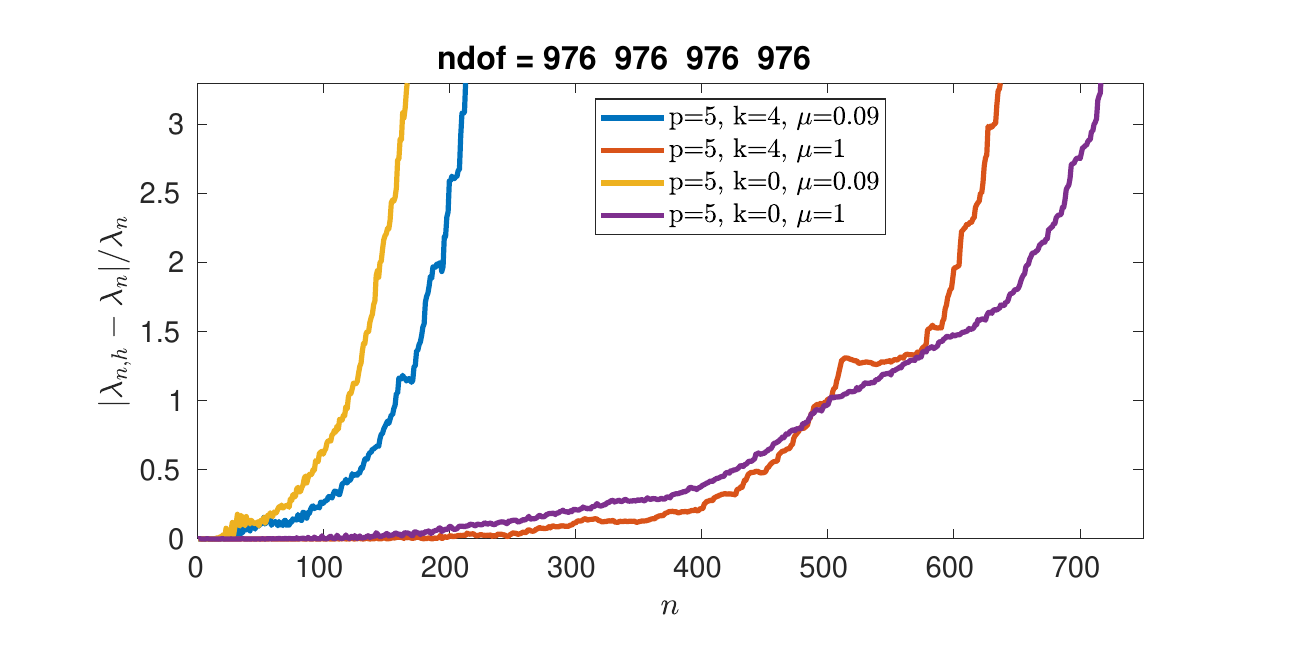}
		\caption{Relative eigenvalue errors, ordered by discrete eigenvalue size, using gradually and uniformly refined NURBS of degree $\bs p =(p,p)$ and regularity $\bs k= (k,k)$. A total of $976$ degrees of freedom is used in all the computations. We set $\mu = 0.9 \frac{\nu_1}{p} =\frac{0.9}{2p}$ for graded meshes and $\mu = 1$ corresponds to uniform refinement.}
		\label{fig: full discrete spectrum}
	\end{figure}
\end{example}

\begin{remark}
	In Example \ref{ex: full discrete spectrum} and Figure \ref{fig: full discrete spectrum}, each discrete eigenvalue is associated to an exact eigenvalue by ordering both sequences in ascending order. However, this matching is not always correct, that is, a discrete eigenvalue $\lambda_{n,h}$ might be an approximation of the exact eigenvalue $\lambda_{m}$ with $n \neq m$. Techniques for a proper association of discrete and exact eigenvalues have been proposed in the literature, but it has also been observed that the impact of mismatching is marginal \cite{HughesRealiSangalli2008,GianiGrubisicHakulaOvall2021,HiemstraHughesRealiSchillinger2021,ManniSandeSpeleers2022}.
\end{remark}

To demonstrate the strength of our proposed graded refinement approach, we need to adopt a slightly different point of view. The primary benefit of gradually refined NURBS is the optimal approximation of eigenfunctions with low regularity and their corresponding eigenvalues. This property is asymptotic which only comes into effect for an increasing number of computational degrees of freedom. Therefore, we decide to compute a fixed number of discrete eigenvalues and increase the number of degrees of freedom. Since our objective is to approximate the eigenvalues of circular sectors in the best possible way and the accuracy of Galerkin methods generally decreases for higher frequencies \cite[Chapter 6.3]{StrangFix2008}, it is natural to consider only the lower discrete eigenvalues and refine the mesh. The analysis of spectral branches up to the highest discrete frequencies is of specific interest for related time-dependent problems in explicit numerical schemes as the critical time-step size calculation is inversely proportional to the maximum discrete eigenfrequency \cite{Hughes2012}, but we do not consider this here. 

%First, it is much easier to check for double eigenvalues as we can simply compute the corresponding Bessel zeros \eqref{eq: formula eigenvalues}. Second, we can assure that the discrete eigenvalues match to their exact counterpart by visually comparing the eigenfunctions. 

\begin{example}
	We compare the relative errors for the $100$ smallest Laplace eigenvalues of the unit disk with crack using suitably graded and uniform meshes. Again, we choose the strong grading parameter $\mu=0.9 \cdot \nu_1 /p$ from \eqref{eq: grading parameter concrete} since we want to approximate multiple eigenvalues at once with a guaranteed optimal convergence rate.
	
	First, we employ biquadratic, $C^1$-continuous NURBS in a system with $68251$ degrees of freedom, which arises for $128$ subdivisions of the coarse mesh in each parametric direction as discussed in Section \ref{section: Isogeometric mesh grading}. The results are presented in the left hand plot of Figure \ref{fig: uniform vs graded mesh first 100 eigenvalues}. The uniform mesh produces some very poor approximations along with many accurate ones. More precisely, the approximates of the eigenvalues $\lambda_2, \lambda_8, \lambda_{20}, \dots$ are significantly less accurate compared to the rest. In Table \ref{tab:ZerosOfBesselfunctions}, we can see that these are the eigenvalues that belong to eigenfunctions of the lowest possible regularity. As outlined in Example \ref{ex: approximation of singular eigenfunction}, uniformly refined approximations do not converge optimally. Hence, these eigenvalues produce outliers in the spectral approximation since the remaining eigenvalues are associated with eigenfunctions of higher regularity which are also approached accurately on uniform meshes. In contrast, the eigenvalue errors achieved by the graded mesh are evenly distributed. All the eigenvalues are approximated with a comparable accuracy, which is naturally decreasing for higher eigenvalues. In particular, the outliers of the uniform version, $\lambda_2, \lambda_8, \lambda_{20}, \dots$, are approximated much better on the graded mesh. However, the graded approximation of the remaining eigenvalues with corresponding eigenfunctions of higher regularity is worse than the uniform variant. For this type of eigenvalues, the grading is unnecessary and has a negative effect on the approximation constant, as discussed in Section \ref{subsec: choice of grading parameter} and Example \ref{ex: approximation of smooth eigenfunction}.
	
	In a second test, we employ biquintic, $C^4$-continuous NURBS in a system with $68251$ degrees of freedom on a uniform and graded mesh. The relative eigenvalue errors are depicted in the middle plot of Figure \ref{fig: uniform vs graded mesh first 100 eigenvalues}. Here, in addition to the poor approximations of $\lambda_2, \lambda_8, \lambda_{20}, \dots$, the uniform mesh produces a second group of outliers, given by the eigenvalues $\lambda_4, \lambda_{13}, \lambda_{26}, \dots$. In Table \ref{tab:ZerosOfBesselfunctions}, we observe that these eigenvalues belong to eigenfunctions with the second strongest singularity. Due to the high degree of the NURBS, the difference between optimal and reduced convergence rates also becomes noticeable for this group of eigenvalues. Contrarily, the graded mesh produces an evenly distributed error point cloud, but less accurate results for eigenvalues that correspond to smoother eigenfunctions.
	
	Lastly, we compare biquintic $C^4$-continuous NURBS with their $C^0$-continuous counterparts on uniform and graded meshes. The results are presented in the right hand plot of Figure \ref{fig: uniform vs graded mesh first 100 eigenvalues}. For both types of mesh, the approximation is more accurate with smooth splines, although the same number of degrees of freedom is used. The superiority of smooth splines is even clearer on graded meshes as they do not produce outliers. Hence, we can again confirm the excellent spectral approximation constant of smooth splines on circular sectors. Since we consider the very lower part of the spectrum here which can not be analyzed well in the study conducted in Figure \ref{fig: full discrete spectrum}, the results in this framework complement the previous findings of Example \ref{ex: full discrete spectrum}.
	
	To summarize, this example shows that the combination of single-patch mesh grading with maximally smooth splines on circular sectors is a very powerful approach to approximate the Laplace eigenvalues of circular sectors. It is computationally cheap, accurate and produces an evenly distributed error throughout the considered eigenvalues.
	
	\begin{figure}[t]
		\includegraphics[width=0.33\linewidth, trim=0.65cm 0.8cm 1.6cm 1.3cm, clip]{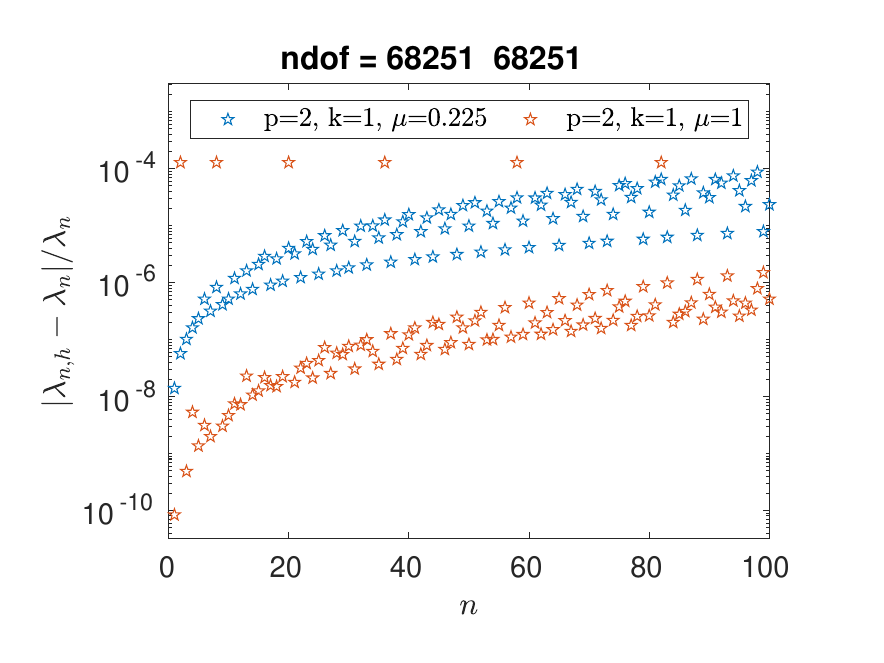}
		\includegraphics[width=0.33\linewidth, trim=0.65cm 0.8cm 1.6cm 1.3cm, clip]{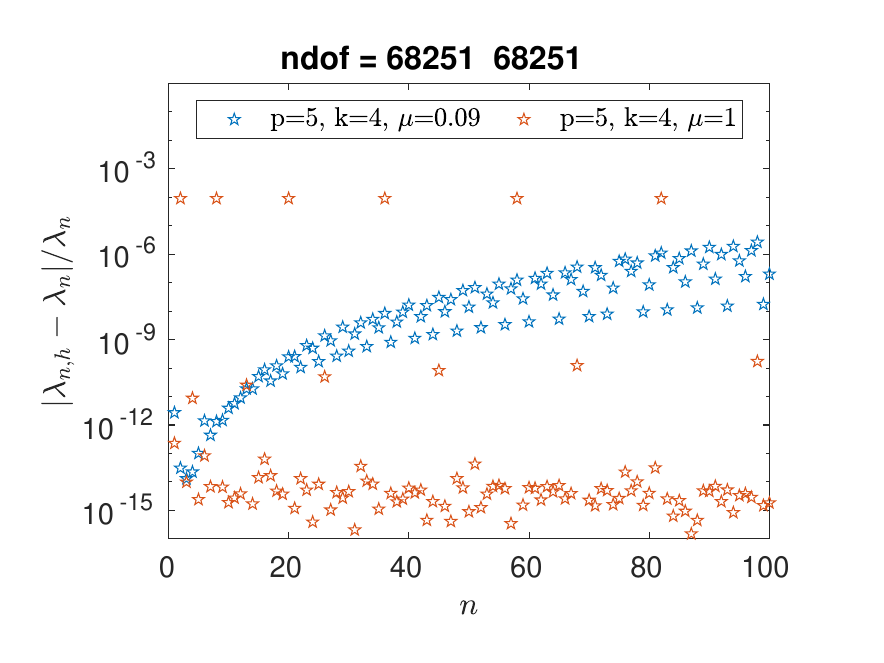}
		\includegraphics[width=0.33\linewidth, trim=0.65cm 0.8cm 1.6cm 1.3cm, clip]{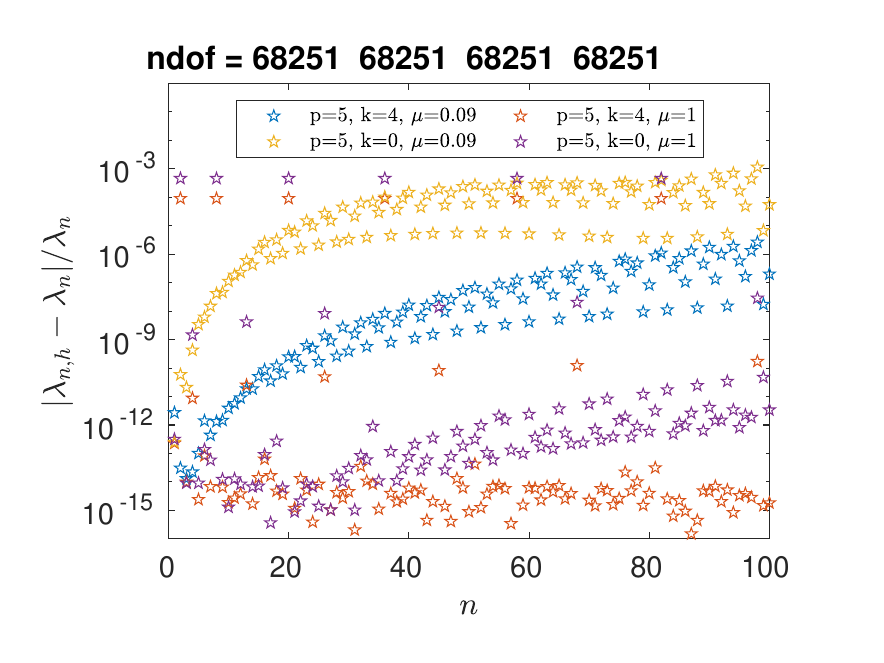}
		\caption{Relative errors for the first $100$ eigenvalues, ordered by discrete eigenvalue size, using gradually and uniformly refined NURBS of degree $\bs p =(p,p)$ and regularity $\bs k= (k,k)$. A total of $68251$ degrees of freedom is used in all the computations. We set $\mu = 0.9 \frac{\nu_1}{p} = \frac{0.9}{2p}$ for graded meshes and $\mu = 1$ corresponds to uniform refinement.}
		\label{fig: uniform vs graded mesh first 100 eigenvalues}
	\end{figure}
\end{example}

\subsection{Hierarchical meshes in combination with graded refinement}
\label{subsec: hierarchical refinement}
Finally, we present a variation of our method that aims at improving the constructed graded meshes. As discussed in Section \ref{subsec: anisotropic meshes}, the physical B\'ezier mesh is anisotropic next to the conical point of the circular sector due to the singular geometry mapping. To prevent these highly stretched elements, we can replace the tensor product meshes used so far by
hierarchical meshes.
%Loosely speaking, for small $j_1$, the vertical rectangles $[\zeta_{1,j_1}, \zeta_{1,j_1+1}] \times [0,1]$ of the parametric mesh can be subdivided in $\zeta_2$-direction in a hierarchical way. Hence, the rings $\{(r, \varphi) : r \in [\zeta_{1,j_1}, \zeta_{1,j_1+1}], \varphi \in [0,\omega] \}$ of the physical mesh are subdivided hierarchically, that is, the circular elements have smaller angles the closer to the conical point they lie. 
In Figure \ref{fig: Graded hierarchical mesh}, we illustrate the hierarchical mesh scheme exemplarily for the unit disk with crack in combination with both uniform and graded refinement. Again, this serves as a prototype for other circular sectors. Similar ideas have been mentioned in the context of singularly parameterized triangles \cite{Takacs2012}, subdivision based isogeometric analysis \cite{Takacs2023preprint}, or finite element methods on spherical domains \cite{ApelPester2005}. In GeoPDEs, the hierarchical refinement scheme can be implemented by using the extension of the package presented in \cite{GarauVazquez2018}. It leads to a purely isotropic mesh in the physical domain; the corresponding aspect ratio \eqref{eq: aspect ratio physical domain} simplifies to
\begin{align*}
\frac{h_1^{1/\mu}}{h_1^{1/\mu}} = 1 \, .
\end{align*}
Besides, the hierarchical meshes offer another advantage. As outlined in Section \ref{subsec: variational crime}, a variational crime is committed in the elements adjacent to the conical point since the corresponding basis functions do not belong to the solution space of the weak problem. By employing a hierarchical approach, the number of such elements is held constant over the refinement steps, which is not the case for standard tensor product meshes. We visualize this effect in Figure \ref{fig: elements close to center}. Hence, the number of problematic basis functions does not increase throughout hierarchical refinement. It can even be shown that the stiffness matrix entries \eqref{eq: approximate stiffness matrix entries} involving singular basis functions stay constant during refinement, but we do not go into detail here.

In the following numerical examples, we test the described hierarchical approach for approximating the Laplace eigenvalues and eigenfunctions of circular sectors.

%\begin{table}
%		\begin{center}
%		\begin{tabular}{ |c|c|c|c|c|c| } 
%			\hline
%			$J_1=J_2$ & $\tilde{a}(N_{\bs 1,\bs p},N_{\bs 1,\bs p})$ &  $\tilde{a}(N_{\bs 0,\bs p},N_{\bs 0,\bs p})$ \\ 
%			\hline
%			2 & 5.6542 & 1.885 \\
%			4 & 11.8237 & 1.885 \\
%			8 & 24.2659 & 1.885 \\
%			16 & 49.1976 & 1.885 \\
%			32 & 99.0838 & 1.885 \\
%			64 & 14.2684 & 1.885 \\
%			128 & 14.2684 & 1.885	\\		
%		\end{tabular}
%	\end{center}
%	\caption{matrix entries of the system \ref{eq: matrix eigenvalue system variational crime} and .}
%	\label{tab: stiffness amtrix entries singular vs modified}
%\end{table}

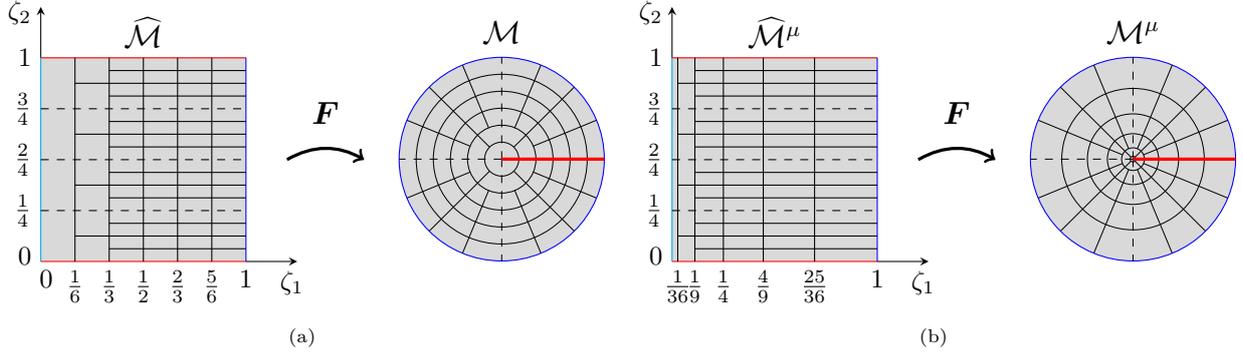
\begin{figure}
	\begin{center}
		\begin{subfigure}{0.495\linewidth}
			\begin{center}
				\begin{tikzpicture}[baseline, scale =0.675]			
				\def \N {5}			% Anzahl der Elemente in xi-Richtung
				\def \hxi{1/6}		% Elementgröße in xi-Richtung
				
				\def \M {3} 		% Anzahl der Elemente in eta-Richtung
				\def \heta{1/4}		% Elementgröße in eta-Richtung
				
				\node at (2,4.5) {\large $\widehat{\mathcal{M}}$};
				
				\fill[gray!30] (0,0) rectangle (4,4);
				
				\draw[-stealth] (4,0) -- (5, 0)  node[below, xshift = -2] {$\zeta_1$};
				\draw[-stealth] (0,4) -- (0, 5) node[left, yshift = -2] {$\zeta_2$};
				
				\draw[red] (0,0) -- (4, 0);
				\draw[red] (0,4) -- (4, 4);
				
				% Graded refinement
				\foreach \i in {1,...,\N}{
					\draw (\i*\hxi*4,0) -- (\i*\hxi*4,4);
				}
				
				\draw[cyan] (0,0) -- (0, 4);
				\draw[blue] (4,0) -- (4, 4);
				
				\foreach \i in {1,...,3}{
					\draw[dashed] (0, \i) -- (4, \i);
				}
				
				% Hierarchical refinement:
				\foreach \j in {0,...,3}{	
					\draw (\hxi*4, 2*\heta + \j) -- (4, 2*\heta + \j); 			
					\foreach \i in {1,3}{
						\draw (\hxi*2*4, \i*\heta + \j) -- (4, \i*\heta + \j);
					}
				}
				
				\node[left, yshift = 2] at (0,0) {$0$};
				\node[left] at (0,1) {$\frac14$};
				\node[left] at (0,2) {$\frac24$};
				\node[left] at (0,3) {$\frac34$};
				\node[left] at (0,4) {$1$};
				
				\node[below, xshift = 2] at (0,0) {$0$};
				\node[below] at (\hxi*1*4,0) {$\frac16$};
				\node[below] at (\hxi*2*4,0) {$\frac13$};
				\node[below] at (\hxi*3*4,0) {$\frac12$};
				\node[below] at (\hxi*4*4,0) {$\frac23$};
				\node[below] at (\hxi*5*4,0) {$\frac56$};
				\node[below] at (4,0) {$1$};
				
				\draw [->, very thick] (4.8,2) to [out=30,in=150] (6.3,2) node[above] at (5.55,2.5) {\large $\bs F$};
				\end{tikzpicture}
				\hspace*{2mm}
				\begin{tikzpicture}[scale = 1.35]		
				
				\def \N {5}			% Anzahl der Elemente in xi-Richtung
				\def \hxi{1/6}		% Elementgröße in xi-Richtung
				
				\node at (0,1.25) {\large $\mathcal{M}$};
				
				\fill[gray!30] (0,0) circle (1);
				
				\foreach \i in {0,...,\N}{
					\draw [line width=0.05mm] (0:0) circle (\i*\hxi);
				}
				\draw [blue] (0:0) circle (1);
				
				\foreach \x in {0.707}{
					\draw [line width=0.05mm] (\hxi*\x,{\hxi*sqrt(1-\x*\x)}) -- (\x,{sqrt(1-\x*\x)});
					\draw [line width=0.05mm] (-\hxi*\x,{\hxi*sqrt(1-\x*\x)}) -- (-\x,{sqrt(1-\x*\x)});
					\draw [line width=0.05mm] (-\hxi*\x,-{\hxi*sqrt(1-\x*\x)}) -- (-\x,-{sqrt(1-\x*\x)});
					\draw [line width=0.05mm] (\hxi*\x,-{\hxi*sqrt(1-\x*\x)}) -- (\x,-{sqrt(1-\x*\x)});
				}	
				
				\foreach \x in {0.382, 0.924}{
					\draw [line width=0.05mm] (2*\hxi*\x,{2*\hxi*sqrt(1-\x*\x)}) -- (\x,{sqrt(1-\x*\x)});
					\draw [line width=0.05mm] (-2*\hxi*\x,{2*\hxi*sqrt(1-\x*\x)}) -- (-\x,{sqrt(1-\x*\x)});
					\draw [line width=0.05mm] (-2*\hxi*\x,-{2*\hxi*sqrt(1-\x*\x)}) -- (-\x,-{sqrt(1-\x*\x)});
					\draw [line width=0.05mm] (2*\hxi*\x,-{2*\hxi*sqrt(1-\x*\x)}) -- (\x,-{sqrt(1-\x*\x)});
				}
				
				\draw [line width=0.05mm, dashed] (0,0) -- (0,1);
				\draw [line width=0.05mm, dashed] (0,0) -- (-1,0);
				\draw [line width=0.05mm, dashed] (0,0) -- (0,-1);		
				\draw [-,red, very thick] (0:0) -- (0:1);
				
				\end{tikzpicture}
				\subcaption{}
			\end{center}
		\end{subfigure}
		\hfill
		\begin{subfigure}{0.495\linewidth}
			\begin{center}
				\begin{tikzpicture}[baseline, scale = 0.675]
				
				\def \N {5}			% Anzahl der Elemente in xi-Richtung
				\def \hxi{1/6}		% Elementgröße in xi-Richtung
				
				\def \M {3} 		% Anzahl der Elemente in eta-Richtung
				\def \heta{1/4}		% Elementgröße in eta-Richtung
				
				\node at (2,4.5) {\large $\widehat{\mathcal{M}}^\mu$};
				
				\fill[gray!30] (0,0) rectangle (4,4);
				
				\draw[-stealth] (4,0) -- (5, 0)  node[below, xshift = -2] {$\zeta_1$};
				\draw[-stealth] (0,4) -- (0, 5) node[left, yshift = -2] {$\zeta_2$};

				\draw[red] (0,0) -- (4, 0);
				\draw[red] (0,4) -- (4, 4);
				
				% Graded refinement:
				\foreach \i in {1,...,\N}{
					\draw (\i*\i*\hxi*\hxi*4,0) -- (\i*\i*\hxi*\hxi*4,4);
				}
				
				\draw[cyan] (0,0) -- (0, 4);
				\draw[blue] (4,0) -- (4, 4);
				
				\foreach \i in {1,...,3}{
					\draw[dashed] (0, \i) -- (4, \i);
				}
				
				% Hierarchical refinement: 
				\foreach \j in {0,...,3}{	
					\draw (\hxi*\hxi*4, 2*\heta + \j) -- (4, 2*\heta + \j); 			
					\foreach \i in {1,3}{
						\draw (\hxi*\hxi*2*2*4, \i*\heta + \j) -- (4, \i*\heta + \j);
					}
				}
				
				\node[left, yshift = 2] at (0,0) {$0$};
				\node[left] at (0,1) {$\frac14$};
				\node[left] at (0,2) {$\frac24$};
				\node[left] at (0,3) {$\frac34$};
				\node[left] at (0,4) {$1$};
				
				\node[below] at (\hxi*\hxi*1*1*4,0) {$\frac{1}{36}$};
				\node[below] at (\hxi*\hxi*2*2*4,0) {$\frac{1}{9}$};
				\node[below] at (\hxi*\hxi*3*3*4,0) {$\frac{1}{4}$};
				\node[below] at (\hxi*\hxi*4*4*4,0) {$\frac{4}{9}$};
				\node[below] at (\hxi*\hxi*5*5*4,0) {$\frac{25}{36}$};
				\node[below] at (4,0) {$1$};
				
				\draw [->, very thick] (4.8,2) to [out=30,in=150] (6.3,2) node[above] at (5.55,2.5) {\large $\bs F$};
				\end{tikzpicture}	
				\hspace*{2mm}
				\begin{tikzpicture}[scale = 1.35]
				
				\def \N {5}			% Anzahl der Elemente in xi-Richtung
				\def \hxi{1/6}		% Elementgröße in xi-Richtung
				
				\node at (0,1.25) {\large $\mathcal{M}^\mu$};
				
				\fill[gray!30] (0,0) circle (1);
				
				\foreach \i in {0,...,\N}{
					\draw [line width=0.05mm] (0:0) circle (\i*\i*\hxi*\hxi);
				}
				\draw [blue] (0:0) circle (1);
				
				\foreach \x in {0.707}{
					\draw [line width=0.05mm] (\hxi*\hxi*\x,{\hxi*\hxi*sqrt(1-\x*\x)}) -- (\x,{sqrt(1-\x*\x)});
					\draw [line width=0.05mm] (-\hxi*\hxi*\x,{\hxi*\hxi*sqrt(1-\x*\x)}) -- (-\x,{sqrt(1-\x*\x)});
					\draw [line width=0.05mm] (-\hxi*\hxi*\x,-{\hxi*\hxi*sqrt(1-\x*\x)}) -- (-\x,-{sqrt(1-\x*\x)});
					\draw [line width=0.05mm] (\hxi*\hxi*\x,-{\hxi*\hxi*sqrt(1-\x*\x)}) -- (\x,-{sqrt(1-\x*\x)});
				}	
				
				\foreach \x in {0.382, 0.924}{
					\draw [line width=0.05mm] (2*\hxi*2*\hxi*\x,{2*\hxi*2*\hxi*sqrt(1-\x*\x)}) -- (\x,{sqrt(1-\x*\x)});
					\draw [line width=0.05mm] (-2*\hxi*2*\hxi*\x,{2*\hxi*2*\hxi*sqrt(1-\x*\x)}) -- (-\x,{sqrt(1-\x*\x)});
					\draw [line width=0.05mm] (-2*\hxi*2*\hxi*\x,-{2*\hxi*2*\hxi*sqrt(1-\x*\x)}) -- (-\x,-{sqrt(1-\x*\x)});
					\draw [line width=0.05mm] (2*\hxi*2*\hxi*\x,-{2*\hxi*2*\hxi*sqrt(1-\x*\x)}) -- (\x,-{sqrt(1-\x*\x)});
				}	
				
				\draw [line width=0.05mm, dashed] (0,0) -- (0,1);
				\draw [line width=0.05mm, dashed] (0,0) -- (-1,0);
				\draw [line width=0.05mm, dashed] (0,0) -- (0,-1);		
				\draw [-,red,very thick] (0:0) -- (0:1);
				
				\end{tikzpicture}
				\subcaption{}
			\end{center}
		\end{subfigure}
		\vspace{-2mm}
		\caption{Hierarchical parametric and physical B\'ezier meshes $\widehat{\mathcal{M}}$ and $\mathcal{M}$ after refining with $J_1=6$ and $J_2=16$. Solid black lines are used for new subdivisions and dashed black lines for the initial coarse meshes $\widehat{\mathcal{M}_0}$ and $\mathcal{M}_0$. Note the difference to the tensor product meshes in Figure \ref{fig: refined meshes}. (a) Uniform refinement. (b) Graded refinement with grading parameter $\mu =\frac12$.}
		\label{fig: Graded hierarchical mesh}		
	\end{center}
\end{figure}

\begin{figure}[t]
	\begin{center}
		\begin{subfigure}{0.495\linewidth}
			\begin{center}
				\begin{tikzpicture}[baseline, scale = 1.3]		
				\draw[fill = gray!30] (0:0) circle (1);
				
				\draw [line width=0.05mm, dashed] (0,0) -- (0,1);
				\draw [line width=0.05mm, dashed] (0,0) -- (-1,0);
				\draw [line width=0.05mm, dashed] (0,0) -- (0,-1);		
				\draw [-,red, very thick] (0:0) -- (0:1);
				\end{tikzpicture}
				\begin{tikzpicture}[baseline, scale = 1.1]		
				
				\draw[fill = gray!30] (0:0) circle (1);
				
				% formula: \x = sin(\heta*j*pi/2)
				\foreach \x in {0.707}{
					\draw [line width=0.05mm] (0,0) -- (\x,{sqrt(1-\x*\x)});
					\draw [line width=0.05mm] (0,0) -- (-\x,{sqrt(1-\x*\x)});
					\draw [line width=0.05mm] (0,0) -- (-\x,-{sqrt(1-\x*\x)});
					\draw [line width=0.05mm] (0,0) -- (\x,-{sqrt(1-\x*\x)});
				}	
				
				\draw [line width=0.05mm, dashed] (0,0) -- (0,1);
				\draw [line width=0.05mm, dashed] (0,0) -- (-1,0);
				\draw [line width=0.05mm, dashed] (0,0) -- (0,-1);		
				\draw [-,red, very thick] (0:0) -- (0:1);
				\end{tikzpicture}
				\begin{tikzpicture}[baseline, scale = 0.85]		
				\draw[fill = gray!30] (0:0) circle (1);
				
				% formula: \x = sin(\heta*j*pi/2)
				\foreach \x in {0.382,0.707, 0.924}{
					\draw [line width=0.05mm] (0,0) -- (\x,{sqrt(1-\x*\x)});
					\draw [line width=0.05mm] (0,0) -- (-\x,{sqrt(1-\x*\x)});
					\draw [line width=0.05mm] (0,0) -- (-\x,-{sqrt(1-\x*\x)});
					\draw [line width=0.05mm] (0,0) -- (\x,-{sqrt(1-\x*\x)});
				}	
				
				\draw [line width=0.05mm, dashed] (0,0) -- (0,1);
				\draw [line width=0.05mm, dashed] (0,0) -- (-1,0);
				\draw [line width=0.05mm, dashed] (0,0) -- (0,-1);		
				\draw [-,red, very thick] (0:0) -- (0:1);
				\end{tikzpicture}
				\begin{tikzpicture}[baseline, scale = 0.6]		
				\draw[fill = gray!30] (0:0) circle (1);
				
				% formula: \x = sin(\heta*j*pi/2)
				\foreach \x in {0.195,0.382, 0.555, 0.707, 0.831, 0.924, 0.981}{
					\draw [line width=0.05mm] (0,0) -- (\x,{sqrt(1-\x*\x)});
					\draw [line width=0.05mm] (0,0) -- (-\x,{sqrt(1-\x*\x)});
					\draw [line width=0.05mm] (0,0) -- (-\x,-{sqrt(1-\x*\x)});
					\draw [line width=0.05mm] (0,0) -- (\x,-{sqrt(1-\x*\x)});
				}	
				
				\draw [line width=0.05mm, dashed] (0,0) -- (0,1);
				\draw [line width=0.05mm, dashed] (0,0) -- (-1,0);
				\draw [line width=0.05mm, dashed] (0,0) -- (0,-1);		
				\draw [-,red, very thick] (0:0) -- (0:1);
				\end{tikzpicture}
				\subcaption{}
			\end{center}
		\end{subfigure}
		\hfill
		\begin{subfigure}{0.495\linewidth}
			\begin{center}
				\begin{tikzpicture}[baseline, scale = 1.3]		
				\draw[fill = gray!30] (0:0) circle (1);
				
				\draw [line width=0.05mm, dashed] (0,0) -- (0,1);
				\draw [line width=0.05mm, dashed] (0,0) -- (-1,0);
				\draw [line width=0.05mm, dashed] (0,0) -- (0,-1);		
				\draw [-,red, very thick] (0:0) -- (0:1);
				\end{tikzpicture}
				\begin{tikzpicture}[baseline, scale = 1.1]		
				
				\draw[fill = gray!30] (0:0) circle (1);
				
				\draw [line width=0.05mm, dashed] (0,0) -- (0,1);
				\draw [line width=0.05mm, dashed] (0,0) -- (-1,0);
				\draw [line width=0.05mm, dashed] (0,0) -- (0,-1);		
				\draw [-,red, very thick] (0:0) -- (0:1);
				\end{tikzpicture}
				\begin{tikzpicture}[baseline, scale = 0.85]		
				\draw[fill = gray!30] (0:0) circle (1);	
				
				\draw [line width=0.05mm, dashed] (0,0) -- (0,1);
				\draw [line width=0.05mm, dashed] (0,0) -- (-1,0);
				\draw [line width=0.05mm, dashed] (0,0) -- (0,-1);		
				\draw [-,red, very thick] (0:0) -- (0:1);
				\end{tikzpicture}
				\begin{tikzpicture}[baseline, scale = 0.6]		
				\draw[fill = gray!30] (0:0) circle (1);
				
				\draw [line width=0.05mm, dashed] (0,0) -- (0,1);
				\draw [line width=0.05mm, dashed] (0,0) -- (-1,0);
				\draw [line width=0.05mm, dashed] (0,0) -- (0,-1);		
				\draw [-,red, very thick] (0:0) -- (0:1);
				\end{tikzpicture}
				\subcaption{}
				\label{fig: elements close to center hierarchical}
			\end{center}
		\end{subfigure}
		\vspace{-2mm}
		\caption{Sketch of the physical mesh elements near the conical point during refinement. (a) tensor product refinement. (b) hierarchical refinement. The scaling between the refinement steps is manipulated for better visibility.}
		\label{fig: elements close to center}
	\end{center}	
\end{figure}
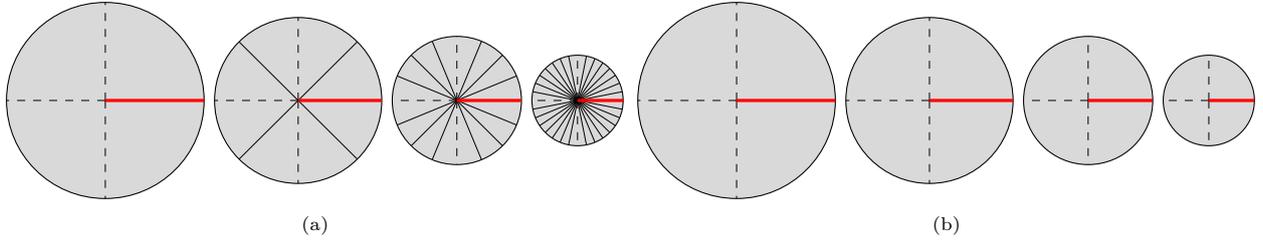

\begin{example}[Eigenfunction of Type (A)]
	We perform the same experiment as in Example \ref{ex: approximation of singular eigenfunction}, this time employing a hierarchical refinement approach. In more detail, we compute the eigenvalue $\lambda_{\nu_1,1}$ of the unit disk with crack and its corresponding eigenfunction $u_{\nu_1,1}$, which is characterized by its low regularity, on both uniform and graded hierarchical meshes. The approximation errors for NURBS of polynomial degree $\bs p =(p,p)$, where $p\in\{2,3,4\}$, and regularity $\bs k= (p-1,p-1)$ are presented in Figure \ref{fig: approximation errors for second eigenfunction hierarchical} and closely resemble those obtained with tensor product meshes in Figure \ref{fig: approximation errors for second eigenfunction variational crime}. Optimal convergence rates for the $H^1_h(\Omega)$- and $L^2_h(\Omega)$-eigenfunction error and the absolute eigenvalue error are recovered by using an appropriate grading parameter. An explicit comparison of the approximation errors generated by hierarchical and tensor product meshes is depicted in Figure \ref{fig: tensor vs hierachical mesh second eigenfunction}, where NURBS of degree $\bs p = (4,4)$ and regularity $\bs k = (3,3)$ are refined gradually and uniformly with grading parameters $\mu = 0.9 \cdot \frac{\nu_1}{4} = 0.1125$ and $\mu =1$, respectively. The results indicate that the approximation constant of hierarchical meshes is slightly better than the one of tensor product meshes. This effect can be explained by the omission of redundant degrees of freedom during the hierarchical refinement procedure. As illustrated in Figure \ref{fig: elements close to center}, the circular elements close to the singularity are subdivided fewer times while the same approximation accuracy is achieved. 
	
	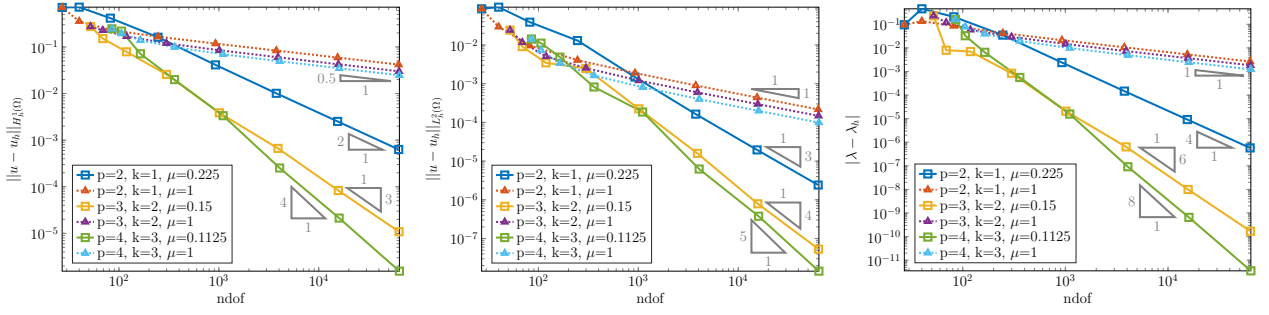
\begin{figure}[t]
		\resizebox{0.32\linewidth}{!}{% This file was created by matlab2tikz.
%
%The latest updates can be retrieved from
%  http://www.mathworks.com/matlabcentral/fileexchange/22022-matlab2tikz-matlab2tikz
%where you can also make suggestions and rate matlab2tikz.
%
\definecolor{mycolor1}{rgb}{0.00000,0.44700,0.74100}%
\definecolor{mycolor2}{rgb}{0.85000,0.32500,0.09800}%
\definecolor{mycolor3}{rgb}{0.92900,0.69400,0.12500}%
\definecolor{mycolor4}{rgb}{0.49400,0.18400,0.55600}%
\definecolor{mycolor5}{rgb}{0.46600,0.67400,0.18800}%
\definecolor{mycolor6}{rgb}{0.30100,0.74500,0.93300}%
\begin{tikzpicture}
\tikzstyle{every node}=[font=\Large]
\begin{axis}[%
width=4.568in,
height=3.603in,
at={(0.766in,0.486in)},
scale only axis,
xmode=log,
xmin=27,
xmax=64192,
xminorticks=true,
xlabel={ndof},
ymode=log,
ymin=1.54532894553275e-06,
ymax=0.714027634338522,
yminorticks=true,
ylabel={$||u-u_h||_{H^1_h(\Omega)}$},
axis background/.style={fill=white},
legend style={at={(0.03,0.03)}, anchor=south west, legend cell align=left, align=left, draw=white!15!black}
]
\addplot [color=mycolor1, line width=2.0pt, mark size=3.5pt, mark=square, mark options={solid, mycolor1}]
  table[row sep=crcr]{%
27	0.704063044506597\\
40	0.714027634338522\\
82	0.41496082570124\\
246	0.161804539766382\\
926	0.0412091998471589\\
3758	0.0101242294875468\\
15438	0.00252340986672391\\
63118	0.000631502542642988\\
};
\addlegendentry{p=2, k=1, $\mu$=0.225}

\addplot [color=mycolor2, dotted, line width=2.0pt, mark size=3.3pt, mark=triangle, mark options={solid, mycolor2}]
  table[row sep=crcr]{%
27	0.704063044506597\\
40	0.355963806077109\\
82	0.232051691599364\\
246	0.165998371619807\\
926	0.118052452256192\\
3758	0.0836140919729532\\
15438	0.0591488564892584\\
63118	0.0418283025727952\\
};
\addlegendentry{p=2, k=1, $\mu$=1}

\addplot [color=mycolor3, line width=2.0pt, mark size=3.5pt, mark=square, mark options={solid, mycolor3}]
  table[row sep=crcr]{%
52	0.274663630306856\\
69	0.151268473544978\\
119	0.0786734605253124\\
299	0.0257689035488057\\
1011	0.00386155081203868\\
3907	0.000665616393264005\\
15715	8.30087531915839e-05\\
63651	1.08494909805248e-05\\
};
\addlegendentry{p=3, k=2, $\mu$=0.15}

\addplot [color=mycolor4, dotted, line width=2.0pt, mark size=3.3pt, mark=triangle, mark options={solid, mycolor4}]
  table[row sep=crcr]{%
52	0.274663630306856\\
69	0.227343471106563\\
119	0.168994444449163\\
299	0.120019401665133\\
1011	0.084919042383957\\
3907	0.0600521334949714\\
15715	0.0424613421858944\\
63651	0.0300233447407702\\
};
\addlegendentry{p=3, k=2, $\mu$=1}

\addplot [color=mycolor5, line width=2.0pt, mark size=3.5pt, mark=square, mark options={solid, mycolor5}]
  table[row sep=crcr]{%
85	0.245261895704637\\
106	0.218692013038747\\
164	0.0713589159201717\\
360	0.0198924119675585\\
1104	0.00335118724551938\\
4064	0.000252660956848008\\
16000	2.12837883820422e-05\\
64192	1.54532894553275e-06\\
};
\addlegendentry{p=4, k=3, $\mu$=0.1125}

\addplot [color=mycolor6, dotted, line width=2.0pt, mark size=3.3pt, mark=triangle, mark options={solid, mycolor6}]
  table[row sep=crcr]{%
85	0.245261895704637\\
106	0.193871062936146\\
164	0.139167253373216\\
360	0.098641196706933\\
1104	0.0697693805278115\\
4064	0.0493315582468908\\
16000	0.0348807049560481\\
64192	0.0246634803276778\\
};
\addlegendentry{p=4, k=3, $\mu$=1}

\logLogSlopeTriangle{0.825}{0.15}{0.72}{0.25}{gray, line width = 2pt}{0.5};
\logLogSlopeTriangle{0.85}{0.1}{0.46}{01}{gray, line width = 2pt}{2};
\logLogSlopeTriangleUpsideDown{0.945}{0.1}{0.315}{1.5}{gray, line width = 2pt}{3};
\logLogSlopeTriangle{0.68}{0.1}{0.2}{2}{gray, line width = 2pt}{4};

\end{axis}
\end{tikzpicture}%
}
		\hspace*{1mm}
		\resizebox{0.32\linewidth}{!}{% This file was created by matlab2tikz.
%
%The latest updates can be retrieved from
%  http://www.mathworks.com/matlabcentral/fileexchange/22022-matlab2tikz-matlab2tikz
%where you can also make suggestions and rate matlab2tikz.
%
\definecolor{mycolor1}{rgb}{0.00000,0.44700,0.74100}%
\definecolor{mycolor2}{rgb}{0.85000,0.32500,0.09800}%
\definecolor{mycolor3}{rgb}{0.92900,0.69400,0.12500}%
\definecolor{mycolor4}{rgb}{0.49400,0.18400,0.55600}%
\definecolor{mycolor5}{rgb}{0.46600,0.67400,0.18800}%
\definecolor{mycolor6}{rgb}{0.30100,0.74500,0.93300}%
\begin{tikzpicture}
\tikzstyle{every node}=[font=\Large]
\begin{axis}[%
width=4.568in,
height=3.603in,
at={(0.766in,0.486in)},
scale only axis,
xmode=log,
xmin=27,
xmax=64192,
xminorticks=true,
xlabel={ndof},
ymode=log,
ymin=1.435996037364e-08,
ymax=0.0949391013233217,
yminorticks=true,
ylabel={$||u-u_h||_{L^2_h(\Omega)}$},
axis background/.style={fill=white},
legend style={at={(0.03,0.03)}, anchor=south west, legend cell align=left, align=left, draw=white!15!black}
]
\addplot [color=mycolor1, line width=2.0pt, mark size=3.5pt, mark=square, mark options={solid, mycolor1}]
  table[row sep=crcr]{%
27	0.0860868587643081\\
40	0.0949391013233217\\
82	0.0388838776215731\\
246	0.0130220841758417\\
926	0.00148137934139673\\
3758	0.000163519087629206\\
15438	1.94971577019341e-05\\
63118	2.40968621664826e-06\\
};
\addlegendentry{p=2, k=1, $\mu$=0.225}

\addplot [color=mycolor2, dotted, line width=2.0pt, mark size=3.3pt, mark=triangle, mark options={solid, mycolor2}]
  table[row sep=crcr]{%
27	0.0860868587643081\\
40	0.0295102810755566\\
82	0.00961925000992503\\
246	0.00405855980898774\\
926	0.00189352785225156\\
3758	0.00090061430239201\\
15438	0.000438799311218948\\
63118	0.000216471306340585\\
};
\addlegendentry{p=2, k=1, $\mu$=1}

\addplot [color=mycolor3, line width=2.0pt, mark size=3.5pt, mark=square, mark options={solid, mycolor3}]
  table[row sep=crcr]{%
52	0.0242102770274886\\
69	0.00907800455244273\\
119	0.00346246367328862\\
299	0.00247393460500645\\
1011	0.000224489336012727\\
3907	1.58198703258883e-05\\
15715	7.88609023795975e-07\\
63651	5.32936056508302e-08\\
};
\addlegendentry{p=3, k=2, $\mu$=0.15}

\addplot [color=mycolor4, dotted, line width=2.0pt, mark size=3.3pt, mark=triangle, mark options={solid, mycolor4}]
  table[row sep=crcr]{%
52	0.0242102770274886\\
69	0.0117095961019742\\
119	0.00507890085199116\\
299	0.00254133519201163\\
1011	0.00120782620194031\\
3907	0.000597227907851035\\
15715	0.000296833568083216\\
63651	0.000147971195235997\\
};
\addlegendentry{p=3, k=2, $\mu$=1}

\addplot [color=mycolor5, line width=2.0pt, mark size=3.5pt, mark=square, mark options={solid, mycolor5}]
  table[row sep=crcr]{%
85	0.0142911219244662\\
106	0.0111783727259739\\
164	0.00486684992232974\\
360	0.00082230239212875\\
1104	0.000185349669865249\\
4064	6.2507151147787e-06\\
16000	3.76600260068177e-07\\
64192	1.435996037364e-08\\
};
\addlegendentry{p=4, k=3, $\mu$=0.1125}

\addplot [color=mycolor6, dotted, line width=2.0pt, mark size=3.3pt, mark=triangle, mark options={solid, mycolor6}]
  table[row sep=crcr]{%
85	0.0142911219244662\\
106	0.00734986680720939\\
164	0.0033832802509867\\
360	0.00162446726623942\\
1104	0.000811767766895346\\
4064	0.000400061392250434\\
16000	0.00019900694052975\\
64192	9.92567713314827e-05\\
};
\addlegendentry{p=4, k=3, $\mu$=1}

\logLogSlopeTriangleUpsideDown{0.94}{0.14}{0.69}{0.5}{gray, line width = 2pt}{1};
\logLogSlopeTriangleUpsideDown{0.945}{0.1}{0.47}{1.5}{gray, line width = 2pt}{3};
\logLogSlopeTriangleUpsideDown{0.945}{0.1}{0.26}{2}{gray, line width = 2pt}{4};
\logLogSlopeTriangle{0.8}{0.1}{0.07}{2.5}{gray, line width = 2pt}{5};

\end{axis}
\end{tikzpicture}%
}
		\hspace*{1mm}
		\resizebox{0.33\linewidth}{!}{% This file was created by matlab2tikz.
%
%The latest updates can be retrieved from
%  http://www.mathworks.com/matlabcentral/fileexchange/22022-matlab2tikz-matlab2tikz
%where you can also make suggestions and rate matlab2tikz.
%
\definecolor{mycolor1}{rgb}{0.00000,0.44700,0.74100}%
\definecolor{mycolor2}{rgb}{0.85000,0.32500,0.09800}%
\definecolor{mycolor3}{rgb}{0.92900,0.69400,0.12500}%
\definecolor{mycolor4}{rgb}{0.49400,0.18400,0.55600}%
\definecolor{mycolor5}{rgb}{0.46600,0.67400,0.18800}%
\definecolor{mycolor6}{rgb}{0.30100,0.74500,0.93300}%
\begin{tikzpicture}
\tikzstyle{every node}=[font=\Large]
\begin{axis}[%
width=4.65in,
height=3.543in,
at={(0.78in,0.478in)},
scale only axis,
xmode=log,
xmin=27,
xmax=64192,
xminorticks=true,
xlabel={ndof},
ymode=log,
ymin=3.5473846082823e-12,
ymax=0.45492084123245,
yminorticks=true,
ylabel={$|\lambda-\lambda_h|$},
axis background/.style={fill=white},
legend style={at={(0.03,0.03)}, anchor=south west, legend cell align=left, align=left, draw=white!15!black}
]
\addplot [color=mycolor1, line width=2.0pt, mark size=3.5pt, mark=square, mark options={solid, mycolor1}]
  table[row sep=crcr]{%
27	0.0964562259659676\\
40	0.45492084123245\\
82	0.208080623397649\\
246	0.0349090472651721\\
926	0.00240087756493423\\
3758	0.000147410017326877\\
15438	9.1984556451763e-06\\
63118	5.76922218442633e-07\\
};
\addlegendentry{p=2, k=1, $\mu$=0.225}

\addplot [color=mycolor2, dotted, line width=2.0pt, mark size=3.3pt, mark=triangle, mark options={solid, mycolor2}]
  table[row sep=crcr]{%
27	0.0964562259659676\\
40	0.135616705744871\\
82	0.0850708604580124\\
246	0.0426700051611526\\
926	0.0212560601158813\\
3758	0.0106152047323995\\
15438	0.00530659757872165\\
63118	0.00265341029508015\\
};
\addlegendentry{p=2, k=1, $\mu$=1}

\addplot [color=mycolor3, line width=2.0pt, mark size=3.5pt, mark=square, mark options={solid, mycolor3}]
  table[row sep=crcr]{%
52	0.230709987377175\\
69	0.00798908860394043\\
119	0.00700935975297234\\
299	0.000844130143395461\\
1011	2.08571681703518e-05\\
3907	6.37515867651928e-07\\
15715	9.96152138554862e-09\\
63651	1.7036505539636e-10\\
};
\addlegendentry{p=3, k=2, $\mu$=0.15}

\addplot [color=mycolor4, dotted, line width=2.0pt, mark size=3.3pt, mark=triangle, mark options={solid, mycolor4}]
  table[row sep=crcr]{%
52	0.230709987377175\\
69	0.118928848922451\\
119	0.0586197026996498\\
299	0.0293046440695797\\
1011	0.0146561173010173\\
3907	0.00733044659143545\\
15715	0.00366602869781119\\
63651	0.00183324262107121\\
};
\addlegendentry{p=3, k=2, $\mu$=1}

\addplot [color=mycolor5, line width=2.0pt, mark size=3.5pt, mark=square, mark options={solid, mycolor5}]
  table[row sep=crcr]{%
85	0.161030701651573\\
106	0.0328862106645502\\
164	0.00652394480608187\\
360	0.000558446724063799\\
1104	1.57037113179825e-05\\
4064	9.18157621043747e-08\\
16000	6.54097220831318e-10\\
64192	3.5473846082823e-12\\
};
\addlegendentry{p=4, k=3, $\mu$=0.1125}

\addplot [color=mycolor6, dotted, line width=2.0pt, mark size=3.3pt, mark=triangle, mark options={solid, mycolor6}]
  table[row sep=crcr]{%
85	0.161030701651573\\
106	0.0788466596309032\\
164	0.039329538106486\\
360	0.0196660965616751\\
1104	0.00983770795966166\\
4064	0.00492035565685001\\
16000	0.00246059550390854\\
64192	0.00123040748667513\\
};
\addlegendentry{p=4, k=3, $\mu$=1}

\logLogSlopeTriangle{0.84}{0.14}{0.745}{0.5}{gray, line width = 2pt}{1};
\logLogSlopeTriangle{0.845}{0.1}{0.47}{2}{gray, line width = 2pt}{4};
\logLogSlopeTriangleUpsideDown{0.78}{0.1}{0.47}{3}{gray, line width = 2pt}{6};
\logLogSlopeTriangle{0.68}{0.1}{0.205}{4}{gray, line width = 2pt}{8};

\end{axis}
\end{tikzpicture}%
}
		
		\caption{Approximation errors for the eigenfunction $u=u_{\nu_1,1}$ and its corresponding eigenvalue $\lambda=\lambda_{\nu_1,1}$ using gradually and uniformly refined hierarchical NURBS of degree $\bs p =(p,p)$, $p\in \{2,3,4\}$, and regularity $\bs k = (p-1,p-1)$. We set $\mu = 0.9 \frac{\nu_1}{p} = \frac{0.9}{2p}$ for graded meshes and $\mu = 1$ corresponds to uniform refinement.}
		\label{fig: approximation errors for second eigenfunction hierarchical}
	\end{figure}
	
	\begin{figure}[t]
		\resizebox{0.325\linewidth}{!}{% This file was created by matlab2tikz.
%
%The latest updates can be retrieved from
%  http://www.mathworks.com/matlabcentral/fileexchange/22022-matlab2tikz-matlab2tikz
%where you can also make suggestions and rate matlab2tikz.
%
\definecolor{mycolor1}{rgb}{0.00000,0.44700,0.74100}%
\definecolor{mycolor2}{rgb}{0.85000,0.32500,0.09800}%
\definecolor{mycolor3}{rgb}{0.92900,0.69400,0.12500}%
\definecolor{mycolor4}{rgb}{0.49400,0.18400,0.55600}%
\begin{tikzpicture}

\begin{axis}[%
width=4.568in,
height=3.603in,
at={(0.766in,0.486in)},
scale only axis,
xmode=log,
xmin=85,
xmax=69300,
xminorticks=true,
xlabel={ndof},
ymode=log,
ymin=1.49881653044161e-06,
ymax=0.285130980933794,
yminorticks=true,
ylabel={$||u-u_h||_{H^1_h(\Omega)}$},
label style={font=\Large\color{white!15!black}},
tick label style={font=\Large},
axis background/.style={fill=white},
legend style={at={(0.03,0.03)}, anchor=south west, legend cell align=left, align=left, draw=white!15!black, font=\large}
]
\addplot [color=mycolor1, line width=2.0pt, mark size=3.5pt, mark=square, mark options={solid, mycolor1}]
  table[row sep=crcr]{%
85	0.245261895704637\\
106	0.218692013038747\\
164	0.0713589159201717\\
360	0.0198924119675585\\
1104	0.00335118724551938\\
4064	0.000252660956848008\\
16000	2.12837883820422e-05\\
64192	1.54532894553275e-06\\
};
\addlegendentry{p=4, k=3, $\mu$=0.1125, hierarchical}

\addplot [color=mycolor2, dotted, line width=2.0pt, mark size=3.3pt, mark=triangle, mark options={solid, mycolor2}]
  table[row sep=crcr]{%
85	0.245261895704637\\
106	0.193871062936146\\
164	0.139167253373216\\
360	0.098641196706933\\
1104	0.0697693805278115\\
4064	0.0493315582468908\\
16000	0.0348807049560481\\
64192	0.0246634803276778\\
};
\addlegendentry{p=4, k=3, $\mu$=1, hierarchical}

\addplot [color=mycolor3, line width=2.0pt, mark size=3.5pt, mark=square, mark options={solid, mycolor3}]
  table[row sep=crcr]{%
85	0.285130980933794\\
126	0.197192902111528\\
232	0.0684174963842545\\
540	0.0184593947159994\\
1540	0.00309010425777908\\
5076	0.000233681290113897\\
18292	2.04180864991597e-05\\
69300	1.49881653044161e-06\\
};
\addlegendentry{p=4, k=3, $\mu$=0.1125, tensor}

\addplot [color=mycolor4, dotted, line width=2.0pt, mark size=3.3pt, mark=triangle, mark options={solid, mycolor4}]
  table[row sep=crcr]{%
85	0.285130980933794\\
126	0.215215149918408\\
232	0.152336446494169\\
540	0.107703365407248\\
1540	0.0761198911791054\\
5076	0.0538043220381938\\
18292	0.0380373770185348\\
69300	0.0268934998324739\\
};
\addlegendentry{p=4, k=3, $\mu$=1, tensor}

\logLogSlopeTriangle{0.58}{0.2}{0.8}{0.25}{gray, line width = 2pt}{0.5};
\logLogSlopeTriangle{0.4}{0.14}{0.42}{2}{gray, line width = 2pt}{4};

\end{axis}
\end{tikzpicture}%
}
		\hspace*{1mm}
		\resizebox{0.325\linewidth}{!}{% This file was created by matlab2tikz.
%
%The latest updates can be retrieved from
%  http://www.mathworks.com/matlabcentral/fileexchange/22022-matlab2tikz-matlab2tikz
%where you can also make suggestions and rate matlab2tikz.
%
\definecolor{mycolor1}{rgb}{0.00000,0.44700,0.74100}%
\definecolor{mycolor2}{rgb}{0.85000,0.32500,0.09800}%
\definecolor{mycolor3}{rgb}{0.92900,0.69400,0.12500}%
\definecolor{mycolor4}{rgb}{0.49400,0.18400,0.55600}%
\begin{tikzpicture}

\begin{axis}[%
width=4.568in,
height=3.603in,
at={(0.766in,0.486in)},
scale only axis,
xmode=log,
xmin=85,
xmax=69300,
xminorticks=true,
xlabel={ndof},
ymode=log,
ymin=1.49210312905765e-08,
ymax=0.0218245255776148,
yminorticks=true,
ylabel={$||u-u_h||_{L^2_h(\Omega)}$},
label style={font=\Large\color{white!15!black}},
tick label style={font=\Large},
axis background/.style={fill=white},
legend style={at={(0.03,0.03)}, anchor=south west, legend cell align=left, align=left, draw=white!15!black, font =\large}
]
\addplot [color=mycolor1, line width=2.0pt, mark size=3.5pt, mark=square, mark options={solid, mycolor1}]
  table[row sep=crcr]{%
85	0.0142911219244662\\
106	0.0111783727259739\\
164	0.00486684992232974\\
360	0.00082230239212875\\
1104	0.000185349669865249\\
4064	6.2507151147787e-06\\
16000	3.76600260068177e-07\\
64192	1.435996037364e-08\\
};
\addlegendentry{p=4, k=3, $\mu$=0.1125, hierarchical}

\addplot [color=mycolor2, dotted, line width=2.0pt, mark size=3.3pt, mark=triangle, mark options={solid, mycolor2}]
  table[row sep=crcr]{%
85	0.0142911219244662\\
106	0.00734986680720939\\
164	0.0033832802509867\\
360	0.00162446726623942\\
1104	0.000811767766895346\\
4064	0.000400061392250434\\
16000	0.00019900694052975\\
64192	9.92567713314827e-05\\
};
\addlegendentry{p=4, k=3, $\mu$=1, hierarchical}

\addplot [color=mycolor3, line width=2.0pt, mark size=3.5pt, mark=square, mark options={solid, mycolor3}]
  table[row sep=crcr]{%
85	0.0218245255776148\\
126	0.0166805972139888\\
232	0.00636391724458617\\
540	0.00086718477449919\\
1540	0.000190757296784033\\
5076	6.47598826435504e-06\\
18292	3.90997195279284e-07\\
69300	1.49210312905765e-08\\
};
\addlegendentry{p=4, k=3, $\mu$=0.1125, tensor}

\addplot [color=mycolor4, dotted, line width=2.0pt, mark size=3.3pt, mark=triangle, mark options={solid, mycolor4}]
  table[row sep=crcr]{%
85	0.0218245255776148\\
126	0.0106005765107467\\
232	0.00503350639672426\\
540	0.00246910262315302\\
1540	0.00123304497087492\\
5076	0.000613714461013339\\
18292	0.000306280721830317\\
69300	0.00015299766670742\\
};
\addlegendentry{p=4, k=3, $\mu$=1, tensor}

\logLogSlopeTriangle{0.75}{0.2}{0.6}{0.5}{gray, line width = 2pt}{1};
\logLogSlopeTriangle{0.4}{0.14}{0.43}{2.5}{gray, line width = 2pt}{5};

\end{axis}
\end{tikzpicture}%
}
		\hspace*{1mm}
		\resizebox{0.325\linewidth}{!}{% This file was created by matlab2tikz.
%
%The latest updates can be retrieved from
%  http://www.mathworks.com/matlabcentral/fileexchange/22022-matlab2tikz-matlab2tikz
%where you can also make suggestions and rate matlab2tikz.
%
\definecolor{mycolor1}{rgb}{0.00000,0.44700,0.74100}%
\definecolor{mycolor2}{rgb}{0.85000,0.32500,0.09800}%
\definecolor{mycolor3}{rgb}{0.92900,0.69400,0.12500}%
\definecolor{mycolor4}{rgb}{0.49400,0.18400,0.55600}%
\begin{tikzpicture}

\begin{axis}[%
width=4.568in,
height=3.603in,
at={(0.766in,0.486in)},
scale only axis,
xmode=log,
xmin=85,
xmax=69300,
xminorticks=true,
xlabel={ndof},
ymode=log,
ymin=3.36264349698467e-12,
ymax=0.233719550518147,
yminorticks=true,
ylabel={$|\lambda-\lambda_h|$},
label style={font=\Large\color{white!15!black}},
tick label style={font=\Large},
axis background/.style={fill=white},
legend style={at={(0.03,0.03)}, anchor=south west, legend cell align=left, align=left, draw=white!15!black}, font = \large]
\addplot [color=mycolor1, line width=2.0pt, mark size=3.5pt, mark=square, mark options={solid, mycolor1}]
  table[row sep=crcr]{%
85	0.161030701651573\\
106	0.0328862106645502\\
164	0.00652394480608187\\
360	0.000558446724063799\\
1104	1.57037113179825e-05\\
4064	9.18157621043747e-08\\
16000	6.54097220831318e-10\\
64192	3.5473846082823e-12\\
};
\addlegendentry{p=4, k=3, $\mu$=0.1125, hierarchical}

\addplot [color=mycolor2, dotted, line width=2.0pt, mark size=3.3pt, mark=triangle, mark options={solid, mycolor2}]
  table[row sep=crcr]{%
85	0.161030701651573\\
106	0.0788466596309032\\
164	0.039329538106486\\
360	0.0196660965616751\\
1104	0.00983770795966166\\
4064	0.00492035565685001\\
16000	0.00246059550390854\\
64192	0.00123040748667513\\
};
\addlegendentry{p=4, k=3, $\mu$=1, hierarchical}

\addplot [color=mycolor3, line width=2.0pt, mark size=3.5pt, mark=square, mark options={solid, mycolor3}]
  table[row sep=crcr]{%
85	0.233719550518147\\
126	0.00964208886796136\\
232	0.0034025182553421\\
540	0.000666690004381465\\
1540	1.67618577862072e-05\\
5076	8.63119886673758e-08\\
18292	6.10988593052753e-10\\
69300	3.36264349698467e-12\\
};
\addlegendentry{p=4, k=3, $\mu$=0.1125, tensor}

\addplot [color=mycolor4, dotted, line width=2.0pt, mark size=3.3pt, mark=triangle, mark options={solid, mycolor4}]
  table[row sep=crcr]{%
85	0.233719550518147\\
126	0.120388686835186\\
232	0.0605152510429487\\
540	0.0303257672081472\\
1540	0.0151803114069047\\
5076	0.00759452995609067\\
18292	0.00379836031525116\\
69300	0.00189945420095938\\
};
\addlegendentry{p=4, k=3, $\mu$=1, tensor}

\logLogSlopeTriangle{0.59}{0.2}{0.79}{0.5}{gray, line width = 2pt}{1};
\logLogSlopeTriangle{0.4}{0.14}{0.41}{4}{gray, line width = 2pt}{8};

\end{axis}
\end{tikzpicture}%
}
		\caption{Approximation errors for the eigenfunction $u=u_{\nu_1,1}$ and its corresponding eigenvalue $\lambda=\lambda_{\nu_1,1}$ using NURBS of degree $\bs p=(4,4)$ and regularity $\bs k= (3,3)$ on tensor product and hierarchical meshes. We set $\mu = 0.9 \frac{\nu_1}{4} = 0.1125$ for graded meshes and $\mu = 1$ corresponds to uniform refinement.}
		\label{fig: tensor vs hierachical mesh second eigenfunction}
	\end{figure}
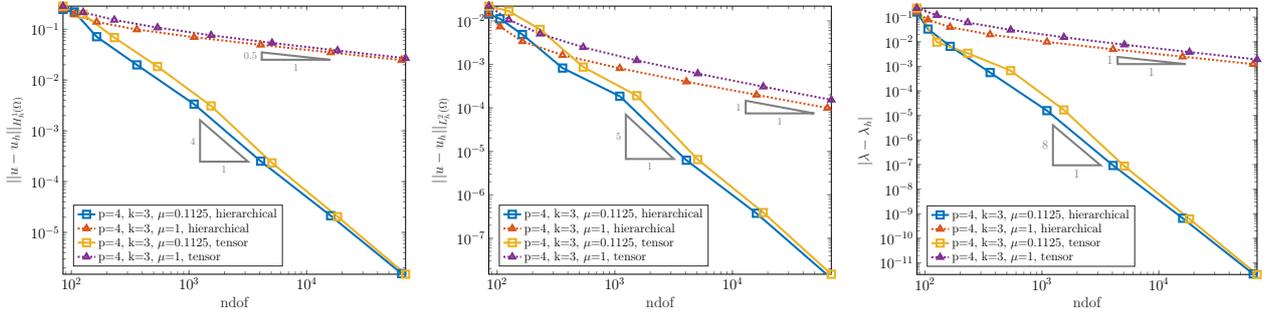
\end{example}	

\begin{example}[Eigenfunction of Type (B)]
	Finally, we repeat the experiment from Example \ref{ex: approximation of smooth eigenfunction}, but now using the hierarchical refinement scheme. We approximate the smooth eigenfunction $u_{\nu_2,1}$ and its corresponding eigenvalue $\lambda_{\nu_2,1}$ of the unit disk with crack on both uniform and graded hierarchical meshes. The approximation errors for NURBS of polynomial degree $\bs p =(p,p)$, where $p\in\{2,3,4\}$, and regularity $\bs k= (p-1,p-1)$ are illustrated in Figure \ref{fig: approximation errors for third eigenfunction hierarchical} and clearly differ from the ones obtained with tensor product refinement in Figure \ref{fig: approximation errors for third eigenfunction variational crime}. Independent from the NURBS degree, the convergence rates for the $H^1_h(\Omega)$- and $L^2_h(\Omega)$-eigenfunction error and the absolute eigenvalue error generated by uniform hierarchical refinement are $1$, $2$ and $2$, respectively, and thus not optimal. In contrast, graded hierarchical meshes do produce optimal convergence rates of $p$, $p+1$ and $2p$ for the $H^1_h(\Omega)$- and $L^2_h(\Omega)$-eigenfunction error and the absolute eigenvalue error, respectively. The different precisions obtained with hierarchical and tensor product meshes are demonstrated explicitly in Figure \ref{fig: tensor vs hierachical mesh third eigenfunction}, where NURBS of degree $\bs p = (4,4)$ and regularity $\bs k = (3,3)$ are refined both gradually and uniformly with grading parameters $\mu = 0.9 \cdot \frac{\nu_1}{4} = 0.1125$ and $\mu =1$, respectively. These results indicate that mesh grading restores optimal convergence of the hierarchically refined solutions. However, in this case, the grading is not required by the regularity of the approximated eigenfunction but by the hierarchical structure of the mesh, which is consistent with the theoretical findings in \cite{Takacs2023preprint}. This is because the size of the hierarchical mesh elements close to the singularity, as illustrated in Figure \ref{fig: elements close to center hierarchical}, does not decrease fast enough during uniform refinement. Yet, with sufficiently strong mesh grading towards the singularity, this issue is addressed.
	
	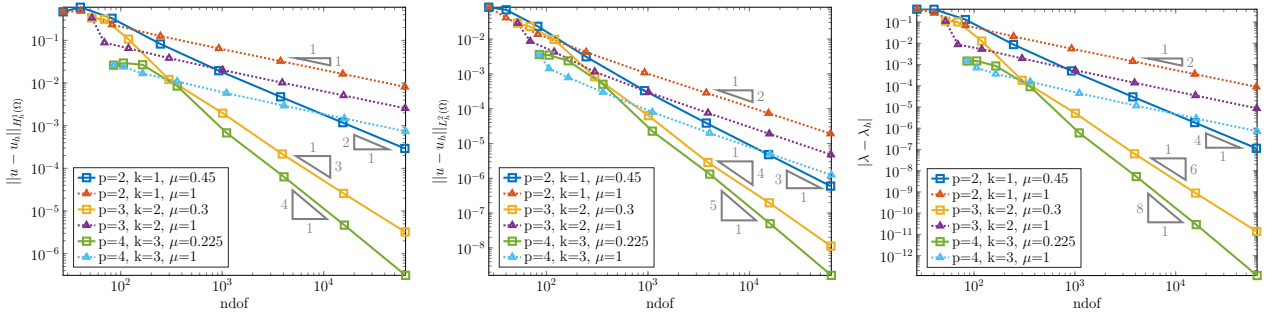
\begin{figure}[t]
		\resizebox{0.325\linewidth}{!}{% This file was created by matlab2tikz.
%
%The latest updates can be retrieved from
%  http://www.mathworks.com/matlabcentral/fileexchange/22022-matlab2tikz-matlab2tikz
%where you can also make suggestions and rate matlab2tikz.
%
\definecolor{mycolor1}{rgb}{0.00000,0.44700,0.74100}%
\definecolor{mycolor2}{rgb}{0.85000,0.32500,0.09800}%
\definecolor{mycolor3}{rgb}{0.92900,0.69400,0.12500}%
\definecolor{mycolor4}{rgb}{0.49400,0.18400,0.55600}%
\definecolor{mycolor5}{rgb}{0.46600,0.67400,0.18800}%
\definecolor{mycolor6}{rgb}{0.30100,0.74500,0.93300}%
\begin{tikzpicture}
\tikzstyle{every node}=[font=\Large]
\begin{axis}[%
width=4.568in,
height=3.603in,
at={(0.766in,0.486in)},
scale only axis,
xmode=log,
xmin=27,
xmax=64192,
xminorticks=true,
xlabel={ndof},
ymode=log,
ymin=3.10821127562824e-07,
ymax=0.599860404174659,
yminorticks=true,
ylabel={$||u-u_h||_{H^1_h(\Omega)}$},
axis background/.style={fill=white},
legend style={at={(0.03,0.03)}, anchor=south west, legend cell align=left, align=left, draw=white!15!black}
]
\addplot [color=mycolor1, line width=2.0pt, mark size=3.5pt, mark=square, mark options={solid, mycolor1}]
  table[row sep=crcr]{%
27	0.468273323255172\\
40	0.599860404174659\\
82	0.330380278409346\\
246	0.082362847532566\\
926	0.0195015654562329\\
3758	0.00476903899545639\\
15438	0.00117937632192133\\
63118	0.000292892375950811\\
};
\addlegendentry{p=2, k=1, $\mu$=0.45}

\addplot [color=mycolor2, dotted, line width=2.0pt, mark size=3.3pt, mark=triangle, mark options={solid, mycolor2}]
  table[row sep=crcr]{%
27	0.468273323255172\\
40	0.487644706128008\\
82	0.234521813096592\\
246	0.126475937928948\\
926	0.0648560322661512\\
3758	0.0326411356242137\\
15438	0.0163474848327489\\
63118	0.00817711115199149\\
};
\addlegendentry{p=2, k=1, $\mu$=1}

\addplot [color=mycolor3, line width=2.0pt, mark size=3.5pt, mark=square, mark options={solid, mycolor3}]
  table[row sep=crcr]{%
52	0.336248523585391\\
69	0.313079840292247\\
119	0.106993083278129\\
299	0.0121602074821732\\
1011	0.00196675237154749\\
3907	0.000217381000118804\\
15715	2.58651004809897e-05\\
63651	3.22799065856916e-06\\
};
\addlegendentry{p=3, k=2, $\mu$=0.3}

\addplot [color=mycolor4, dotted, line width=2.0pt, mark size=3.3pt, mark=triangle, mark options={solid, mycolor4}]
  table[row sep=crcr]{%
52	0.336248523585391\\
69	0.0886024869509845\\
119	0.0653558488730191\\
299	0.0385287523994931\\
1011	0.0201712208416168\\
3907	0.0102048136821523\\
15715	0.00511747484576287\\
63651	0.00256062625329559\\
};
\addlegendentry{p=3, k=2, $\mu$=1}

\addplot [color=mycolor5, line width=2.0pt, mark size=3.5pt, mark=square, mark options={solid, mycolor5}]
  table[row sep=crcr]{%
85	0.0263123255558246\\
106	0.0294778800424176\\
164	0.0269110160766774\\
360	0.00851094398645413\\
1104	0.00067993014446556\\
4064	6.33541231299188e-05\\
16000	4.6738199864072e-06\\
64192	3.10821127562824e-07\\
};
\addlegendentry{p=4, k=3, $\mu$=0.225}

\addplot [color=mycolor6, dotted, line width=2.0pt, mark size=3.3pt, mark=triangle, mark options={solid, mycolor6}]
  table[row sep=crcr]{%
85	0.0263123255558246\\
106	0.0239853457493315\\
164	0.0168095828025716\\
360	0.010896409371414\\
1104	0.00580471829247579\\
4064	0.00294887121713901\\
16000	0.00148033739720268\\
64192	0.000740908722456078\\
};
\addlegendentry{p=4, k=3, $\mu$=1}

\logLogSlopeTriangleUpsideDown{0.78}{0.1}{0.81}{0.5}{gray, line width = 2pt}{1};
\logLogSlopeTriangle{0.85}{0.1}{0.47}{1}{gray, line width = 2pt}{2};
\logLogSlopeTriangleUpsideDown{0.78}{0.1}{0.445}{1.5}{gray, line width = 2pt}{3};
\logLogSlopeTriangle{0.67}{0.1}{0.21}{2}{gray, line width = 2pt}{4};

\end{axis}
\end{tikzpicture}%
}
		\hspace*{1mm}
		\resizebox{0.325\linewidth}{!}{% This file was created by matlab2tikz.
%
%The latest updates can be retrieved from
%  http://www.mathworks.com/matlabcentral/fileexchange/22022-matlab2tikz-matlab2tikz
%where you can also make suggestions and rate matlab2tikz.
%
\definecolor{mycolor1}{rgb}{0.00000,0.44700,0.74100}%
\definecolor{mycolor2}{rgb}{0.85000,0.32500,0.09800}%
\definecolor{mycolor3}{rgb}{0.92900,0.69400,0.12500}%
\definecolor{mycolor4}{rgb}{0.49400,0.18400,0.55600}%
\definecolor{mycolor5}{rgb}{0.46600,0.67400,0.18800}%
\definecolor{mycolor6}{rgb}{0.30100,0.74500,0.93300}%
\begin{tikzpicture}
\tikzstyle{every node}=[font=\Large]
\begin{axis}[%
width=4.568in,
height=3.603in,
at={(0.766in,0.486in)},
scale only axis,
xmode=log,
xmin=27,
xmax=64192,
xminorticks=true,
xlabel={ndof},
ymode=log,
ymin=1.62178360891727e-09,
ymax=0.0823165156232398,
yminorticks=true,
ylabel={$||u-u_h||_{L^2_h(\Omega)}$},
axis background/.style={fill=white},
legend style={at={(0.03,0.03)}, anchor=south west, legend cell align=left, align=left, draw=white!15!black}
]
\addplot [color=mycolor1, line width=2.0pt, mark size=3.5pt, mark=square, mark options={solid, mycolor1}]
  table[row sep=crcr]{%
27	0.0823165156232398\\
40	0.0694580355622535\\
82	0.0232924089568051\\
246	0.00309995799815063\\
926	0.000329063495469562\\
3758	3.85579023263391e-05\\
15438	4.77024528311006e-06\\
63118	5.95011619385012e-07\\
};
\addlegendentry{p=2, k=1, $\mu$=0.45}

\addplot [color=mycolor2, dotted, line width=2.0pt, mark size=3.3pt, mark=triangle, mark options={solid, mycolor2}]
  table[row sep=crcr]{%
27	0.0823165156232398\\
40	0.041100145638054\\
82	0.013762385049116\\
246	0.00420854274843953\\
926	0.00108497827701383\\
3758	0.000285172006724673\\
15438	7.40841651386948e-05\\
63118	1.91497102781948e-05\\
};
\addlegendentry{p=2, k=1, $\mu$=1}

\addplot [color=mycolor3, line width=2.0pt, mark size=3.5pt, mark=square, mark options={solid, mycolor3}]
  table[row sep=crcr]{%
52	0.0288077296478047\\
69	0.0227121902862282\\
119	0.00975262560517137\\
299	0.000826006311078931\\
1011	6.31195363133068e-05\\
3907	2.84574729153551e-06\\
15715	1.96693128474105e-07\\
63651	1.12176520164936e-08\\
};
\addlegendentry{p=3, k=2, $\mu$=0.3}

\addplot [color=mycolor4, dotted, line width=2.0pt, mark size=3.3pt, mark=triangle, mark options={solid, mycolor4}]
  table[row sep=crcr]{%
52	0.0288077296478047\\
69	0.00872618973610759\\
119	0.00428609031335559\\
299	0.0011490595589441\\
1011	0.000295934857322451\\
3907	7.53116828048415e-05\\
15715	1.89856274451204e-05\\
63651	4.77362135176686e-06\\
};
\addlegendentry{p=3, k=2, $\mu$=1}

\addplot [color=mycolor5, line width=2.0pt, mark size=3.5pt, mark=square, mark options={solid, mycolor5}]
  table[row sep=crcr]{%
85	0.00359414515353777\\
106	0.0034532443376467\\
164	0.00237218927695881\\
360	0.000502189708721946\\
1104	2.26214135997688e-05\\
4064	1.31266774626689e-06\\
16000	4.94982858739754e-08\\
64192	1.62178360891727e-09\\
};
\addlegendentry{p=4, k=3, $\mu$=0.225}

\addplot [color=mycolor6, dotted, line width=2.0pt, mark size=3.3pt, mark=triangle, mark options={solid, mycolor6}]
  table[row sep=crcr]{%
85	0.00359414515353777\\
106	0.00141856640143475\\
164	0.000783631291251699\\
360	0.000298244917939538\\
1104	7.95690619102752e-05\\
4064	1.98797486240661e-05\\
16000	4.99590520129593e-06\\
64192	1.2510729846586e-06\\
};
\addlegendentry{p=4, k=3, $\mu$=1}

\logLogSlopeTriangleUpsideDown{0.77}{0.1}{0.69}{1}{gray, line width = 2pt}{2};
\logLogSlopeTriangle{0.87}{0.1}{0.325}{1.5}{gray, line width = 2pt}{3};
\logLogSlopeTriangleUpsideDown{0.77}{0.1}{0.425}{2}{gray, line width = 2pt}{4};
\logLogSlopeTriangle{0.68}{0.1}{0.205}{2.5}{gray, line width = 2pt}{5};

\end{axis}
\end{tikzpicture}%
}
		\hspace*{1mm}
		\resizebox{0.325\linewidth}{!}{% This file was created by matlab2tikz.
%
%The latest updates can be retrieved from
%  http://www.mathworks.com/matlabcentral/fileexchange/22022-matlab2tikz-matlab2tikz
%where you can also make suggestions and rate matlab2tikz.
%
\definecolor{mycolor1}{rgb}{0.00000,0.44700,0.74100}%
\definecolor{mycolor2}{rgb}{0.85000,0.32500,0.09800}%
\definecolor{mycolor3}{rgb}{0.92900,0.69400,0.12500}%
\definecolor{mycolor4}{rgb}{0.49400,0.18400,0.55600}%
\definecolor{mycolor5}{rgb}{0.46600,0.67400,0.18800}%
\definecolor{mycolor6}{rgb}{0.30100,0.74500,0.93300}%
\begin{tikzpicture}
\tikzstyle{every node}=[font=\Large]

\begin{axis}[%
width=4.568in,
height=3.603in,
at={(0.766in,0.486in)},
scale only axis,
xmode=log,
xmin=27,
xmax=64192,
xminorticks=true,
xlabel={ndof},
ymode=log,
ymin=1.15463194561016e-13,
ymax=0.404275209235804,
yminorticks=true,
ylabel={$|\lambda-\lambda_h|$},
axis background/.style={fill=white},
legend style={at={(0.03,0.03)}, anchor=south west, legend cell align=left, align=left, draw=white!15!black}
]
\addplot [color=mycolor1, line width=2.0pt, mark size=3.5pt, mark=square, mark options={solid, mycolor1}]
  table[row sep=crcr]{%
27	0.404275209235804\\
40	0.400981538212474\\
82	0.131652041879288\\
246	0.00874214272949203\\
926	0.000501752116099041\\
3758	3.0168010090037e-05\\
15438	1.84738112807281e-06\\
63118	1.13973927184929e-07\\
};
\addlegendentry{p=2, k=1, $\mu$=0.45}

\addplot [color=mycolor2, dotted, line width=2.0pt, mark size=3.3pt, mark=triangle, mark options={solid, mycolor2}]
  table[row sep=crcr]{%
27	0.404275209235804\\
40	0.266264303243268\\
82	0.0692477261508877\\
246	0.0210138065775851\\
926	0.00559056292718729\\
3758	0.00142101447062259\\
15438	0.000356766288705401\\
63118	8.92878929690255e-05\\
};
\addlegendentry{p=2, k=1, $\mu$=1}

\addplot [color=mycolor3, line width=2.0pt, mark size=3.5pt, mark=square, mark options={solid, mycolor3}]
  table[row sep=crcr]{%
52	0.110879788798819\\
69	0.101136434305769\\
119	0.0128208767300872\\
299	0.000182077341795761\\
1011	5.06236390407366e-06\\
3907	6.26429130790029e-08\\
15715	8.88110918140228e-10\\
63651	1.38449252062856e-11\\
};
\addlegendentry{p=3, k=2, $\mu$=0.3}

\addplot [color=mycolor4, dotted, line width=2.0pt, mark size=3.3pt, mark=triangle, mark options={solid, mycolor4}]
  table[row sep=crcr]{%
52	0.110879788798819\\
69	0.00878770481477531\\
119	0.00531028879345463\\
299	0.00194858874272796\\
1011	0.000539830705514888\\
3907	0.000138532225800958\\
15715	3.48612825380457e-05\\
63651	8.72967084930565e-06\\
};
\addlegendentry{p=3, k=2, $\mu$=1}

\addplot [color=mycolor5, line width=2.0pt, mark size=3.5pt, mark=square, mark options={solid, mycolor5}]
  table[row sep=crcr]{%
85	0.00146607272176524\\
106	0.00147802543548892\\
164	0.000858591735269698\\
360	9.07837522081678e-05\\
1104	6.02960509965556e-07\\
4064	5.29626120737703e-09\\
16000	2.88906676360057e-11\\
64192	1.15463194561016e-13\\
};
\addlegendentry{p=4, k=3, $\mu$=0.225}

\addplot [color=mycolor6, dotted, line width=2.0pt, mark size=3.3pt, mark=triangle, mark options={solid, mycolor6}]
  table[row sep=crcr]{%
85	0.00146607272176524\\
106	0.000709876535388787\\
164	0.000361928295005143\\
360	0.000156018920586476\\
1104	4.46803379166028e-05\\
4064	1.15572557461974e-05\\
16000	2.91409962471789e-06\\
64192	7.30083462130438e-07\\
};
\addlegendentry{p=4, k=3, $\mu$=1}

\logLogSlopeTriangleUpsideDown{0.78}{0.1}{0.815}{1}{gray, line width = 2pt}{2};
\logLogSlopeTriangle{0.85}{0.1}{0.48}{2}{gray, line width = 2pt}{4};
\logLogSlopeTriangleUpsideDown{0.79}{0.1}{0.44}{3}{gray, line width = 2pt}{6};
\logLogSlopeTriangle{0.68}{0.1}{0.2}{4}{gray, line width = 2pt}{8};

\end{axis}
\end{tikzpicture}%
}
		
		\caption{Approximation errors for the eigenfunction $u=u_{\nu_2,1}$ and its corresponding eigenvalue $\lambda=\lambda_{\nu_2,1}$ using gradually and uniformly refined hierarchical NURBS of degree $\bs p =(p,p)$, $p\in \{2,3,4\}$, and regularity $\bs k = (p-1,p-1)$. We set $\mu = 0.9 \frac{\nu_2}{p} = \frac{0.9}{p}$ for graded meshes and $\mu = 1$ corresponds to uniform refinement.}
		\label{fig: approximation errors for }
		\label{fig: approximation errors for third eigenfunction hierarchical}
	\end{figure}
	
	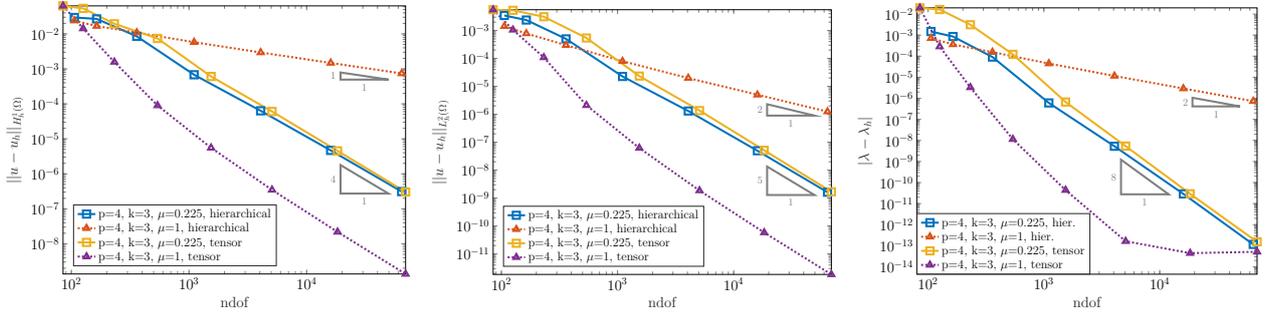
\begin{figure}[t]
		\resizebox{0.325\linewidth}{!}{% This file was created by matlab2tikz.
%
%The latest updates can be retrieved from
%  http://www.mathworks.com/matlabcentral/fileexchange/22022-matlab2tikz-matlab2tikz
%where you can also make suggestions and rate matlab2tikz.
%
\definecolor{mycolor1}{rgb}{0.00000,0.44700,0.74100}%
\definecolor{mycolor2}{rgb}{0.85000,0.32500,0.09800}%
\definecolor{mycolor3}{rgb}{0.92900,0.69400,0.12500}%
\definecolor{mycolor4}{rgb}{0.49400,0.18400,0.55600}%
\begin{tikzpicture}
\begin{axis}[%
width=4.568in,
height=3.603in,
at={(0.766in,0.486in)},
scale only axis,
xmode=log,
xmin=85,
xmax=69300,
xminorticks=true,
xlabel={ndof},
ymode=log,
ymin=1.36869504869903e-09,
ymax=0.0633823339081802,
yminorticks=true,
ylabel={$||u-u_h||_{H^1_h(\Omega)}$},
label style={font=\Large\color{white!15!black}},
tick label style={font=\Large},
axis background/.style={fill=white},
legend style={at={(0.03,0.03)}, anchor=south west, legend cell align=left, align=left, draw=white!15!black, font=\large}
]
\addplot [color=mycolor1, line width=2.0pt, mark size=3.5pt, mark=square, mark options={solid, mycolor1}]
  table[row sep=crcr]{%
106	0.0294778800424176\\
164	0.0269110160766774\\
360	0.00851094398645413\\
1104	0.00067993014446556\\
4064	6.33541231299188e-05\\
16000	4.6738199864072e-06\\
64192	3.10821127562824e-07\\
};
\addlegendentry{p=4, k=3, $\mu$=0.225, hierarchical}

\addplot [color=mycolor2, dotted, line width=2.0pt, mark size=3.3pt, mark=triangle, mark options={solid, mycolor2}]
  table[row sep=crcr]{%
106	0.0239853457493315\\
164	0.0168095828025716\\
360	0.010896409371414\\
1104	0.00580471829247579\\
4064	0.00294887121713901\\
16000	0.00148033739720268\\
64192	0.000740908722456078\\
};
\addlegendentry{p=4, k=3, $\mu$=1, hierarchical}

\addplot [color=mycolor3, line width=2.0pt, mark size=3.5pt, mark=square, mark options={solid, mycolor3}]
  table[row sep=crcr]{%
85	0.0633823339081802\\
126	0.0537437639049071\\
232	0.0195278652947162\\
540	0.00742491663452471\\
1540	0.000606562999893627\\
5076	6.05563669971165e-05\\
18292	4.52977325308729e-06\\
69300	3.02423322081481e-07\\
};
\addlegendentry{p=4, k=3, $\mu$=0.225, tensor}

\addplot [color=mycolor4, dotted, line width=2.0pt, mark size=3.3pt, mark=triangle, mark options={solid, mycolor4}]
  table[row sep=crcr]{%
85	0.0633823339081802\\
126	0.0142728705614051\\
232	0.00157750394781962\\
540	8.99425529648439e-05\\
1540	5.53625226893008e-06\\
5076	3.47308617810386e-07\\
18292	2.18216891122536e-08\\
69300	1.36869504869903e-09\\
};
\addlegendentry{p=4, k=3, $\mu$=1, tensor}

\logLogSlopeTriangle{0.81}{0.14}{0.725}{0.5}{gray, line width = 2pt}{1};
\logLogSlopeTriangle{0.81}{0.14}{0.3}{2}{gray, line width = 2pt}{4};

\end{axis}
\end{tikzpicture}%
}
		\hspace*{1mm}
		\resizebox{0.325\linewidth}{!}{% This file was created by matlab2tikz.
%
%The latest updates can be retrieved from
%  http://www.mathworks.com/matlabcentral/fileexchange/22022-matlab2tikz-matlab2tikz
%where you can also make suggestions and rate matlab2tikz.
%
\definecolor{mycolor1}{rgb}{0.00000,0.44700,0.74100}%
\definecolor{mycolor2}{rgb}{0.85000,0.32500,0.09800}%
\definecolor{mycolor3}{rgb}{0.92900,0.69400,0.12500}%
\definecolor{mycolor4}{rgb}{0.49400,0.18400,0.55600}%
\begin{tikzpicture}
\begin{axis}[%
width=4.568in,
height=3.603in,
at={(0.766in,0.486in)},
scale only axis,
xmode=log,
xmin=85,
xmax=69300,
xminorticks=true,
xlabel={ndof},
ymode=log,
ymin=1.8497194333741e-12,
ymax=0.00555264680203518,
yminorticks=true,
label style={font=\Large\color{white!15!black}},
tick label style={font=\Large},
ylabel={$||u-u_h||_{L^2_h(\Omega)}$},
axis background/.style={fill=white},
legend style={at={(0.03,0.03)}, anchor=south west, legend cell align=left, align=left, draw=white!15!black,font=\large}
]
\addplot [color=mycolor1, line width=2.0pt, mark size=3.5pt, mark=square, mark options={solid, mycolor1}]
  table[row sep=crcr]{%
106	0.0034532443376467\\
164	0.00237218927695881\\
360	0.000502189708721946\\
1104	2.26214135997688e-05\\
4064	1.31266774626689e-06\\
16000	4.94982858739754e-08\\
64192	1.62178360891727e-09\\
};
\addlegendentry{p=4, k=3, $\mu$=0.225, hierarchical}

\addplot [color=mycolor2, dotted, line width=2.0pt, mark size=3.3pt, mark=triangle, mark options={solid, mycolor2}]
  table[row sep=crcr]{%
106	0.00141856640143475\\
164	0.000783631291251699\\
360	0.000298244917939538\\
1104	7.95690619102752e-05\\
4064	1.98797486240661e-05\\
16000	4.99590520129593e-06\\
64192	1.2510729846586e-06\\
};
\addlegendentry{p=4, k=3, $\mu$=1, hierarchical}

\addplot [color=mycolor3, line width=2.0pt, mark size=3.5pt, mark=square, mark options={solid, mycolor3}]
  table[row sep=crcr]{%
85	0.00555264680203518\\
126	0.00528779477460469\\
232	0.00311185819807807\\
540	0.000539547018733666\\
1540	2.33542884488691e-05\\
5076	1.36289672675293e-06\\
18292	5.14210759118573e-08\\
69300	1.68794180593417e-09\\
};
\addlegendentry{p=4, k=3, $\mu$=0.225, tensor}

\addplot [color=mycolor4, dotted, line width=2.0pt, mark size=3.3pt, mark=triangle, mark options={solid, mycolor4}]
  table[row sep=crcr]{%
85	0.00555264680203518\\
126	0.00105234918596993\\
232	0.000108533149651327\\
540	2.09919052100588e-06\\
1540	6.21250300758878e-08\\
5076	1.86633239841443e-09\\
18292	5.84655691974127e-11\\
69300	1.8497194333741e-12\\
};
\addlegendentry{p=4, k=3, $\mu$=1, tensor}

\logLogSlopeTriangle{0.81}{0.14}{0.6}{1}{gray, line width = 2pt}{2};
\logLogSlopeTriangle{0.81}{0.14}{0.3}{2.5}{gray, line width = 2pt}{5};

\end{axis}
\end{tikzpicture}%
}
		\hspace*{1mm}
		\resizebox{0.325\linewidth}{!}{% This file was created by matlab2tikz.
%
%The latest updates can be retrieved from
%  http://www.mathworks.com/matlabcentral/fileexchange/22022-matlab2tikz-matlab2tikz
%where you can also make suggestions and rate matlab2tikz.
%
\definecolor{mycolor1}{rgb}{0.00000,0.44700,0.74100}%
\definecolor{mycolor2}{rgb}{0.85000,0.32500,0.09800}%
\definecolor{mycolor3}{rgb}{0.92900,0.69400,0.12500}%
\definecolor{mycolor4}{rgb}{0.49400,0.18400,0.55600}%
\begin{tikzpicture}

\begin{axis}[%
width=4.568in,
height=3.603in,
at={(0.766in,0.486in)},
scale only axis,
xmode=log,
xmin=80,
xmax=69300,
xminorticks=true,
xlabel={ndof},
ymode=log,
ymin=4.44089209850063e-15,
ymax=0.0197036887358539,
yminorticks=true,
ylabel={$|\lambda-\lambda_h|$},
label style={font=\Large\color{white!15!black}},
tick label style={font=\Large},
axis background/.style={fill=white},
legend style={at={(0,0)}, anchor=south west, legend cell align=left, align=left, draw=white!15!black,font = \large}
]
\addplot [color=mycolor1, line width=2.0pt, mark size=3.5pt, mark=square, mark options={solid, mycolor1}]
  table[row sep=crcr]{%
106	0.00147802543548892\\
164	0.000858591735269698\\
360	9.07837522081678e-05\\
1104	6.02960509965556e-07\\
4064	5.29626120737703e-09\\
16000	2.88906676360057e-11\\
64192	1.15463194561016e-13\\
};
\addlegendentry{p=4, k=3, $\mu$=0.225, hier.}

\addplot [color=mycolor2, dotted, line width=2.0pt, mark size=3.3pt, mark=triangle, mark options={solid, mycolor2}]
  table[row sep=crcr]{%
106	0.000709876535388787\\
164	0.000361928295005143\\
360	0.000156018920586476\\
1104	4.46803379166028e-05\\
4064	1.15572557461974e-05\\
16000	2.91409962471789e-06\\
64192	7.30083462130438e-07\\
};
\addlegendentry{p=4, k=3, $\mu$=1, hier.}

\addplot [color=mycolor3, line width=2.0pt, mark size=3.5pt, mark=square, mark options={solid, mycolor3}]
  table[row sep=crcr]{%
85	0.0197036887358539\\
126	0.0164096667762159\\
232	0.00311682058013218\\
540	0.000122334096065657\\
1540	6.6684811983464e-07\\
5076	5.37607469652812e-09\\
18292	2.91517920913975e-11\\
69300	1.52766688188422e-13\\
};
\addlegendentry{p=4, k=3, $\mu$=0.225, tensor}

\addplot [color=mycolor4, dotted, line width=2.0pt, mark size=3.3pt, mark=triangle, mark options={solid, mycolor4}]
  table[row sep=crcr]{%
85	0.0197036887358539\\
126	0.000287161251179668\\
232	3.27961671686694e-06\\
540	1.14351266233825e-08\\
1540	4.3648640257743e-11\\
5076	1.63424829224823e-13\\
18292	4.44089209850063e-14\\
69300	5.15143483426073e-14\\
};
\addlegendentry{p=4, k=3, $\mu$=1, tensor}

\logLogSlopeTriangle{0.81}{0.14}{0.63}{1}{gray, line width = 2pt}{2};
\logLogSlopeTriangle{0.6}{0.14}{0.3}{4}{gray, line width = 2pt}{8};

\end{axis}
\end{tikzpicture}%
}
		
		\caption{Approximation errors for the eigenfunction $u=u_{\nu_2,1}$ and its corresponding eigenvalue $\lambda=\lambda_{\nu_2,1}$ using NURBS of degree $\bs p=(4,4)$ and regularity $\bs k= (3,3)$ on tensor product and hierarchical meshes. We set $\mu = 0.9 \frac{\nu_2}{4} = 0.225$ for graded meshes and $\mu = 1$ corresponds to uniform refinement.}
		\label{fig: tensor vs hierachical mesh third eigenfunction}
	\end{figure}
	
\end{example}

In summary, the combination of hierarchical mesh structure and graded refinement is very promising for approximating the Laplace eigenvalues of circular sectors and their corresponding eigenfunctions of both high and low regularity. However, the implementation and mathematical analysis are more difficult. Our numerical results can be a starting point for further investigations, especially in context with the theoretical explanations in \cite{Takacs2023preprint}.

\section{Conclusion \& Outlook }
\label{section: Conclusion}
In this paper, we studied the Laplace eigenvalue problem on circular sectors both analytically and numerically. We verified crucial regularity properties of the exact eigenfunctions and developed an effective numerical method to address the occurring singularities. In particular, we introduced a single-patch graded mesh refinement algorithm for isogeometric analysis on circular sectors, enabling local refinement towards the conical point, where some of the eigenfunctions show singular behavior. Since the polar-like isogeometric parameterization of circular sectors is also singular at the conical point, the method comes with a variational crime, but the numerical results indicate its robustness. We demonstrated optimal convergence rates for the eigenvalues and eigenfunctions numerically. Furthermore, we illustrated that smooth splines have a better approximation constant than their $C^0$-continuous counterparts on both uniform and graded meshes, at least for the lower part of the Laplace spectrum. Hence, we were able to extend some of the excellent spectral approximation properties of smooth splines, which have been discovered in the literature mainly for domains of rectangular nature, to circular sectors. In addition, we demonstrated the power of smooth splines on graded meshes for an accurate approximation of the $100$ smallest eigenvalues. Lastly, we considered a hierarchical refinement approach to avoid anisotropic elements in the physical B\'ezier mesh. We showed the efficiency of graded hierarchical meshes to simulate eigenfunctions with and without singularities and compared them with the previously used tensor product meshes. Here, we observed two advantages of the hierarchical scheme: redundant degrees of freedom are omitted and the number of basis functions contributing to the variational crime is held constant.

This contribution serves as motivation for various directions of further research. In general, we are lacking theoretical error estimates for isogeometric analysis on singularly parameterized circular sectors, both for tensor product and hierarchical meshes. We aim to address this gap in future work and, specifically, prove optimal approximation of non-smooth functions using our graded refinement algorithm. In addition, the presented method is extendable to other eigenvalue and boundary value problems on arbitrary two-dimensional single-patch domains with corner singularities. Here, our results serve as a proof of concept since the local neighborhood of a singular point can always be described by a circular sector. More work remains to be done for an extension of the method to multi-patch IGA and three-dimensional bodies with corner and edge singularities. Besides, it is well-established in finite element analysis that the $hp$-method is effective for handling singularities. Therefore, another idea for future research is to explore the $p$- and $k$-versions of graded refinement or to consider graded multi-degree polar splines \cite{ToshniwalSpeleersHiemstraHughes2017}. Finally, our findings provide a basis for future work on spectral approximation properties of smooth splines on circular sectors and, more generally, of singularly parameterized domains. Particularly, the appearance of spectral branches and the approximation constant in the upper part of the discrete spectrum remain to be investigated.

%%% The Appendices part is started with the command \appendix;
%% appendix sections are then done as normal sections
 \appendix
 
\section{Explicit representation of the single-patch parameterization}
 \label{sec: appendix explicit representation}
 To gain a better understanding, we derive an explicit representation of the single-patch parameterization of the unit disk with crack that is used throughout this paper. A similar computation can be found in the literature \cite{JeongOhKangKim2013,OhKimJeong2014}. By evaluating the univariate knot vector $\Xi^0_{1} = \left\{0,0,1,1\right\}$, we obtain two linear B-spline basis functions,
 \begin{alignat*}{2}
 &\widehat{B}_{1,1}(\zeta_1) = (1- \zeta_1) \quad &&\text{ for all } \zeta_1\in[0,1] \, ,\\
 &\widehat{B}_{2,1}(\zeta_1) = \zeta_1 \quad &&\text{ for all } \zeta_1\in[0,1] \, .
 \end{alignat*}
 The knot vector $\Xi^0_2 = \left\{0,0, 0,\frac14,\frac14, \frac12, \frac12, \frac34, \frac34,1,1,1 \right\}$ defines nine piecewise quadratic B-spline functions,
 \allowdisplaybreaks
 \begin{align*}
 &\widehat{B}_{1,2}(\zeta_2) = \begin{cases}
 (1-4 \zeta_2)^2 \quad &\text{ if } \zeta_2 \in[0,\frac14], \\ 
 0 \quad &\text{ if } \zeta_2 \notin [0,\frac14],
 \end{cases}
 &&\widehat{B}_{2,2}(\zeta_2) = \begin{cases}
 8 \zeta_2 (1-4 \zeta_2) \quad &\text{ if } \zeta_2 \in[0,\frac14], \\ 
 0 \quad &\text{ if } \zeta_2 \notin [0,\frac14],
 \end{cases} \\
 &\widehat{B}_{3,2}(\zeta_2) = \begin{cases}
 (4 \zeta_2)^2 \quad &\text{ if } \zeta_2 \in[0,\frac14], \\ 
 (2-4\zeta_2)^2 \quad &\text{ if } \zeta_2 \in [\frac14, \frac12], \\
 0  &\text{ if } \zeta_2 \notin [0,\frac12],
 \end{cases}
 &&\widehat{B}_{4,2}(\zeta_2) = \begin{cases}
 2(4\zeta_2-1)(2-4\zeta_2) \quad &\text{ if } \zeta_2 \in[\frac14,\frac12], \\ 
 0 \quad &\text{ if } \zeta_2 \notin [\frac14, \frac12],\\
 \end{cases} \\
 &\widehat{B}_{5,2}(\zeta_2) = \begin{cases}
 (4 \zeta_2-1)^2 \quad &\text{ if } \zeta_2 \in[\frac14,\frac12], \\ 
 (3-4\zeta_2)^2 \quad &\text{ if } \zeta_2 \in [\frac12, \frac34], \\
 0  &\text{ if } \zeta_2 \notin [\frac14,\frac34],
 \end{cases}
 &&\widehat{B}_{6,2}(\zeta_2) = \begin{cases}
 2(4\zeta_2-2)(3-4\zeta_2) \quad &\text{ if } \zeta_2 \in[\frac12,\frac34], \\ 
 0 \quad &\text{ if } \zeta_2 \notin [\frac12,\frac34], \\
 \end{cases} \\
 &\widehat{B}_{7,2}(\zeta_2) = \begin{cases}
 (4 \zeta_2-2)^2 \quad &\text{ if } \zeta_2 \in[\frac12,\frac34], \\ 
 (4-4\zeta_2)^2 \quad &\text{ if } \zeta_2 \in [\frac34, 1], \\
 0  &\text{ if } \zeta_2 \notin [\frac12,1],
 \end{cases}
 &&\widehat{B}_{8,2}(\zeta_2) = \begin{cases}
 8(4\zeta_2-3)(1-\zeta_2) \quad &\text{ if } \zeta_2 \in[\frac34,1], \\ 
 0 \quad &\text{ if } \zeta_2 \notin [\frac34,1], \\
 \end{cases} \\
 &\widehat{B}_{9,2}(\zeta_2) = \begin{cases}
 (4\zeta_2-3)^2 \quad &\text{ if } \zeta_2 \in[\frac34,1], \\ 
 0 \quad &\text{ if } \zeta_2 \notin [\frac34,1].
 \end{cases}
 \end{align*}
 Note that the functions $\widehat{B}_{i_2,2}$, $i_2=1,2,\dots,9$, are smooth in $[0,1] \setminus \left \{\frac14, \frac24, \frac34 \right\}$. In the remaining three points, the functions are only $C^0$-continuous due to the repeated knot values, as discussed in Section \ref{subsec: univariate splines}. The tensor product B-spline basis functions are given by
 \begin{align*}
 \widehat{B}_{\bs i, (1,2)}(\bs \zeta) = \widehat{B}_{i_1,1}(\zeta_1) \, \widehat{B}_{i_2,2}(\zeta_2) , \quad  \bs i \in \bs I_0\, ,
 \end{align*}
 and the final NURBS basis functions are then defined by
 \begin{align*}
 \widehat{N}_{\bs i, (1,2)}(\bs \zeta) = \frac{w_{\bs i} \, \widehat{B}_{\bs i, (1,2)}(\bs \zeta)}{W(\bs \zeta)} , \quad  \bs i  \in \bs I_0 \, ,
 \end{align*}
 with $W$ from \eqref{eq: weight function} and $\bs I_0 = \{ \bs i = (i_1,i_2): 1 \leq i_1 \leq 2, 1\leq i_2 \leq 9\}$. By inserting the values $w_{\bs i}$ from Table \ref{tab: Control points and weights for unit disk with crack}, where $w_{1,i_2}=w_{2,i_2}$ for $i_2=1,\dots,9$, and exploiting the partition of unity property of B-spline basis functions, the weight function can be determined in more detail, 
 \begin{align*}
 W(\bs \zeta) &= \sum_{i_1=1}^{2} \widehat{B}_{i_1,1}(\zeta_1) \sum_{i_2=2}^9 \widehat{B}_{i_2,2}(\zeta_2) \, w_{(2,i_2)} \\
 & =   \left(\widehat{B}_{1,2}(\zeta_2) + \widehat{B}_{3,2}(\zeta_2) + \widehat{B}_{5,2}(\zeta_2) + \widehat{B}_{7,2}(\zeta_2) + \widehat{B}_{9,2}(\zeta_2)\right) \, + \\ 
 & \quad \  \left(\widehat{B}_{2,2}(\zeta_2) + \widehat{B}_{4,2}(\zeta_2) + \widehat{B}_{6,2}(\zeta_2) + \widehat{B}_{8,2}(\zeta_2)\right) \cdot \frac{1}{\sqrt{2}} \\
 & \equiv w(\zeta_2) \, ,
 \end{align*}
 which is a function depending only on $\zeta_2$. Moreover, we note from Table \ref{tab: Control points and weights for unit disk with crack} that $\bs c_{(1,i_2)} = (0,0)$ for $i_2=1, \dots, 9$ and use the definition of $\widehat{B}_{2,1}$ to express the resulting NURBS parameterization $\bs F \colon \widehat{\Omega} \to \Omega$ in a compact way,	
 \begin{align}
 \bs F(\bs \zeta) &= \sum_{\bs i \in \bs I} \widehat{N}_{\bs i, (1,2)}(\bs \zeta) \, \bs c_{\bs i} \nonumber \\ 
 &= \sum_{i_1=1}^2 \sum_{i_2=1}^9 \bs c_{(i_1,i_2)} \, \frac{w_{(i_1,i_2)} \, \widehat{B}_{i_1,1}(\zeta_1) \, \widehat{B}_{i_2,2}(\zeta_2)}{w(\zeta_2)} \nonumber \\
 &= \sum_{i_2=1}^9 \bs c_{(2,i_2)} \, \frac{w_{(2,i_2)} \, \widehat{B}_{2,1}(\zeta_1) \, \widehat{B}_{i_2,2}(\zeta_2)}{w(\zeta_2)}  \nonumber \\
 &= \zeta_1 \sum_{i_2=1}^9 \bs c_{(2,i_2)} \, \frac{w_{(2,i_2)} \, \widehat{B}_{i_2,2}(\zeta_2)}{w(\zeta_2)} \nonumber\\
 &=: \zeta_1 \left(g_1(\zeta_2), g_2(\zeta_2) \right) \, ,
 \label{eq: Tensor-product form of F}
 \end{align}
 with functions $g_1, g_2 \colon [0,1] \to \R$ depending only on $\zeta_2$. The latter can be further simplified as some of the control points $\bs c_{(2,i_2)}$, $i_2=1,2, \dots, 9$, have zero entries,
 \begin{align*}
 g_1(\zeta_2) &= \frac{1}{w(\zeta_2)} \left(\widehat{B}_{1,2} + \frac{\widehat{B}_{2,2}}{\sqrt{2}} - \frac{\widehat{B}_{4,2}}{\sqrt{2}} - \widehat{B}_{5,2} - \frac{\widehat{B}_{6,2}}{\sqrt{2}} + \frac{\widehat{B}_{8,2}}{\sqrt{2}} + \widehat{B}_{9,2}\right) (\zeta_2)\, , \\
 g_2(\zeta_2) &= \frac{1}{w(\zeta_2)} \left(\frac{\widehat{B}_{2,2}}{\sqrt{2}} + \widehat{B}_{3,2} + \frac{\widehat{B}_{4,2}}{\sqrt{2}} - \frac{\widehat{B}_{6,2}}{\sqrt{2}} - \widehat{B}_{7,2} - \frac{\widehat{B}_{8,2}}{\sqrt{2}} \right) (\zeta_2) \, .
 \end{align*}
 Moreover, it can be shown that the weight function is bounded away from zero \cite{JeongOhKangKim2013,OhKimJeong2014},
 \begin{align*}
 \frac{2+\sqrt{2}}{4} \leq W(\zeta_1, \zeta_2) \equiv w(\zeta_2) \leq 1 \, .
 \end{align*}

\section{Comparison with polar coordinates}
 \label{sec: appendix comparison}
 
 \begin{figure}
 	\begin{subfigure}{0.495 \linewidth}
 		\includegraphics[width=\linewidth, trim=2.2cm 0.7cm 2.2cm 1.4cm, clip]{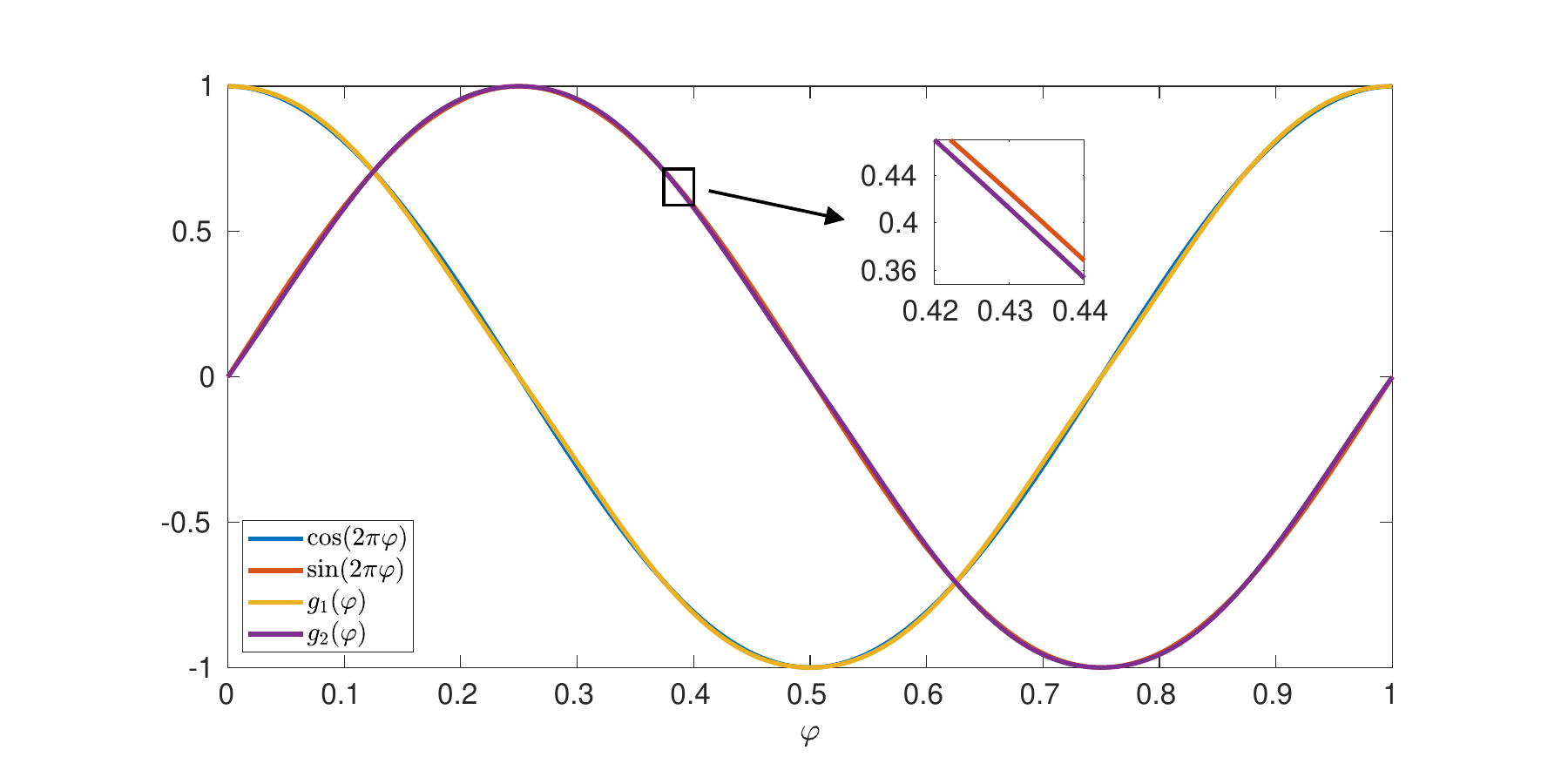}
 		\caption{}
 		\label{fig:parmeterizationpolar a}
 	\end{subfigure}
 	\hfill
 	\begin{subfigure}{0.495 \linewidth}
 		\includegraphics[width=\linewidth, trim=2.2cm 0.7cm 2.2cm 1.4cm, clip]{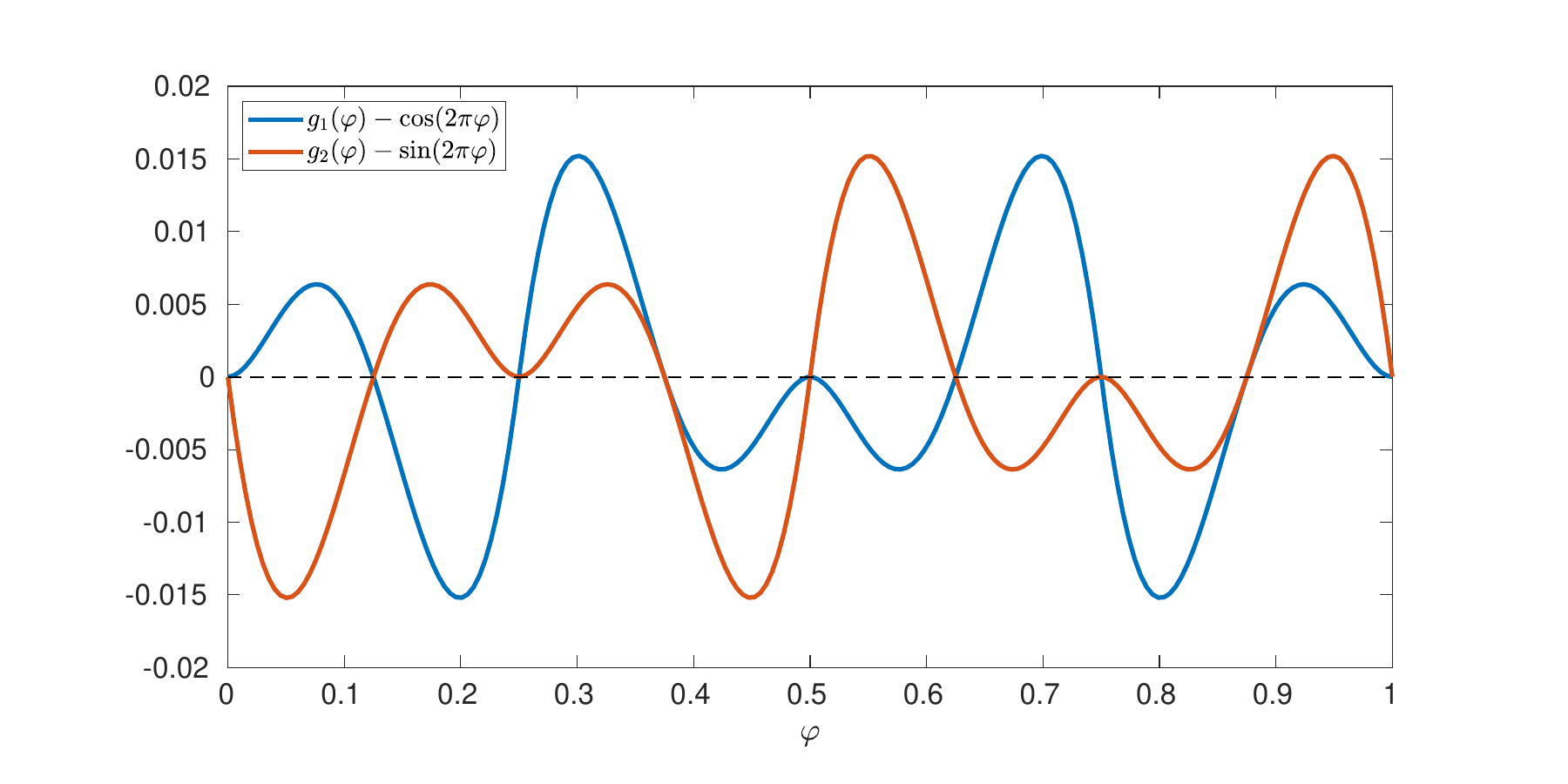}
 		\caption{}
 		\label{fig:parmeterizationpolar b}
 	\end{subfigure}
 	\vspace{-2mm}
 	\caption{(a): Isogeometric mappings $g_1$ and $g_2$ versus scaled polar angular mappings. (b): Difference between the corresponding functions.}
 	\label{fig:parmeterizationpolar}
 \end{figure}	 
 
 As the geometry mapping $\bs F$ resembles the classic transformation from polar to Cartesian coordinates, 
 \begin{align*}
 \bs{\bs{\mathring{F}}} \colon (0,1) \times (0,2 \pi) \to \Omega , \, (r, \varphi) \mapsto (r \cos \varphi, r \sin \varphi) \,,
 \end{align*}
 we point out some properties to compare the two parameterizations. First, let $(r, \varphi) \in (0,1) \times (0,2 \pi)$ and $(\zeta_1, \zeta_2) \in (0,1) \times (0,1)$ such that $\bs{\bs{\mathring{F}}}(r, \varphi) = \bs F(\zeta_1, \zeta_2)$. Using the representation \eqref{eq: Tensor-product form of F} of $ \bs F$ and the fact that $\left(g_1(\zeta_2)\right)^2 + \left(g_2(\zeta_2)\right)^2 =  1$ for all $\zeta_2 \in [0,1]$, see \cite{OhKimJeong2014}, it follows that
 \begin{align*}
 r = |\bs{\bs{\mathring{F}}}(r, \varphi)|= \abs{\bs F(\zeta_1, \zeta_2)} = \zeta_1 \left( \left(g_1(\zeta_2)\right)^{(2)} + \left(g_2(\zeta_2)\right)^{(2)}  \right) = \zeta_1 \, .
 \end{align*}
 Hence, the radial component $r$ in polar coordinates corresponds one-to-one with the isogeometric $\zeta_1$-component. However, the same does not apply for the angular components $\varphi$ and $\zeta_2$. Figure \ref{fig:parmeterizationpolar a} presents the isogeometric mappings $g_1$ and $g_2$ from representation \eqref{eq: Tensor-product form of F} in comparison with the scaled polar angular mappings $[0,1] \to \R, \ \varphi \mapsto \cos(2\pi \varphi)$ and $[0,1] \to \R, \ \varphi \mapsto \sin(2\pi \varphi)$. Figure \ref{fig:parmeterizationpolar b} depicts the difference of the corresponding functions. We observe that
 \begin{align*}
 \abs{g_1(\varphi) - \cos(2\pi\varphi)} \leq C \quad \text{ and } \quad \abs{g_2(\varphi) - \sin(2\pi\varphi)} \leq C \quad \text{ for all } \varphi \in [0,1] \, ,
 \end{align*}
 that is, the parameterizations coincide up to a maximal difference of $C \approx 0.015$. The mappings are even identical in a few points,
 \begin{align}
 \label{eq: equality polar and NURBS coordinates}
 g_1(\varphi) = \cos(2\pi\varphi) \quad \text{ and } \quad g_2(\varphi) = \sin(2\pi\varphi) \quad \text{for} \quad \varphi = \frac{i}{8}, i=0,\dots, 8 \, .
 \end{align}
 Furthermore, we can show another useful correlation. Let $K=\bs F(Q) \in \mathcal{M}$ be an element of the B\'ezier mesh. Then, $K$ can be represented in polar coordinates, that is, we can find $\mathring{Q} \subset (0,1) \times (0,2 \pi)$ such that 
 \begin{align*}
 K =  \bs F \left(Q \right) = \bs{\mathring{F}}\left(\mathring{Q}\right) \, .
 \end{align*}
 In more detail, if $Q = Q_{\bs j} = \left(\zeta_{1,j_1}, \zeta_{1,j_1+1}\right) \times \left(\zeta_{2,j_2}, \zeta_{2,j_2+1}\right)$ for some $\bs j \in \bs J$, there are polar angles $\varphi_{j_2}, \varphi_{j_2+1} \in [0,2\pi]$ such that $\mathring{Q} =  \left(\zeta_{1,j_1}, \zeta_{1,j_1+1}\right) \times \left(\varphi_{j_2}, \varphi_{j_2+1}\right)$. This can be proven using the bijectivity of the parameterizations away from the singular edge $\{(0,\zeta_2) : \zeta_2 \in [0,1]\}$.

\section*{Acknowledgments}
The authors thank Thomas Takacs, Alessandro Reali and Volker Kempf for motivating and insightful scientific discussions which helped to place the results of this paper within the existing literature context.

%% \section{}
%% \label{}

%% If you have bibdatabase file and want bibtex to generate the
%% bibitems, please use
%%
%%  \bibliographystyle{elsarticle-num} 
%%  \bibliography{<your bibdatabase>}

\biboptions{sort&compress}
\bibliographystyle{elsarticle-num-names} 
\bibliography{Bibliography.bib}

%% else use the following coding to input the bibitems directly in the
%% TeX file.

%\begin{thebibliography}{00}
%	
%
%	
%
%%\bibliographystyle{elsarticle-num} 
%%\bibliography{}
%
%%
%\end{thebibliography}
\end{document}